\documentclass[10pt,oneside]{amsart}
\usepackage[T1]{fontenc}
\usepackage{amsthm,amsmath}
\usepackage{amsbsy}
\usepackage{amstext}
\usepackage{amssymb}
\usepackage{esint}

\makeatletter
\numberwithin{equation}{section}
\numberwithin{figure}{section}

\usepackage{mathtools}

\newtheorem {theo} {Theorem} [section]
\newtheorem {prop} [theo] {Proposition}
\newtheorem {cory} [theo] {Corollary}
\newtheorem {lem} [theo] {Lemma}
\newtheorem {defi} [theo] {Definition}

\newtheorem*{theo1}{Theorem~A}
\newtheorem*{theo2}{Theorem~B}
\theoremstyle{definition}
\newtheorem{example}{Example}
\newtheorem {obs} [theo] {Remark}

\def\sideremark#1{\ifvmode\leavevmode\fi\vadjust{\vbox to0pt{\vss 
    \hbox to 0pt{\hskip\hsize\hskip1em           
 \vbox{\hsize2cm\tiny\raggedright\pretolerance10000
 \noindent #1\hfill}\hss}\vbox to8pt{\vfil}\vss}}}%


\usepackage[all]{xy}        
\newcommand{\xyL}[1]{%
	\xydef@\xymatrixrowsep@{#1}
} 
\newcommand{\xyC}[1]{%
	\xydef@\xymatrixcolsep@{#1}
} 
\makeatother

\begin{document}

\title[Realization of analytic moduli for parabolic Dulac germs]{Realization of analytic moduli\\ for parabolic Dulac germs}
\author{P. Marde\v si\'c$^{1}$, M. Resman$^{2}$}

\begin{abstract} In a previous paper \cite{prvi} we have determined analytic invariants, that is, moduli of analytic classification, for parabolic generalized Dulac germs. This class contains parabolic \emph{Dulac} (almost regular) germs, that appear as first return maps of hyperbolic polycycles. Here we solve the problem of realization of these moduli.
\end{abstract}

\maketitle

\noindent \emph{Acknowledgement}. This research was supported by the Croatian Unity Through Knowledge Fund (UKF) \emph{My first collaboration grant} project no. 16, by the Croatian Science Foundation (HRZZ) grants UIP-2017-05-1020 and PZS-2019-02-3055 from \emph{Research Cooperability} funded by the European Social Fund, and by the EIPHI Graduate School (contract ANR-17-EURE-0002). The UKF grant fully supported the $6$-month stay of $^2$ at Universit\' e de Bourgogne in 2018.

\section{\label{sec:introduction}Introduction and main definitions}








\emph{Dulac germs}, called \emph{almost regular} germs in \cite{ilyalim}, appear as first return maps on transversals to hyperbolic polycycles in planar analytic vector fields, see e.g. \cite{ilyalim} or \cite{roussarie}. From the viewpoint of cyclicity of planar vector fields, the most interesting case is the case of Dulac germs tangent to the identity. Using notation similar as in the case of $1$-dimensional analytic diffeomorphisms, we call such germs which are not roots of the identity \emph{parabolic} Dulac germs. 

In \cite{prvi}, we have described the \emph{Ecalle-Voronin-like} moduli of analytic classification for a bigger class of parabolic generalized Dulac germs. Parabolic \emph{generalized Dulac germs} defined in \cite{prvi} are a class of germs, including parabolic Dulac germs, that admit a particular type of transserial power-logarithmic asymptotic expansion, called the \emph{generalized Dulac asymptotic expansion}. They are, like Dulac germs, defined on a \emph{standard quadratic domain}: a universal covering of $\mathbb C$ punctured at the origin with a prescribed decreasing radius as the absolute value of the argument increases. Their moduli are given as a doubly infinite sequence of pairs of germs of diffeomorphisms fixing the origin, having a symmetry property with respect to the positive real axis and a rate of decrease of radii of convergence adapted to the standard quadratic domain of definition. Similarly as in the well-known case of analytic parabolic germs treated in \cite{voronin}, it was shown in \cite{prvi} and \cite{mrrz2} that the formal class of a  generalized Dulac germ is described by three parameters, but the normalizing change diverges and defines analytic functions only on overlapping attracting and repelling sector-like domains called \emph{petals}. There are countably many petals filling the standard quadratic domain and the comparison of normalizing changes on their intersections, together with the formal class of the germ, gives its \emph{modulus of analytic classification}. 

As a continuation of \cite{prvi}, in this paper we describe the space of moduli, that is, we solve the problem of realization of moduli  of analytic classification in the class of parabolic generalized Dulac germs. For each formal class and a double sequence of germs of diffeomorphisms fixing the origin with controlled radii of convergence, we construct an analytic germ defined on a standard quadratic domain realizing them. 

However, on a \emph{big} standard quadratic domain we did not succeed in attributing a \emph{unique} power-logarithmic transserial asymptotic expansion to the constructed germ. Transseries are indexed by ordinals, which can either be successor ordinals or limit ordinals. The definition of a transserial asymptotic expansion of a certain type is dependent on the choice of the summation method at limit ordinal steps. This choice is called a \emph{section function} in \cite{MRRZ2Fatou}. To ensure uniqueness of the asymptotic expansion, we should be able to make a canonical choice of the section function. See \cite{MRRZ2Fatou} for more details on the problem of well-defined transserial asymptotic expansions and the notion of section functions. 

Moreover, we prove that, on a \emph{smaller} linear domain, there exists a parabolic generalized Dulac germ of a given formal type which realizes the given sequence of diffeomorphisms as its analytic moduli. On this smaller domain we are able to choose a canonical method of summation on limit ordinal steps, a \emph{Gevrey-type} sum, corresponding to the definition of the generalized Dulac asymptotic expansion requested in the definition of generalized Dulac germs.

In both constructions we use a \emph{Cauchy-Heine} integral construction as in e.g. \cite{loray2}, motivated by the realization of analytic moduli for saddle nodes in \cite{teyssier}. The advantage of the Cauchy-Heine construction over the standard use of uniformization method, as in \cite{voronin}, is that Cauchy-Heine integrals provide the control of power-logarithmic asymptotic expansions. 
\bigskip

Let us first recall briefly main definitions and results from \cite{prvi}. 

\subsection{Main definitions}

Recall from Ilyashenko \cite{ilyalim} the definition of almost regular germs. We call them \emph{Dulac germs} in \cite{prvi} and also here. They are defined on a \emph{standard quadratic domain} $\mathcal R_C$. It is a subset of the Riemann surface of the logarithm, in the logarithmic chart $\zeta=-\log z$ given by: 
\begin{equation}\label{eq:sqd}
\varphi\big(\mathbb C^+\setminus \overline K(0,R)\big),\ \varphi(\zeta)=\zeta+C(\zeta+1)^{\frac{1}{2}},\ C>0,\, R>0.
\end{equation}
Here, $\mathbb C^+=\{\zeta\in\mathbb C:\ \text{Re}(\zeta)>0\}$ and $\overline{K}(0,R)=\{\zeta\in\mathbb C:|\zeta|\leq R\}.$ See Figure~\ref{fig:prva}. In the sequel, we switch between the two variables, the $z$-variable and the $\zeta$-variable, as needed. By abuse, we use the same name \emph{standard quadratic domain} for the domain in the $\zeta$-variable defined by \eqref{eq:sqd} and for its preimage by $\zeta=-\log z$ in the universal covering of $\mathbb C^*$. In $z$-variable we use the notation $\mathcal R_C$, while we use the notation $\widetilde{\mathcal R}_C$ for its image by $\zeta=-\log z$ in the $\zeta$-variable.

\begin{defi}[Definition~2.1. in \cite{prvi}, adapted from \cite{ilyalim}, \cite{roussarie}]\label{def:Dulac} We say that a germ $f$ is a \emph{Dulac germ} if it is
\begin{enumerate}

\item \emph{holomorphic} and bounded on some standard quadratic domain $\mathcal R_C$ and real on $\{z\in\mathcal R_C:\ \mathrm{Arg} (z)=0\}$,

\item admits in $\mathcal R_C$ a \emph{Dulac asymptotic expansion}\footnote{\emph{uniformly} on $\mathcal R_C$, see \cite[Section 24E]{ilya}: for every $\lambda>0$ there exists $n\in\mathbb N$ such that
$$
\big|f(z)-\sum_{i=1}^{n}z^{\lambda_i}P_i(-\log z)\big|=o(z^{\lambda}),
$$
uniformly on $\mathcal R_C$ as $|z|\to 0$. }
\begin{equation}\label{eq:dulac}
\widehat f(z)=\sum_{i=1}^{\infty}z^{\lambda_i} P_i(-\log z),\ c>0,\ z\to 0,
\end{equation}
where $\lambda_i\geq 1$, $i\in\mathbb N$, are strictly positive, finitely generated\footnote{\emph{Finitely generated} in the sense that there exists $n\in\mathbb N$ and $\alpha_1>0,\ldots,\alpha_n>0,$ such that each $\lambda_i$, $i\in\mathbb N$, is a finite linear combination of $\alpha_j$, $j=1,\ldots,n$, with coefficients from $\mathbb Z_{\geq 0}.$ For Dulac maps that are the first return maps of saddle polycycles, $\alpha_j$, $j=1,\ldots,n$, are related to the ratios of hyperbolicity of the saddles.} and strictly increasing to $+\infty$ and $P_i$ is a sequence of polynomials with real coefficients, and $P_1\equiv A$, $A>0$.
\end{enumerate}
Moreover, if a Dulac germ is tangent to the identity, i.e,
$$
f(z)=z+o(z),\ z\in\mathcal R_C,
$$
and if $f^{\circ q}\neq\mathrm{id}$, $q\in\mathbb N$, we call it a \emph{parabolic} Dulac germ.
\end{defi}

By a \emph{germ on a standard quadratic domain} \cite{ilyalim}, we mean an equivalence class of functions that coincide on some standard quadratic domain (for arbitrarily big $R>0$ and $C>0$). 
\begin{figure}[h!]
\includegraphics[scale=0.4]{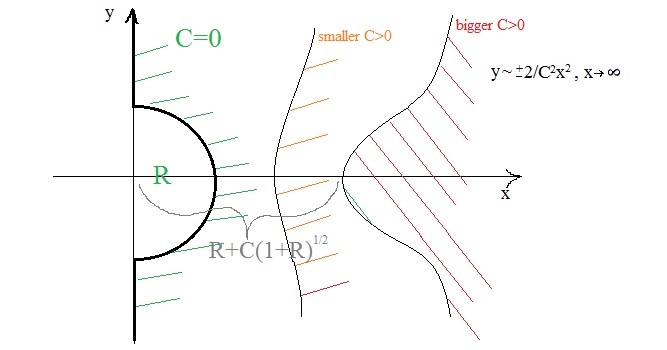}
\caption{Several standard quadratic domains $\widetilde{\mathcal R}_C$, $C>0$, in the logarithmic chart.}\label{fig:prva}
\end{figure}

The radii of the standard quadratic domain in $z$-variable tend to zero with an exponential speed as we increase the level of the Riemann surface. If by $\theta\in\big[(k-1)\pi,(k+1)\pi\big)$ we denote the arguments of the $k$-th level of the surface $\mathcal R_C$, $k\in\mathbb Z,$ and by $\theta_k:=k\pi$, $k\in\mathbb Z,$ then the maximal radii $r(\theta_k)$ by levels $k\in\mathbb Z$ decrease at most at the rate:
\begin{equation*}
Ke^{-D\sqrt{\frac{|k|\pi}{2}}},\ |k|\to\infty,\ \text{for some }D>0,\ K>0.
\end{equation*}

\bigskip

In \cite[Definition 2.3]{prvi}, a larger \emph{parabolic generalized Dulac class} is introduced. It contains parabolic Dulac germs. We repeat the definition of \emph{parabolic generalized Dulac germs} in Definition~\ref{def:gD} below. The Dulac asymptotic expansion requested in Definition~\ref{def:Dulac} of Dulac germs is substituted by a particular \emph{transserial} power-logarithmic asymptotic expansion. 

In this paper, we give realization results for any given sequence of moduli satisfying some uniform bound in the parabolic generalized Dulac class, but for parabolic generalized Dulac germs defined on a smaller domain that we call a \emph{standard linear domain}. Due to technical reasons in Cauchy-Heine construction, on standard quadratic domain we get realization results by germs for which we are unable to prove unicity of the transserial asymptotic expansion after the first three terms. To be able to define the unique transserial asymptotic expansion of a germ of a certain type, we should be able to prescribe a \emph{canonical} way of summation, or \emph{section function} \cite{MRRZ2Fatou}, at limit ordinal steps. In the linear case, the estimates of the Cauchy-Heine integrals give us sufficiently good \emph{Gevrey-like} bounds, and thus a canonical way to attribute the sum, at limit ordinal steps. This canonical choice is the one defining parabolic generalized Dulac germs and expansions, see Definitions~\ref{def:logg} and \ref{def:gD} below. On the other hand, the bounds in our construction on standard quadratic domains are weaker. 

It is important to note that the germ  obtained by Cauchy-Heine construction on a linear domain is not the restriction of the germ  constructed on a bigger quadratic domain, since we apply Cauchy-Heine integrals along different lines of integration, see \eqref{eg:lint} under condition \eqref{eq:imam} for standard quadratic domains or \eqref{eq:imam1} for standard linear domains. For details, see Remarks~\ref{rem:linqua} and \ref{rem:final}.

By \cite{ilya, roussarie}, a standard linear domain is not sufficiently \emph{large} to apply Phragmen-Lindel\" of and get injectivity of the mapping $f\mapsto \widehat f$, where $\widehat f$ is the generalized Dulac asymptotic expansion of $f$. 

\begin{defi}
A standard linear domain $\widetilde{\mathcal R}_{a,b}$, $a>0,\ b\geq 0,$ in the logarithmic chart is a subset of $\mathbb C_+$ given by:
$$
\widetilde{\mathcal R}_{a,b}:=\Big\{\zeta\in\mathbb C_+: \ b-a\mathrm{Re}(\zeta)<\mathrm{Im}(\zeta)<-b+a\mathrm{Re}(\zeta),\ \mathrm{Re}(\zeta)>\frac{b}{a}\Big\}.
$$
\end{defi}
\begin{figure}[h!]
\includegraphics[scale=0.4]{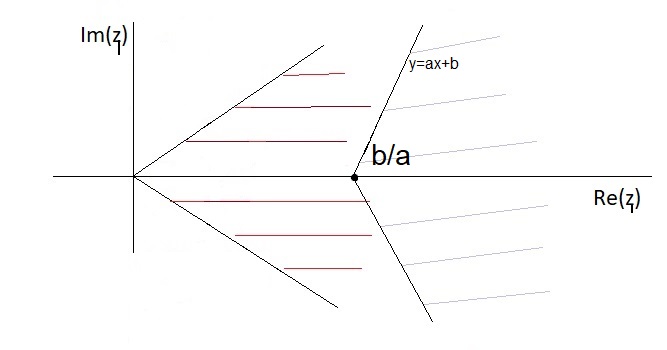}
\caption{Several standard linear domains $\widetilde{\mathcal R}_{a,b}$, $a>0,\ b\geq 0$, in the logarithmic chart.}
\end{figure}
Analogously, by $\mathcal R_{a,b}$ we denote the image by $z=e^{-\zeta}$ of $\widetilde{\mathcal R}_{a,b}$. It is a subset of the Riemann surface of the logarithm. 
\medskip

We recall from \cite{prvi} the definition of the \emph{parabolic generalized Dulac class}. We will call an \emph{$\boldsymbol\ell$-cusp} an open cusp that is the image of an open sector $V$ of positive opening at $0$ by the change of variables $\boldsymbol\ell=-\frac{1}{\log z}$, and we will denote it by $S=\boldsymbol\ell(V)$. See Figure~\ref{figura:fig3}. Any open $\boldsymbol\ell$-cusp $\boldsymbol\ell(V')\subset S$, where $V'\subset V$ is a proper subsector, will be called a \emph{proper $\boldsymbol\ell$-subcusp} of $S$.

\begin{figure}[h!]
\includegraphics[scale=0.2]{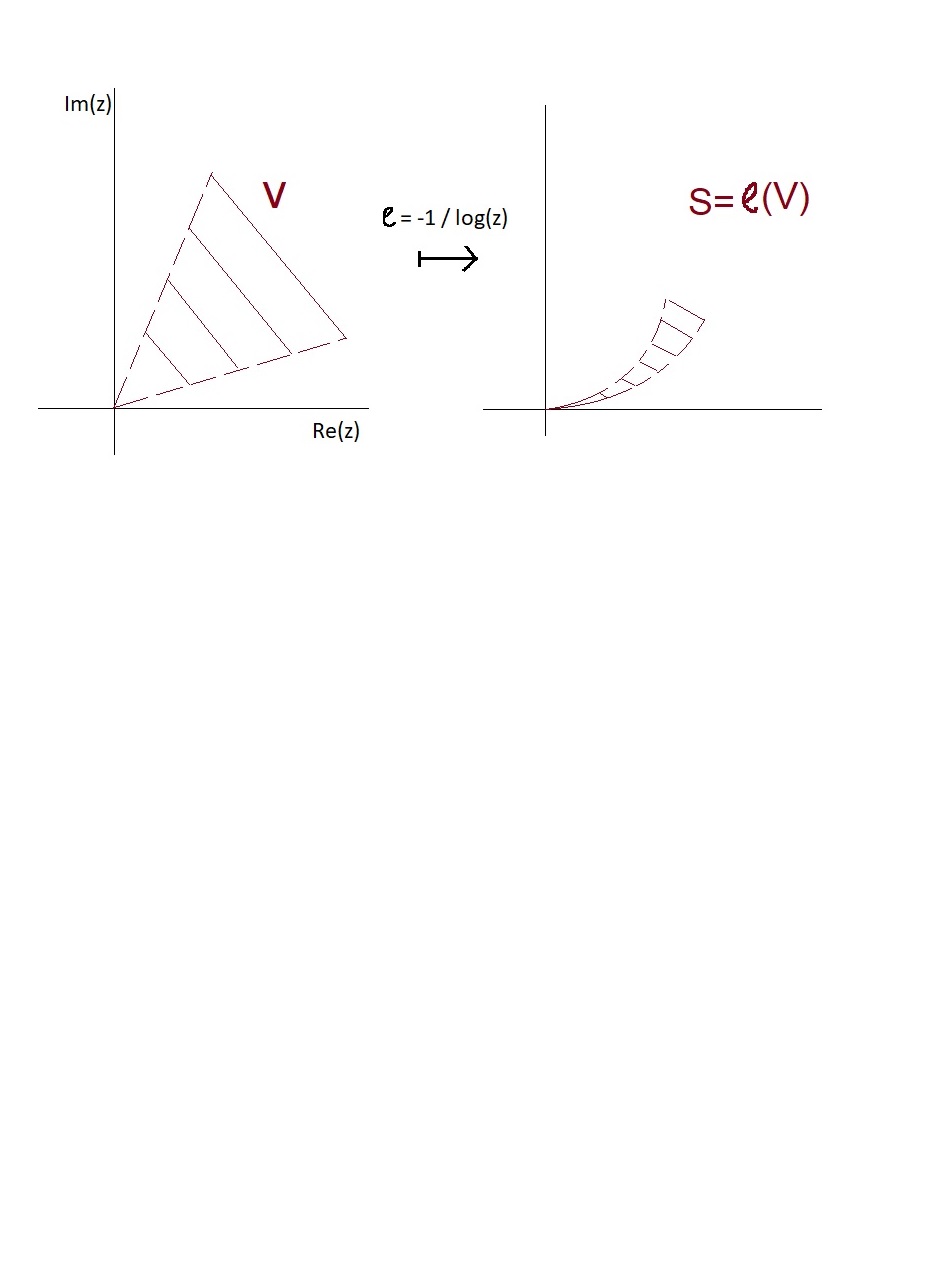}
\vspace{-4.5cm}
\caption{$\boldsymbol\ell$-cusp.}\label{figura:fig3}
\end{figure}
 
\begin{defi}[$\log$-Gevrey asymptotic expansions on $\boldsymbol\ell$-cusps,\ Definition~4.1 in \cite{prvi}]\label{def:logg}
Let $F$ be a germ analytic on an $\boldsymbol\ell$-cusp $S=\boldsymbol\ell(V)$. We say that $F$ admits $\widehat F(\boldsymbol\ell)=\sum_{k=0}^{\infty} a_k \boldsymbol\ell^k$, $a_k\in\mathbb C$, as its \emph{$\log$-Gevrey asymptotic expansion of order $m>0$} if, for every proper $\boldsymbol\ell$-subcusp $S'=\boldsymbol\ell(V')\subset S$, $V'\subset V,$ there exists a constant $C_{S'}>0$ such that, for every $n\in\mathbb N$, $n\geq 2,$ it holds that:
\begin{equation*}\label{eq:lg1}
|F(\boldsymbol\ell)-\sum_{k=0}^{n-1}a_k \boldsymbol\ell^k|\leq C_{S'} \cdot m^{-n}\cdot \log^n n \cdot e^{-\frac{n}{\log n}}|\boldsymbol\ell|^n,\ \boldsymbol\ell\in S'.
\end{equation*}
\end{defi}

\noindent For more details on properties of $\log$-Gevrey classes and for the proof of their closedness to algebraic operations $+,\cdot$ and to differentiation, see \cite[Section 4]{prvi}. We state here just the following variation of the \emph{Watson's lemma} for $\log$-Gevrey expansions, that will be immediately important for the definition and the uniqueness of generalized Dulac expansions. If $\widehat F(\boldsymbol\ell)$ is the $\log$-Gevrey asymptotic expansion of order $m>0$ of a function $F$ analytic on $\boldsymbol\ell$-cusp $S=\boldsymbol\ell(V)$, where $V$ is a sector of opening \emph{strictly bigger} than $\frac{\pi}{m}$, then $F$ is the unique analytic function on $S$ that admits $\widehat F(\boldsymbol\ell)$ as its $\log$-Gevrey asymptotic expansion of order $m$. The proof can be found in \cite[Section 4, Corollary 4.4]{prvi}.
\medskip

We prove in \cite[Proposition 2.2]{prvi} that every parabolic germ $f$ on $\mathcal R_C$ (resp. $\mathcal R_{a,b}$) that satisfies the uniform asymptotics: 
\begin{align*}
\big|f(z)-(z-az^{\alpha}\boldsymbol\ell^m)\big|\leq  c |z^\alpha\boldsymbol\ell^{m+1}|,\ &z\in\mathcal R_{C} \text{ (resp. $\mathcal R_{a,b}$)},\\
 &\alpha>1,\ m\in\mathbb Z,\ a\neq 0,\ c>0,
\end{align*}
has a local \emph{flower-like dynamics} at the \emph{origin}. That is, $\mathcal R_C$ (resp. $\mathcal R_{a,b}$) is a union of countably many overlapping invariant attracting and repelling \emph{petals}\footnote{A \emph{petal} is a union of sectors, whose openings increase continuously, up to some fixed opening, while their radii decrease, see e.g. \cite{loray2}.} $V_j^+$ resp. $V_j^-$, $j\in\mathbb Z$, centered at directions $a^{-\frac{1}{\alpha-1}}$ resp. $(-a)^{-\frac{1}{\alpha-1}}$, and of opening $\frac{2\pi}{\alpha-1}$. 

The dynamics in the $\zeta$-variable on a standard quadratic domain $\widetilde{\mathcal R}_C$ is shown on Figure~\ref{fig:dyn}. The \emph{sectors} of opening $\theta>0$ in the $z$-variable become \emph{horizontal strips} of width $\theta>0$ in the $\zeta$-variable. Analogously, the images of petals of opening $\frac{2\pi}{\alpha-1}$ in the $z$-variable are open sets tangentially approaching strips of width $\frac{2\pi}{\alpha-1}$, as $\mathrm{Im}(\zeta)\to +\infty$, in the $\zeta$-variable. We denote them in the $\zeta$-variable by $\tilde V_j^+$ and $\tilde V_j^-$, $j\in\mathbb Z$, see Figure~\ref{fig:dyn}. By abuse,  in the $\zeta$-variable, we also call them \emph{petals}.

\begin{figure}[h!]
\includegraphics[scale=0.2]{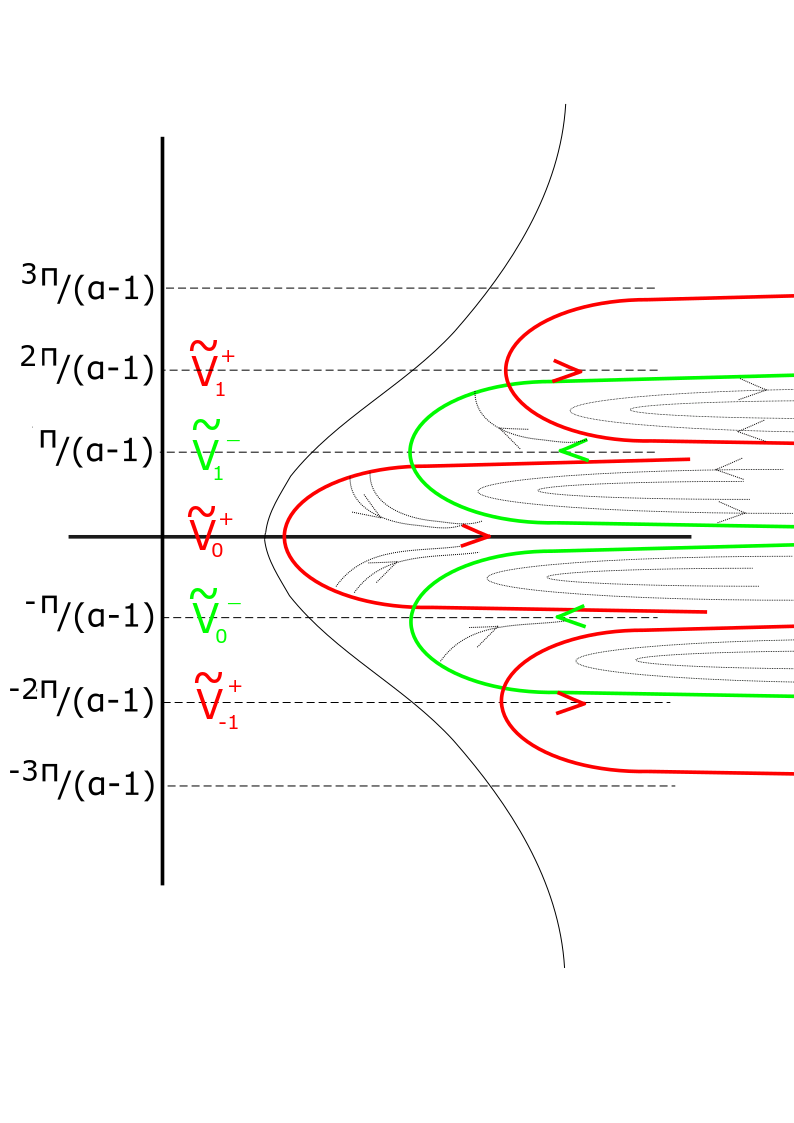}
\vspace{-1cm}
\caption{The outline of the dynamics of a generalized Dulac germ on petals $\tilde V_j^{\pm}$, $j\in\mathbb Z$, along a standard quadratic domain $\widetilde{\mathcal R}_C$ in the logarithmic chart, case $a>0$ in \eqref{eq:uniest}, \cite[Figure 3.1]{prvi}.}\label{fig:dyn}
\end{figure}

\begin{defi}[Parabolic generalized Dulac germs, Definition 2.3 in \cite{prvi}]\label{def:gD} We say that a parabolic germ $f$, analytic on a standard quadratic domain $\mathcal R_C$ $($or: standard linear domain $\mathcal R_{a,b})$,  that maps $\{\mathrm{arg}(z)=0\}\cap\mathcal R_C$ $($resp. $\{\mathrm{arg}(z)=0\}\cap\mathcal R_{a,b})$ to itself, satisfying \begin{align}\label{eq:uniest}|f(z)-z+az^{\alpha}\boldsymbol\ell^m|\leq c|z^{\alpha}\boldsymbol\ell^{m+1}|,\ a\neq 0,\ \alpha>1,&\ m\in\mathbb Z, \ c>0,\\ &z\in\mathcal R_C \text{, resp. $z\in\mathcal R_{a,b}$}\nonumber,\end{align}  is a \emph{parabolic generalized Dulac germ} if, on its every invariant petal $V_j^{\pm},\ j\in\mathbb Z$, of opening $\frac{2\pi}{\alpha-1} $, it admits an asymptotic expansion of the form:
\begin{equation*}
f(z)=z+\sum_{i=1}^{n} z^{\alpha_i} R_i^{j,\pm}(\boldsymbol\ell)+o(z^{\alpha_n+\delta_n}), \ \delta_n>0,
\end{equation*}
for every $n\in\mathbb N$, as $z\to 0$ on $V_j^{\pm}$. Here, $\alpha_1=\alpha$, $\alpha_i>1$ are strictly increasing to $+\infty$ and finitely generated, and $R_i^{j,\pm}(\boldsymbol\ell)$ are analytic functions on open cusps $\boldsymbol\ell(V_j^{\pm})$ which admit common \emph{$\log$-Gevrey asymptotic expansions} $\hat R_i(\boldsymbol\ell)$ of order strictly bigger than $\frac{\alpha-1}{2}$, as $\boldsymbol\ell\to 0$:
$$
\hat R_i(\boldsymbol\ell)=\sum_{k=N_i}^{\infty}a_k^i \boldsymbol\ell^k,\ a_k^i\in\mathbb R,\ N_i\in\mathbb Z.
$$
\smallskip

We then say that the transseries $\widehat f$ given by:
\begin{equation}\label{eq:gDexp}
\widehat f(z):=z+\sum_{i=1}^{\infty} z^{\alpha_i} \widehat R_i(\boldsymbol\ell)
\end{equation}
is the unique generalized Dulac asymptotic expansion of $f$. Such $\widehat f$ is called a parabolic \emph{generalized Dulac series}.
\end{defi}
\noindent Note that all coeficients of the expansion are real, due to the invariance of $\mathbb R_+$ under $f$. 

Moreover, we assume in the sequel that $a>0$ in \eqref{eq:uniest}. That is, that $\mathbb R_+\cap\mathcal R_C$ is an attracting direction. If $a<0$, we consider the inverse generalized Dulac germ $f^{-1}$. Indeed, it was proven in \cite[Proposition 8.2]{prvi} that parabolic generalized Dulac germs form a group under composition.
\medskip

A generalized Dulac asymptotic expansion is an asymptotic expansion in the formal class of transseries $\widehat {\mathcal L}(\mathbb R)$. The class of power-logarithmic transseries $\widehat {\mathcal L}(\mathbb R)$ was first introduced in \cite{mrrz2}, as the class of transseries of the form:
$$
\widehat{f}(z)=\sum_{i=1}^{\infty}\sum_{m=N_i}^{\infty} a_{i,m} z^{\alpha_i} \boldsymbol\ell^{m},\ a_{i,m}\in\mathbb R,
$$
where $\alpha_i>0$ are \emph{finitely generated} and strictly increasing to $+\infty$, and $N_i\in\mathbb Z$, $i\in\mathbb N$. Here, $\boldsymbol\ell=\frac{1}{-\log z}$.
As discussed in \cite{MRRZ2Fatou}, an asymptotic expansion of a germ in $\widehat{\mathcal L}(\mathbb R)$ is in general not well-defined, nor unique. The generalized Dulac expansion is a sectional asymptotic expansion (see \cite{MRRZ2Fatou} for precise definition of sections) that becomes unique after a canonical choice of section functions (the summation method) at limit ordinal steps - here, the \emph{$\log$-Gevrey sums} of a certain order. 
\smallskip


The parabolic Dulac (almost regular in \cite{ilyalim}) germs from Definition~\ref{def:Dulac} are trivially parabolic generalized Dulac germs. In that case we have a canonical choice for summation at limit ordinal steps, since $\widehat R_i$ in \eqref{eq:gDexp} are \emph{polynomials} in $\boldsymbol\ell^{-1}$. Polynomial functions in $\boldsymbol\ell^{-1}$ are convergent Laurent series in $\boldsymbol\ell$. 
\medskip

Recall the following \emph{formal classification} result from \cite{mrrz2}, repeated in \cite{prvi} for the case of real coefficients. By a normalizing change of variables $\widehat\varphi\in\widehat{\mathcal L}(\mathbb R)$ of the form $\widehat \varphi(z)=cz+\mathrm{h.o.t.}$\footnote{higher-order terms, lexicographically with respect to orders of monomials}, $c\neq 0$, every parabolic transseries $\widehat f\in\widehat{\mathcal L}(\mathbb R)$ of the form: 
$$
\widehat f(z)=z-az^\alpha\boldsymbol\ell^m+\mathrm{h.o.t.},\ a>0,\ \alpha>1,\ m\in\mathbb Z,
$$
can be reduced to a formal normal form given as a formal time-$1$ map of a vector field:
\begin{align}\begin{split}
\widehat f_0:=\text{Exp}(X_0).\mathrm{id}=z-&z^\alpha\boldsymbol\ell^m+\rho z^{2\alpha-1}\boldsymbol\ell^{2m+1}+\mathrm{h.o.t.},\\
&\text{where }X_0(z)=\frac{-z^\alpha\boldsymbol\ell^{m}}{1+\frac{-\alpha}{2}z^{\alpha-1}\boldsymbol\ell^{m}+\big(\frac{m}{2}+\rho\big) z^{\alpha-1}\boldsymbol\ell^{m+1}}\frac{d}{dz}.\label{eq:fnf}\end{split}
\end{align}
The triple $(\alpha,m,\rho)$, $\alpha>1$, $m\in\mathbb Z$, $\rho\in\mathbb R$, are called the $\widehat{\mathcal L}(\mathbb R)$-\emph{formal invariants} of $\widehat f$. 
\medskip

In \cite{prvi}, we have introduced the notion of \emph{analytic conjugacy} or \emph{analytic equivalence} of parabolic generalized Dulac germs, see \cite[Definition 2.4]{prvi}. We repeat the definition here.
For simplicity, we work here with \emph{normalized} parabolic generalized Dulac germs whose second coefficient is equal to $-1$. Each parabolic generalized Dulac germ of the form $f(z)=z-az^\alpha\boldsymbol\ell^m+o(z^\alpha\boldsymbol\ell^m)$ , $a>0$, can be brought to a parabolic generalized Dulac germ of the form:
\begin{equation}\label{eq:norm}
f(z)=z-z^\alpha\boldsymbol\ell^m+o(z^\alpha\boldsymbol\ell^m),\ \alpha>1,\ m\in\mathbb Z.
\end{equation}
This is done simply by a real homothecy $\varphi(z)=a^{\frac{1}{\alpha-1}}z$, taking $a^{\frac{1}{\alpha-1}}\in\mathbb R_+$, which preserves the invariance of $\mathbb R_+$ in the definition of generalized Dulac germs.

In the case when $a<0$, we work with the inverse parabolic generalized Dulac germ.

\begin{defi}[Analytic equivalence of parabolic generalized Dulac germs, Definition~2.4 from \cite{prvi}] We say that two normalized parabolic generalized Dulac germs $f$ and $g$ of the form \eqref{eq:norm} defined on a standard quadratic domain $\mathcal R_C$  $($or on a standard linear domain $\mathcal R_{a,b}$$)$ are \emph{analytically conjugated} if:
\begin{enumerate}
\item their generalized Dulac asymptotic expansions $\widehat f$ and $\widehat g$ are formally conjugated\footnote{i.e. have the same $\widehat{\mathcal L}(\mathbb R)$-formal invariants $(\alpha,m,\rho),\ \alpha>1,\ m\in\mathbb Z,\ \rho\in\mathbb R.$} in $\widehat{\mathcal L}(\mathbb R)$, and
\item there exists a germ of a diffeomorphism $h(z)=z+o(z)$ of a standard quadratic domain $\mathcal R_C$ $($or a standard linear domain $\mathcal R_{a,b}$$)$, such that:
$$
g=h^{-1}\circ f\circ h, \text{\quad  on } \mathcal R_{C} \text{ $($resp. }\mathcal R_{a,b}).
$$
\end{enumerate}
\end{defi}
\medskip

In \cite[Theorem B]{prvi} we have derived the following result about the moduli of analytic classification for parabolic generalized Dulac germs in the sense of Ecalle, Voronin. For more details, see \cite{prvi}. 

\noindent Let $f$ be a parabolic generalized Dulac germ defined on a standard quadratic (or linear) domain, belonging to $\widehat{\mathcal L}(\mathbb R)$-formal class $(2,m,\rho)$, $m\in\mathbb Z,\ \rho\in\mathbb R$. As in \cite{prvi}, $\alpha=2$ is taken for simplicity. This can be done without loss of generality, since any normalized generalized Dulac germ of the form \eqref{eq:norm} can be brought to the form $f(z)=z-z^2\boldsymbol\ell^m+o(z^2\boldsymbol\ell^m)$, $m\in\mathbb Z$, by the change of variables \begin{equation}\label{eq:cha}z\mapsto (\alpha-1)^{-\frac{m}{\alpha-1}}z^\frac{1}{\alpha-1},\end{equation} analytic on a standard quadratic (i.e. standard linear) domain and depending only on $\alpha$ and $m$. Therefore, two parabolic generalized Dulac germs are analytically conjugated if and only if, after the change of coordinates \eqref{eq:cha}, the corresponding germs are analytically conjugated. For details, see \cite[Proposition 9.1]{prvi}.

Let $(\Psi_{\pm}^j)_{j\in\mathbb Z}$ be the analytic \emph{Fatou coordinates} of $f$ on attracting and repelling petals $V_{\pm}^j$, $j\in\mathbb Z$, along the domain (standard quadratic or standard linear). Recall that a \emph{Fatou coordinate} $\Psi_{\pm}^j$ of a generalized Dulac germ $f$ is an analytic map, defined on the petal $V_{\pm}^j$, $j\in\mathbb Z,$ conjugating the map $f$ on the petal to a translation by $1$: $$\Psi_{\pm}^j\circ f-\Psi_{\pm}^j=1, \text{ on }V_\pm^j,\ j\in\mathbb Z.$$ The existence and the uniqueness, up to an additive constant, of the petal-wise analytic Fatou coordinate of a generalized Dulac germ under some additional assuption on its power-logarithmic asymptotic expansion is proven in \cite[Theorem~A]{prvi}. 

We have proved in \cite[Theorem B]{prvi} that there exists a symmetric (with respect to $\mathbb R_+$-axis) double sequence $(h_{0}^j,h_{\infty}^j)_{j\in\mathbb Z}$ of pairs of analytic germs of diffeomorphism from $\mathrm{Diff}(\mathbb C,0)$, defined on discs of radii $\sigma_j$ bounded from below by:
\begin{equation}\label{eq:fall}
\sigma_j\geq K_1 e^{-Ke^{C\sqrt{|j|}}},\ j\in\mathbb Z,\ \text{for some }K_1,\ K,\ C>0,
\end{equation}
that satisfy:
\begin{align}\begin{split}\label{eq:moduli}
&h_0^j(t):=e^{-2\pi i \Psi_+^{j-1}\circ (\Psi_{-}^{j})^{-1}(-\frac{\log t}{2\pi i})},\ t\in(\mathbb C,0),\\
&h_\infty^j(t):=e^{2\pi i \Psi_-^j\circ (\Psi_{+}^j)^{-1}(\frac{\log t}{2\pi i})},\ t\in(\mathbb C,0),\ j\in\mathbb Z.\end{split}
\end{align}
We have proved that this sequence of pairs of diffeomorphisms and the formal class $(2,m,\rho)$ form a \emph{complete system of analytic invariants} of a parabolic generalized Dulac germ $f$. These diffeomorphisms are called the \emph{horn maps} for $f$. 
\smallskip

As in \cite{prvi}, we say that the sequence of pairs $(h_0^j,h_\infty^j)_{j\in\mathbb Z}$ of analytic germs of diffeomorphisms is \emph{symmetric with respect to $\mathbb R_+$} if the following holds (on the domains of definition of $h_0^j$ and $h_\infty^j$, $j\in\mathbb Z$):
\begin{equation}\label{eq:sim}
\overline{\big(h_{0}^{-j+1}\big)^{-1}(t)}\equiv h_{\infty}^j(\overline t),\ t\in (\mathbb C,0),\ j\in\mathbb Z.
\end{equation}
This symmetry of moduli for parabolic generalized Dulac germs comes from the invariance of $\mathbb R_+$ under $f$, by the Schwarz reflection principle, see \cite[Proposition 9.2]{prvi}.
\smallskip

Note that the lower bound \eqref{eq:fall} comes from the standard quadratic domain of definition of $f$. However, the construction of moduli of analytic classification from \cite[Theorem B]{prvi} goes through in the same way for parabolic generalized Dulac germs defined on smaller standard linear domains. In  the  case that the germ is defined only on a standard \emph{linear} domain, it is easy to see that the radii of definition of its horn maps may decrease quicker. More precisely, they are bounded from below by:
\begin{equation}\label{eq:velradlin}\sigma_j\geq K_1 e^{-Ke^{C|j|}},\ j\in\mathbb Z,\ \text{for some}\ K_1,\ K,\ C>0.\end{equation}
\smallskip

By \emph{horn maps}, we in fact mean the \emph{equivalence classes} of germs, up to the following identifications. Two sequences \begin{equation}\label{eq:hh}(h_{0}^j,h_{\infty}^j;\ \sigma_j)_{j\in\mathbb Z} \text{ and } (k_{0}^j, k_{\infty}^j;\ \breve \sigma_j)_{j\in\mathbb Z}\end{equation} with maximal radii of convergence $\sigma_j$ resp. $\breve \sigma_j$, satisfying lower bounds of the type \eqref{eq:fall} or \eqref{eq:velradlin}, are \emph{equivalent} if there exist sequences $(\beta_j)_{j\in\mathbb Z},\ (\gamma_j)_{j\in\mathbb Z}\in \mathbb C^*$ such that \begin{equation}\label{eq:ekvij}h_0^j(t)=\beta_{j-1}\cdot k_0^j\big(\gamma_j t\big),\ h_\infty^j(t)=\gamma_j\cdot k_\infty^j\big(\beta_j t\big),\ j\in\mathbb Z.\end{equation} 

Additionally, we assume that the sequences of pairs \eqref{eq:hh} are both \emph{symmetric} as in \eqref{eq:sim}, since they represent the moduli of generalized parabolic Dulac germs for which $\mathbb R_+$ is invariant. In this case, the complex sequences $(\gamma_j)_{j\in\mathbb Z}$, $(\beta_j)_{j\in\mathbb Z}$ in equivalence relation \eqref{eq:ekvij} are not arbitrary. Indeed, if such sequences exist, they should, by \eqref{eq:sim} and \eqref{eq:ekvij}, be related to germs of diffeomorphisms $k_0^j$ and $k_\infty^j$, $j\in\mathbb Z,$ by the following: 
\begin{equation}\label{eq:hhb}
\frac{1}{\overline{\gamma_{-j+1}}}\cdot k_\infty^j\big(\frac{1}{\overline{\beta_{-j}}}t\big)=\gamma_j \cdot k_\infty^j(\beta_j t), \ t\in(\mathbb C,0),\ j\in\mathbb Z.
\end{equation}
By basic calculations (comparing the coefficients with each power $t^k$ in the Taylor expansion of \eqref{eq:hhb}), depending on the nature of diffeomorphisms $k_\infty^j$, $j\in\mathbb Z$, the equality \eqref{eq:hhb} is equivalent to the following conditions on the sequences $(\beta_j)_{j\in\mathbb Z}$ and $(\gamma_j)_{j\in\mathbb Z}$:
\begin{enumerate}
\item
$\gamma_j\cdot\overline{\gamma_{-j+1}}=\frac{1}{\overline{\beta_{-j}}\cdot \beta_j}$; for $j\in\mathbb Z$ for which $k_\infty^j$ is linear,
\item
$\beta_j\cdot\overline{\beta_{-j}}=r,\ \gamma_j\cdot\overline{\gamma_{-j+1}}=1/r$, for any $r\in\mathbb C$ such that $r^m=1$; for $j\in\mathbb Z$ for which the non-constant part\footnote{the part obtained by subtracting from $\frac{k_\infty^j}{\mathrm{id}}$ the constant term in its Taylor expansion} of $\frac{k_\infty^j}{\mathrm{id}}$ is a diffeomorphism in the variable $t^m$, for some $m\in\mathbb N$, $m\geq 2$,
\item $\beta_j\cdot\overline{\beta_{-j}}=1,\ \gamma_j\cdot\overline{\gamma_{-j+1}}=1$; for all other $j\in\mathbb Z$ (the \emph{generic case}). 
\end{enumerate}

\section{Main results}  

For simplicity, as in \cite{prvi}, we consider here only parabolic generalized Dulac germs of order $2$ in variable $z$, defined on a standard quadratic domain $\mathcal R_C$, 
$$f(z)=z-az^2 \boldsymbol\ell^{m}+o(z^2\boldsymbol\ell^m),\ a>0,\ m\in\mathbb Z_-.$$ 

More general case $\alpha>1$ can be reduced to the case $\alpha=2$, as discussed above. Also, the realization result for $\alpha>1$ can be concluded in the same way as for $\alpha=2$. The number of petals on each level of the surface of the logarithm depends on $\alpha$.  
\medskip

In this paper, we solve the realization problem in the subset of \emph{prenormalized} parabolic generalized Dulac germs:
\begin{equation*}\label{eq:formpre}
f(z)=z-z^2\boldsymbol\ell^m+\rho z^3\boldsymbol\ell^{2m+1}+o(z^3\boldsymbol\ell^{2m+1}),\ m\in\mathbb Z,\ \rho\in\mathbb R.
\end{equation*}
Note that its formal invariants are $(2,m,\rho)$.

\noindent By Proposition~\ref{prop:pgd} in the Appendix, the sectorial Fatou coordinate of prenormalized germ $f$ is of the form:
$$
\Psi_j^{\pm}=\Psi_\mathrm{nf}+R_j^{\pm},\ \text{ on }V_j^{\pm},
$$
where $\Psi_\mathrm{nf}$ is the Fatou coordinate of the formal normal form $f_0$ of $f$, globally analytic on $\mathcal R_C$, and $R_j^{\pm}=o(1)$, as $z\to 0$, $z\in V_j^{\pm}$, are analytic on petals. Here, the formal normal form $f_0$ of $f$ is an analytic germ on $\mathcal R_C$, given as the time-$1$ map of an analytic vector field on $\mathcal R_C$:
\begin{align}\label{eq:mod}
f_0:=\text{Exp}(X_0).\mathrm{id}=z-&z^2\boldsymbol\ell^m+\rho z^{3}\boldsymbol\ell^{2m+1}+o(z^{3}\boldsymbol\ell^{2m+1}),\\ &X_0(z)=\frac{-z^2\boldsymbol\ell^{m}}{1-z\boldsymbol\ell^{m}+\big(\frac{m}{2}+\rho\big) z\boldsymbol\ell^{m+1}}\frac{d}{dz},\nonumber
\end{align}
see \eqref{eq:fnf} in the case $\alpha=2$.
\medskip

\subsection{Main theorems}\label{subsec:mt}\

Let $$f(z)=z-z^2\boldsymbol\ell^m+o(z^2\boldsymbol\ell^m),\ m\in\mathbb Z,$$ be a parabolic generalized Dulac germ. Let $\Psi_j^{\pm}$, $j\in\mathbb Z$, be its sectorially analytic Fatou coordinates on petals $V_j^{\pm}$, precisely defined in \cite[Theorem A]{prvi}. 

To a sequence of horn maps of $f$, $\big(h_{0}^j,h_\infty^j\big)_{j\in\mathbb Z}$ defined in \cite[Theorem B]{prvi} and in \eqref{eq:moduli}, there naturally corresponds a sequence of exponentially small \emph{cocycles} $(G_{0}^j,G_\infty^j)_{j\in\mathbb Z}$, defined and analytic on intersections $V_0^j$ and $V_\infty^j$ of consecutive petals, such that:
\begin{align*}
&G_0^j(z):=g_0^j(e^{- 2\pi i \Psi_{-}^j(z)}),\ z\in V_0^j,\\
&G_{\infty}^j(z):=g_{\infty}^j(e^{2\pi i \Psi_{+}^j(z)}),\ z\in V_\infty^j.
\end{align*}
Here, $V_0^j:=V_+^{j-1}\cap V_-^j$ and $V_\infty^j:=V_-^j\cap V_+^j$, $j\in\mathbb Z$, see Figure~\ref{fig:dyn}, and $g_{0}^j,\ g_\infty^j,\ j\in\mathbb Z$, are analytic germs at $t\approx 0$, such that \begin{equation}\label{eq:evo}(h_0^j)^{-1}(t)=te^{2\pi i g_0^j(t)},\ h_\infty^j(t)=te^{2\pi i g_\infty^j(t)},\ t\approx 0.\end{equation}
\smallskip

\noindent The following is an equivalent formulation of \eqref{eq:moduli} using $G_{0,\infty}^j$ and $g_{0,\infty}^j$, $j\in\mathbb Z$:
\begin{align*}
&\Psi_+^{j-1}(z)-\Psi_-^j(z)=g_0^j(e^{-2\pi i \Psi_+^{j-1}(z)})=G_0^j(z),\ z\in V_0^j, \\
&\Psi_-^j(z)-\Psi_+^j(z)=g_\infty^j(e^{2\pi i\Psi_+^j(z)})=G_\infty^j(z),\ z\in V_\infty^j,\ \ j\in\mathbb Z.
\end{align*}
\medskip

\begin{prop}[Uniform bounds by levels for horn maps of parabolic generalized Dulac germs on standard linear or quadratic domains]\label{prop:modest} Let $f(z)=z-z^2\boldsymbol\ell^m+\rho z^3\boldsymbol\ell^{2m+1}+o(z^3\boldsymbol\ell^{2m+1}),\ m\in\mathbb Z,\ \rho\in\mathbb R,$ be a prenormalized analytic germ on a standard quadratic or standard linear domain. Assume that there exists a constant $C>0$ such that: 
\begin{equation}\label{eq:fuin}
|f(z)-z+z^2\boldsymbol\ell^m-\rho z^3\boldsymbol\ell^{2m+1}|\leq C|z^3\boldsymbol\ell^{2m+2}|,
\end{equation} on some quadratic or linear subdomain. Let $(h_0^j,h_\infty^j)_{j\in\mathbb Z}$, be a sequence of its horn maps $($constructed in \cite[Theorem A]{prvi}$)$. Let $g_{0,\infty}^j(t),\ j\in\mathbb Z,$ be defined as above in \eqref{eq:evo}. Then the following \emph{uniform} bounds hold $($uniform in $j)$: there exist uniform constants $c_1,\ c_2,\ d_1,\ d_2>0$ such that, equivalently:
\begin{align}
&|h_{0,\infty}^j(t)-t|\leq d_1|t|^2,\ |(h_{0,\infty}^j)'(t)-1|\leq d_2|t|,\ \text{ or }\label{dru}\\
&|g_{0,\infty}^j(t)|\leq c_1|t|,\ |(g_{0,\infty}^j)'(t)|\leq c_2,\ 0<|t|\leq \sigma_j,\ j\in\mathbb Z.\label{prv}
\end{align}
\end{prop}
The proof, which is a consequence of uniform asymptotics \eqref{eq:fuin}, is in the Appendix. 

\noindent Note that parabolic prenormalized generalized Dulac germs, due to \eqref{eq:uniest}, satisfy  assumption \eqref{eq:fuin}, so the sequences of pairs of their horn maps satisfy uniform bounds \eqref{dru}.

\bigskip

We now state two realization theorems, Theorem~A and Theorem~B. They are both dealing with the following realization problem: \emph{given a formal class $(2,m,\rho)$, $m\in\mathbb Z,\ \rho\in\mathbb R$, and a sequence of pairs of analytic germs of diffeomorphisms $(h_0^j,h_\infty^j)_{j\in\mathbb Z}$ fixing the origin, symmetric with respect to $\mathbb R_+$, with radii of convergence $\sigma_j$ satisfying a lower bound of the type \eqref{eq:fall} and satisfying bounds \eqref{dru}, does there exist a parabolic generalized Dulac germ belonging to formal class $(2,m,\rho)$ and realizing this sequence as its sequence of horn maps?} This result can be considered as a generalization of the realization result for regular (i.e. holomorphic) parabolic germs in \cite{voronin}.
\smallskip

First, in Theorem~A, we answer the realization question positively in the class of prenormalized germs of the form:
\begin{equation}\label{eq:poc}
f(z)=z-z^2\boldsymbol\ell^m+\rho z^3\boldsymbol\ell^{2m+1}+o(z^3\boldsymbol\ell^{2m+1}),\ z\in\mathcal R_C,
\end{equation}
leaving $\mathbb R_+$ invariant and analytic on a  standard quadratic domain. However, we do not claim the uniqueness of the transserial asymptotic expansion of $f$ in $\widehat{\mathcal L}(\mathbb R)$ after the first three terms given in \eqref{eq:fuin}. In particular, we do not claim that the constructed germ is a parabolic generalized Dulac germ: we are unable to prove that it admits the generalized Dulac asymptotic expansion as defined in Definiton~\ref{def:gD}, with sufficiently strong $\log$-Gevrey bounds at limit ordinal steps, see Remark~\ref{rem:sq}.
\smallskip

In Theorem~B, we realize any sequence of pairs satisfying bounds \eqref{dru} by \emph{parabolic generalized Dulac germs} of the form \eqref{eq:poc} belonging to the formal class $(2,m,\rho)$, but on a smaller \emph{standard linear domain}. Note that such germs admit a well-defined unique \emph{generalized Dulac asymptotic expansion}. On smaller standard linear domains the map $f\mapsto \widehat f$, for parabolic generalized Dulac germs $f$, is well-defined, but the domain is too small to apply Phragmen-Lindel\" of \cite{ilya} and get injectivity. 

Note that in \cite[Theorem B]{prvi} we construct the moduli of parabolic generalized Dulac germs defined on standard quadratic domains. However, the result can be deduced in the same way for parabolic generalized Dulac germs on smaller standard linear domains, with the only difference that the rate of decrease of moduli follows the rule \eqref{eq:velradlin} instead of \eqref{eq:fall}.

To conclude, we prove in Theorem~B that, on a standard \emph{linear} domain, there is a bijective correspondence between analytic classes of parabolic prenormalized generalized Dulac germs belonging to the same formal class and all sequences of pairs of analytic germs of diffeomorphisms satisfying bounds \eqref{eq:velradlin} and \eqref{dru}, with appropriate identifications on both sides.
\smallskip

\begin{theo1}[Realization by parabolic germs on a standard quadratic domain]\label{real:first}
Let $\rho\in\mathbb R$, $m\in\mathbb Z$. Let $(h_{0}^j,h_{\infty}^{j};\,\sigma_j)_{j\in\mathbb Z}$ be a sequence of pairs of analytic germs from $\mathrm{Diff}(\mathbb C,0)$, symmetric with respect to $\mathbb R_+$ as in \eqref{eq:sim}, and with maximal radii of convergence $\sigma_j$ bounded from below by $$\sigma_j\geq K_1 e^{-Ke^{C\sqrt{|j|}}}, \ j\in\mathbb Z,$$ for some $C,\ K,\ K_1>0$. Let the elements of the sequence on their respective domains of definition satisfy the \emph{uniform} bound \eqref{dru}. Then there exists a germ \begin{equation}\label{eq:forma}f(z)=z-z^2\boldsymbol\ell^m+\rho z^3\boldsymbol\ell^{2m+1}+o(z^3\boldsymbol\ell^{2m+1}),\end{equation} analytic on a standard quadratic domain, leaving $\mathbb R_+$ invariant and satisfying \eqref{eq:fuin},  that realizes this sequence as its horn maps\footnote{To be able to define \emph{horn maps} of such germ, recall from \cite[Theorem A]{prvi} that a germ $f$ analytic on a standard quadratic domain and satisfying uniform estimate \eqref{eq:fuin} admits petalwise dynamics and the existence of petalwise analytic Fatou coordinates along the standard quadratic domain, as described in Theorem~A in \cite{prvi} and recalled here in Figure~\ref{fig:dyn}. The same can be deduced for standard linear domains.}, up to identifications \eqref{eq:ekvij}.
\end{theo1}

\begin{theo2}[Realization by parabolic generalized Dulac germs on a standard linear domain]\label{real:scd}
Let $\rho\in\mathbb R$, $m\in\mathbb Z$. Let $(h_{0}^j,h_{\infty}^{j};\ \sigma_j)_{j\in\mathbb Z}$ be a sequence of pairs of analytic germs from $\mathrm{Diff}(\mathbb C,0)$, symmetric with respect to $\mathbb R_+$ as in \eqref{eq:sim}, and with maximal radii of convergence $\sigma_j$ bounded from below by $$\sigma_j\geq K_1 e^{-Ke^{C|j|}}, \ j\in\mathbb Z,$$ for some $C,\ K,\ K_1>0$. Let the elements of the sequence on their respective domains of definition satisfy the \emph{uniform} bound \eqref{dru}. Then there exists a prenormalized parabolic generalized Dulac germ $$g(z)=z-z^2\boldsymbol\ell^m+\rho z^3\boldsymbol\ell^{2m+1}+o(z^3\boldsymbol\ell^{2m+1}),$$ analytic on a standard linear domain and satisfying \eqref{eq:fuin},  that realizes this sequence as its horn maps, up to identifications \eqref{eq:ekvij}. In particular, $g$ admits a unique generalized Dulac asymptotic expansion, as $z\to 0$.
\end{theo2}
Note that on a standard linear domain we realize any sequence of moduli by a prenormalized parabolic generalized Dulac germ belonging to any formal class $(2,m,\rho)$, $m\in\mathbb Z$, $\rho\in\mathbb R$.
\begin{obs} 
Note the difference between Theorem A and Theorem B. In Theorem A, we realize a sequence of pairs of diffeomorphisms as moduli of a parabolic diffeomorphism $f$ on a bigger (quadratic) domain, but we do not claim that $f$ admits the generalized Dulac asymptotic expansion. In Theorem B, the constructed parabolic diffeomorphism  $g$ realizing the moduli has the required asymptotic expansion, but is defined on a smaller (linear) domain. 

In the course of proof of Theorems~A and B in Sections~\ref{subsec:prvi}-\ref{subsec:treci}, it can be seen that the parabolic generalized Dulac germ $f$ constructed in Theorem~B is \emph{not just the restriction} to a linear domain $\mathcal R_{a,b}\subset \mathcal R_C$ of a germ $g$ constructed in Theorem~A for the same sequence of pairs of  horn maps on a bigger quadratic domain $\mathcal R_C$, see Remarks~\ref{rem:linqua} and \ref{rem:final}. Therefore, \emph{we have not proven that the parabolic generalized Dulac germ constructed on a linear domain and realizing the given sequence of pairs can be extended as an analytic germ to a standard quadratic domain}. As far as we know, nothing can be directly concluded about Gevrey nature and uniqueness of the asymptotic expansion after the first three terms of the germ constructed in Theorem~A on a standard quadratic domain and realizing the given sequence of pairs of horn maps, or of any other representative of the same analytic class on a standard quadratic domain. This prevents extending the realization result in the class of parabolic generalized Dulac germs from linear to a bigger, standard quadratic domain, which remains an open question. However, we can deduce the following Corollary~\ref{cor:zaklj}. 
\end{obs}

\begin{cory}\label{cor:zaklj} Let $(h_{0}^j,h_\infty^{j})_{j\in\mathbb Z}$ be a sequence of pairs of analytic diffeomorphisms, symmetric with respect to $\mathbb R_+$ as in \eqref{eq:sim}, and satisfying \eqref{dru}. Let $m\in\mathbb Z$ and $\rho\in\mathbb R$. Let $f(z)$ be the germ defined on a standard quadratic domain $\mathcal R_C$ of the form

$$
f(z)=z-z^2\boldsymbol\ell^m+\rho z^3\boldsymbol\ell^{2m+1}+o(z^3\boldsymbol\ell^{2m+1}),
$$
that by Theorem~A realizes the above sequence of pairs as its horn maps. Let moreover $g(z)$ be the parabolic generalized Dulac germ of the same form defined on a standard linear domain $\mathcal R_{a,b}\subset\mathcal R_C$ that by Theorem~B realizes the above sequence of pairs as its horn maps. Then, there exists an analytic diffeomorphism $\varphi(z)=z+o(z)$ on $\mathcal R_{a,b}$, such that $\varphi^{-1}\circ g\circ \varphi$ can be extended from $\mathcal R_{a,b}$ analytically to the germ $f$ on $\mathcal R_C$.
\end{cory}
\begin{proof} From the equality of horn maps of $f$ and $g$ on $\mathcal R_{a,b}\subset \mathcal R_C$, by the proof of \cite[Theorem B]{prvi} it follows that $f$ and $g$ are analytically conjugated on $\mathcal R_{a,b}$ by $\varphi(z)=z+o(z)$. The statement follows by uniqueness of analytic continuation from $\mathcal R_{a,b}$ to $\mathcal R_C$. 
\end{proof}
However, since $f$ is not in general parabolic generalized Dulac, we cannot deduce anything about the nature and uniqueness of the power-logarithmic asymptotic expansion of the conjugacy $\varphi$ from Corollary~\ref{cor:zaklj}.

\section{Realization of infinite cocycles on standard linear and standard quadratic domains}\label{subsec:prvi}\ 

In this section, we prove Propositions~\ref{lem:cocyc} and \ref{lem:cocyc1} which are \emph{realization propositions for exponentially small cocycles} on standard quadratic domains $\mathcal R_C\subset \mathcal R$, or standard linear domains $\mathcal R_{a,b}\subset \mathcal R$ respectively. Here, $\mathcal R$ is the Riemann surface of the logarithm. We adapt the construction from \cite{loray} for realization of a cocycle in $\mathbb C$, using \emph{Cauchy-Heine integrals}. Propositions~\ref{lem:cocyc} and \ref{lem:cocyc1} are prerequisites for proving Theorems~A and B. 

In Section~\ref{subsec:drugi}, we prove Theorem~A. Motivated by \cite{teyssier} and realization of analytic moduli for saddle-node vector fields, we find a (prenormalized) parabolic germ $f$ in any formal class $(2,m,\rho)$, $m\in\mathbb Z$, $\rho\in\mathbb R$, analytic on a standard quadratic domain, such that its differences of sectorial Fatou coordinates realize a given cocycle on intersections of its petals. We use Proposition~\ref{lem:cocyc} at each step of the iterative construction of the Fatou coordinate, starting the construction with the Fatou coordinate of the formal normal form and then improving the approximation at each step. Note that $f$ is just \emph{analytic on a standard quadratic domain}; we do not claim any asymptotic expansion in $\widehat{\mathcal L}(\mathbb R)$ of $f$ after the three initial terms.

In Section~\ref{subsec:treci}, we prove Theorem~B. Using Proposition~\ref{lem:cocyc1}, we prove that, if we perform the construction from Section~\ref{subsec:drugi} on a smaller \emph{standard linear domain}, we get that $f$ additionally admits a \emph{generalized Dulac asymptotic expansion}. In this proof, for standard quadratic domains instead of standard linear, a $\log$-Gevrey property of sufficient order on limit ordinal steps of the expansion does not seem to hold, as shown in Remark~\ref{rem:sq}. For standard quadratic domains there is a technical problem of \emph{too long} lines of integration in Cauchy - Heine integrals. This results in insufficient Gevrey-type estimates which prevent canonical summability on limit ordinal steps, and gives non-uniqueness of asymptotic expansion in $\widehat{\mathcal L}(\mathbb R)$. 

\medskip

Classically (see e.g. \cite{loray}), we say that a function $h$ defined and holomorphic on an open sector $V$ is \emph{exponentially flat of order $m>0$ at $0$ in $V$} if, for every subsector $V'\subset V$, there exist constants $C>0$ and $M>0$, such that \begin{equation}\label{eq:flaty}|h(z)|\leq Ce^{-\frac{M}{|z|^m}},\ z\in V'.\end{equation}

\begin{prop}[Realization of infinite cocycles on standard quadratic domains]\label{lem:cocyc} Let $V_{0}^j$ resp. $V_{\infty}^j$, $j\in\mathbb Z$, denote open petals of opening $\pi$ centered at directions $(4j-3)\frac{\pi}{2}$ resp. $(4j-1)\frac{\pi}{2}$, $j\in\mathbb Z,$ along a standard quadratic domain. That is, if we denote by $r_j$ the radii of $V_0^j$ and $V_\infty^j$ at their central directions, then there exist constants $C> 0,\ K>0$ such that:
\begin{equation}\label{radii}
r_j\geq Ke^{-C\sqrt{|j|}},\ j\in\mathbb Z.
\end{equation}
Let $V_j^{+}$ resp. $V_j^-$, $j\in\mathbb Z$, denote the open cover\footnote{It means that the standard quadratic domain is \emph{covered} by open petals as in Figure~\ref{fig:dyn}. The petals $V_0^j$ and $V_\infty^j$, $j\in\mathbb Z,$ are the intersection petals of pairs of consecutive petals.} of the standard quadratic domain by petals of opening $2\pi$ centered at directions $2j\pi$ resp. $(2j-1)\pi$, 
such that
\begin{equation}\label{eq:pres}
V_0^{j+1}=V_{j+1}^-\cap V_j^+,\ V_\infty^{j}=V_{j}^-\cap V_j^+,
\end{equation}
are their intersection petals.

\noindent Let $(G_0^j,G_\infty^j)_{j\in\mathbb Z}$ be pairs of holomorphic functions on $V_{0}^j$ and $V_\infty^j$, $j\in\mathbb Z$, not identically equal to zero and \emph{uniformly} flat of order $m>0$ at $0$. That is, for subsectors $U_{0}^j\subset V_{0}^j$ and $U_{\infty}^j\subset V_{\infty}^j$, centered at central lines of $V_{0}^j$ and $V_\infty^j$, and of \emph{uniform opening} in $j\in\mathbb Z$, there exist $C>0$ and $M>0$ independent of $j$, such that:
\begin{equation}\label{foot}
|G_{0,\infty}^j(z)|\leq Ce^{-\frac{M}{|z|^m}}, \ z\in U_{0,\infty}^j,\ j\in\mathbb Z.
\end{equation}
Then, there exist analytic functions $R_j^{\pm}(z)=o(1)$, as $z\to 0$, defined on petals $V_j^{\pm}$, $j\in\mathbb Z$, such that: 
\begin{align}\begin{split}\label{e:p}
&R_+^{j-1}(z)-R_-^j(z)=G_0^j(z),\ z\in V_0^j,\\
&R_-^j(z)-R_+^j(z)=G_\infty^j(z),\ z\in V_\infty^j,\ \ j\in\mathbb Z.\end{split}
\end{align}
Moreover, for subsectors $S_j^\pm\subset V_j^\pm$ centered at central lines of $V_j^\pm$ and of \emph{uniform opening} in $j$, there exists a uniform $($in $j)$ constant $C>0$ such that:
\begin{equation}\label{eq:ar}
|R_j^{\pm}(z)|\leq C|\boldsymbol\ell|,\ z\in S_j^{\pm},\ j\in\mathbb Z.
\end{equation}
Here, $\boldsymbol\ell:=-\frac{1}{\log z}.$
\end{prop}

\begin{prop}[Realization of infinite cocycles on standard linear domains]\label{lem:cocyc1} 
Let all assumptions and notations as in Proposition~\ref{lem:cocyc} hold, except that \eqref{radii} is replaced by
\begin{equation}\label{radii1}
r_j\geq Ke^{-C|j|},\ j\in\mathbb Z,\ C,\,K>0.
\end{equation}
Let $\{V_j^+,V_j^-\}_{j\in\mathbb Z}$ be an open cover of a \emph{standard linear domain} by petals of opening $2\pi$ centered at directions $2j\pi$ resp. $(2j-1)\pi$, and let $V_0^j$ and $V_\infty^j$ be the intersections of consecutive petals as in \eqref{eq:pres}, $j\in\mathbb Z$. 
Then there exist analytic functions $R_j^{\pm}(z)=o(1)$, as $z\to 0$, defined on petals $V_j^{\pm}$, $j\in\mathbb Z$, such that \eqref{e:p} and \eqref{eq:ar} holds. Moreover, if we put $\boldsymbol\ell:=-\frac{1}{\log z}$ and 
$$
\breve R_j^{\pm}(\boldsymbol\ell):=R_j^{\pm}(z),\ \boldsymbol\ell\in\boldsymbol\ell(V_j^{\pm}),\ j\in\mathbb Z,
$$
then there exists  $\widehat R(\boldsymbol\ell)\in \mathbb C[[\boldsymbol\ell]]$ the common $\log$-Gevrey asymptotic expansion of order $m$ of any $\breve R_j^{\pm}(\boldsymbol\ell)$,  $j\in\mathbb Z$, as $\boldsymbol\ell\to 0$ in $\boldsymbol\ell$-cusp $\boldsymbol\ell(V_j^{\pm})$.
\end{prop}

We will say that functions $(R_j^{\pm}(z))_{j\in\mathbb Z}$ or transseries $\widehat R(\boldsymbol\ell)\in \mathbb C[[\boldsymbol\ell]]$ constructed in Propositions~\ref{lem:cocyc} and \ref{lem:cocyc1} \emph{realize} the given cocycle $(G_0^j,G_\infty^j)_{j\in\mathbb Z}$ on a standard quadratic resp. standard linear domain.
\bigskip

We prove Propositions~\ref{lem:cocyc} and \ref{lem:cocyc1} simultaneously. The proof is based on the following Lemmas~\ref{prop:cah}-\ref{prop:asylin}. 

For simplicity, we work in the logarithmic chart $\zeta=-\log z$. Put: $$\tilde G_{0,\infty}^j(\zeta):=G_{0,\infty}^j\big(e^{-\zeta}\big),\ j\in\mathbb Z.$$ Then $\tilde G_{0,\infty}^j$ are defined and analytic on \emph{petals}\footnote{in the $\zeta$-variable: open sets tangential, as $\mathrm{Re}(\zeta)\to\infty,$ to horizontal strips of a given width, that corresponds to the opening of the petal in the $z$-variable.} in the logarithmic chart $\tilde V_{0,\infty}^j=-\log(V_{0,\infty}^j)$. The petals $\tilde V_{0,\infty}^j$ in the logarithmic chart are bisected by the lines ending at $\mathrm{Re}(\zeta)=\infty$: \begin{align}\begin{split}\label{eg:lint}& \mathcal C_0^j\ldots \Big[-\log r_j + i(4j-3)\frac{\pi}{2},+\infty+i(4j-3)\frac{\pi}{2}\Big],\\
& \mathcal C_\infty^j\ldots \Big[-\log r_j + i(4j-1)\frac{\pi}{2},+\infty+i(4j-1)\frac{\pi}{2}\Big],\end{split}
\end{align} corresponding to the central rays $\big[0,r_j e^{i(4j-3)\frac{\pi}{2}}\big]$ of $V_{0}^j$, i.e. $\big[0,r_j e^{i(4j-1)\frac{\pi}{2}}\big]$ of $V_\infty^j$ in the original $z$-chart. Note that \eqref{radii} gives:
\begin{equation}\label{eq:imam}
-\log r_j\leq C\sqrt{|j|},\ j\in\mathbb Z,
\end{equation}
for a standard quadratic domain from Proposition~\ref{lem:cocyc}, and \eqref{radii1} gives:
\begin{equation}\label{eq:imam1}
-\log r_j\leq C|j|,\ j\in\mathbb Z,
\end{equation}
for a standard linear domain from Proposition~\ref{lem:cocyc1}.
\medskip

 In $\zeta$-chart, \eqref{foot} becomes: for substrips $\tilde U_{0,\infty}^j\subset \tilde V_{0,\infty}^j$ bisected by $\mathcal C_{0,\infty}^j$ and of uniform opening in $j$, there exist $M,\ C>0$ such that \begin{equation}\label{eq:b}|\tilde G_{0,\infty}^j(\zeta)|\leq Ce^{-Me^{m\text{Re}(\zeta)}} , \ \zeta\in\tilde{U}_{0,\infty}^j,\ j\in\mathbb Z.\end{equation} That is, $\tilde G_{0,\infty}^j$, $j\in\mathbb Z$, are uniformly (in $j\in\mathbb Z$) \emph{superexponential} of order $m>0$, as $\text{Re}(\zeta)\to\infty$ in $\tilde V_{0,\infty}^j$. 
\bigskip


\begin{lem}[Cauchy-Heine integrals]\label{prop:cah} Let $\widetilde{\mathcal R}_{0,j}^+$ resp. $\widetilde{\mathcal R}_{\infty,j}^+$, $j\in\mathbb Z$, be the parts of the standard quadratic domain $\widetilde{\mathcal R}_C$ in the logarithmic chart containing $\tilde V_0^j$ resp. $\tilde V_\infty^j$ and all points of the domain above $\tilde V_0^j$ resp. $\tilde V_\infty^j$. Equivalently, let $\widetilde{\mathcal R}_{0,j}^-$ resp. $\widetilde{\mathcal R}_{\infty,j}^-$ be the parts containing $\tilde V_0^j$ resp. $\tilde V_\infty^j$ and all points of the domain below them, see Figure~\ref{fig:indi}. Let $(\tilde G_0^j,\tilde G_\infty^j)_{j\in\mathbb Z}$ defined on $(\tilde V_0^j,\tilde V_\infty^j)_{j\in\mathbb Z}$, be an infinite cocycle, uniformly\footnote{The statement of this lemma holds even without the existence of the uniform constant in $j\in\mathbb Z$ in bound \eqref{foot}.} flat of order $m>0$, as in \eqref{foot}. 
\begin{enumerate}
\item Let the functions $\tilde F_{0,j}^{\pm}$ and $\tilde F_{\infty,j}^{\pm}$, $j\in\mathbb Z$, be defined as the \emph{Cauchy-Heine integrals} of $\tilde G_{0,j},\ \tilde G_{\infty,j}$ along lines $\mathcal C_{0,\infty}^j$:
\begin{align}\begin{split}\label{eq:ch}                                                                                 
&\tilde F_{0,j}^\pm(\zeta):=\frac{1}{2\pi i}\int_{\mathcal C_0^j}\frac{\tilde G_0^j(w)}{w-\zeta} dw=\frac{1}{2\pi i}\int_{-\log r_j+i(4j-3)\frac{\pi}{2}}^{+\infty+i(4j-3)\frac{\pi}{2}}\frac{\tilde G_0^j(w)}{w-\zeta} dw,\\
&\tilde F_{\infty,j}^\pm(\zeta):=\frac{1}{2\pi i}\int_{\mathcal C_\infty^j}\frac{\tilde G_\infty^j(w)}{w-\zeta} dw=\frac{1}{2\pi i}\int_{-\log r_j+i(4j-1)\frac{\pi}{2}}^{+\infty+i(4j-1)\frac{\pi}{2}}\frac{\tilde G_\infty^j(w)}{w-\zeta} dw.\end{split}
\end{align}
They are well-defined and analytic on the standard quadratic domain $\widetilde{\mathcal R}_C$ strictly above $(+)$ resp. below $(-)$ the integration line. 
\smallskip

\item By varying the integration paths inside the petals $\tilde V_{0,\infty}^j$, $\tilde F_{0,j}^\pm$ resp. $\tilde F_{\infty,j}^\pm$ may be extended analytically to the whole domains $\widetilde{\mathcal R}_{0,j}^{\pm}$ resp. $\widetilde{\mathcal R}_{\infty,j}^{\pm}$. 
\smallskip

\item It holds that:
\begin{align}\begin{split}\label{eq:ahhh}
&\tilde F_{0,j}^+(\zeta)-\tilde F_{0,j}^-(\zeta)=\tilde G_{0}^j(\zeta), \ \zeta\in\tilde V_{0}^j,\\
&\tilde F_{\infty,j}^+(\zeta)-\tilde F_{\infty,j}^-(\zeta)=\tilde G_{\infty}^j(\zeta), \ \zeta\in\tilde V_{\infty}^j.\end{split}
\end{align}
\end{enumerate}
\end{lem}

\begin{figure}[h!]
\includegraphics[scale=0.2]{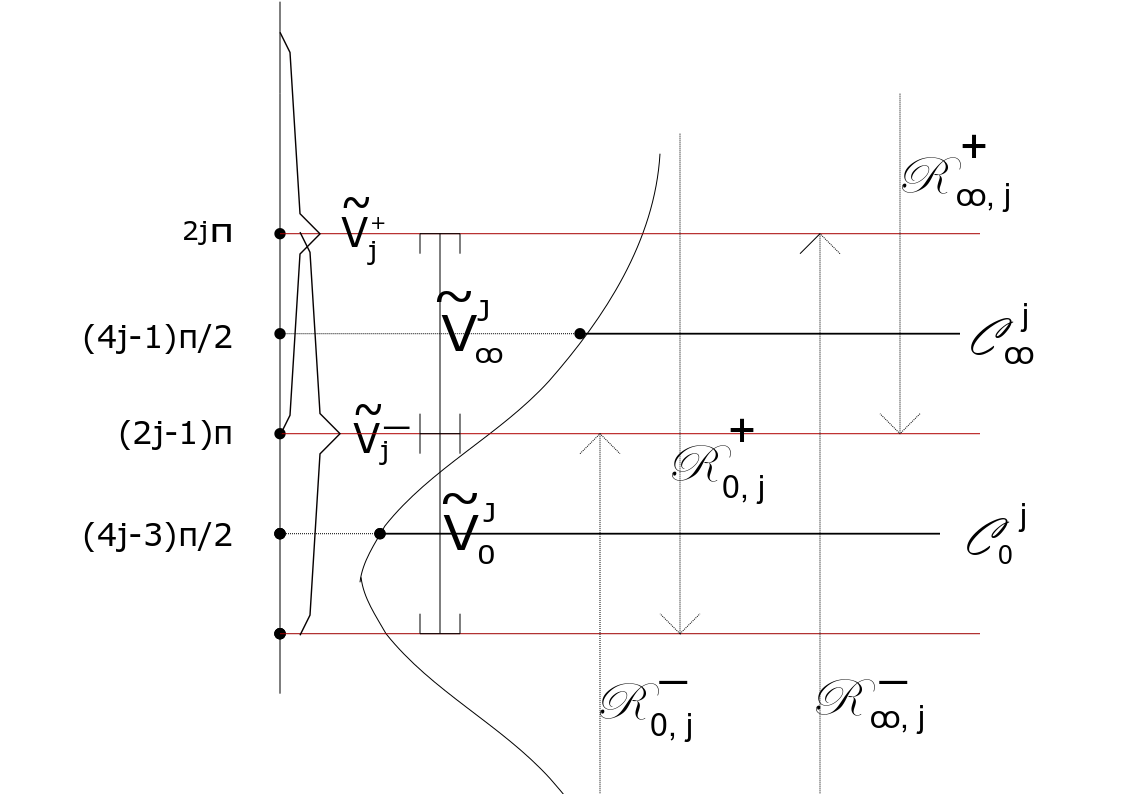}
\caption{The outline of position of petals $\tilde V_j^{\pm}$ and $\tilde V_{0,\infty}^j$, $j\in\mathbb Z$, on a standard quadratic domain $\widetilde{\mathcal R}_C$ in the logarithmic chart.}\label{fig:indi}
\end{figure}

\begin{proof}
We use Cauchy-Heine's construction based on the classical Cauchy's residue theorem. For more details on Cauchy-Heine construction in $\mathbb C$ that we adapt here for standard quadratic (linear) domains, see e.g. \cite{Mloday} or \cite{loray2}.
\smallskip

$(1)$ Obvious.
\smallskip

$(2)$ Suppose that we wish to extend $\tilde F_{0,j}^-$ above the central line $\mathcal C_0^j$ of petal $\tilde V_0^j$. We replace the integration path $\mathcal C_0^j$ in the Cauchy-Heine integral by the union of a horizontal line $\big(\mathcal C_0^j\big)'$ above $\mathcal C_0^j$ in $\tilde V_0^j$ and the portion of the boundary of the petal $\tilde V_0^j$ between the two lines, denoted by $\mathcal S_0^j$, see Figure~\ref{fig:indi} and Figure~\ref{fig:CH}. Here, $\big(\mathcal C_0^j\big)'$ is a horizontal line at some height $\theta\in\big((4j-3)\frac{\pi}{2},(2j-1)\pi\big)$ in the standard quadratic domain in the $\zeta$-variable. It corresponds, in the $z$-variable, to the ray at angle $\theta$ inside the petal $V_0^j$.  For simplicity, we are notationally imprecise, as we do not stress the dependence of $\big(\mathcal C_0^j\big)'$ and $\mathcal S_0^j$ on the height $\theta$.  Let $\gamma_{\theta}:=(\mathcal C_0^j)'\cup \mathcal S_0^j$ be this new integration path. Then, for any $\zeta$ below $\mathcal C_0^j$, the Cauchy-Heine integral along $\gamma_\theta$ is, by the Cauchy's integral theorem, equal to $\tilde F_{0,j}^-$. That is, for $\zeta\in\widetilde{\mathcal R}_C$ below $\mathcal C_0^j$, we get:
\begin{align*}
\tilde F_{0,j}^-(\zeta):=\frac{1}{2\pi i}\int_{\mathcal C_{0}^j}\frac{\tilde G_0^j(w)}{w-\zeta} dw&=\frac{1}{2\pi i}\int_{\gamma_\theta}\frac{\tilde G_0^j(w)}{w-\zeta} d\zeta\\
&=\frac{1}{2\pi i}\int_{(\mathcal C_0^j)'}\frac{\tilde G_0^j(w)}{w-\zeta} dw+\frac{1}{2\pi i}\int_{\mathcal S_0^j} \frac{\tilde G_0^j(w)}{w-\zeta} dw,\nonumber
\end{align*}
see Figure~\ref{fig:CH}.
\begin{figure}[h!]
\includegraphics[scale=0.1]{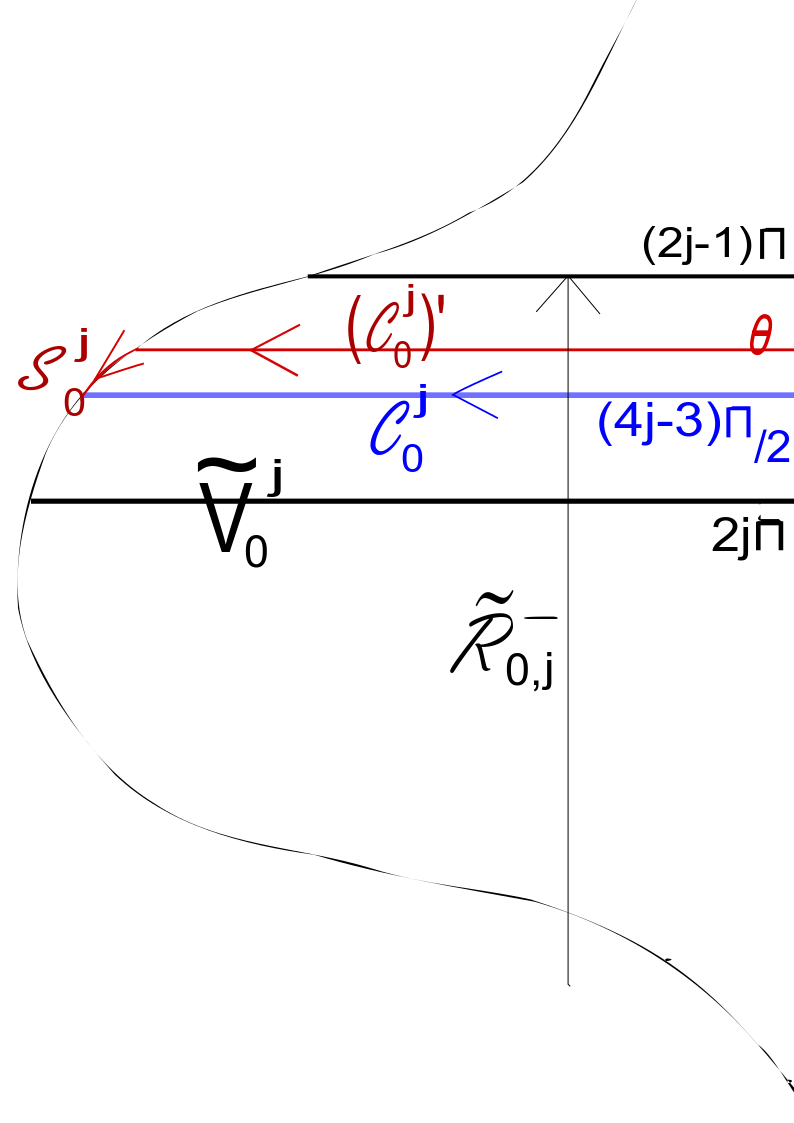}
\caption{The change of integration path in the $\zeta$-variable and Cauchy's integral theorem in the proof of Lemma~\ref{prop:cah} $(2)$.}\label{fig:CH}
\end{figure}

\noindent Therefore, the new integral $\int_{\gamma_\theta}\frac{\tilde G_0^j(w)}{w-\zeta} dw$ along $\gamma_\theta$ is the analytic extension of $\tilde F_{0,j}^-$ up to the line $(\mathcal C_0^j)'$. By varying the line $(\mathcal C_0^j)'$ above the central line $\mathcal C_0^j$ inside the petal $\tilde V_0^j$, we get the desired analytic extension up to the line at height $(2j-1)\pi$. In this way, $\tilde F_{0,j}^-(\zeta)$ given by formula \eqref{eq:ch} can be extended analytically to whole $\widetilde{\mathcal R}_{0,j}^-$. The same can be done for $\tilde F_{0,j}^+(\zeta)$ on $\widetilde{\mathcal R}_{0,j}^+$ and for $\tilde F_{\infty,j}^\pm(\zeta)$ on $\widetilde{\mathcal R}_{\infty,j}^\pm$, $j\in\mathbb Z$.

If we now denote by
$$
\tilde \chi_0^j(\zeta):=\frac{1}{2\pi i}\int_{\mathcal S_0^j} \frac{\tilde G_0^j(w)}{w-\zeta} dw,
$$
we notice that $\tilde \chi_0^j(\zeta)$ is an analytic germ at $\zeta=\infty$ (in the sense that $\xi\mapsto \tilde \chi_0^j(\frac{1}{\xi})$ is analytic at $\xi=0$), that is, that there exists $M_j>0$ such that $\tilde \chi_0^j(\zeta)$ is analytic for $\zeta\in\mathbb C,\ |\zeta|>M_j$. Consequently, it admits a Taylor asymptotic expansion in $\zeta^{-1}$, as $|\zeta|\to\infty$. This will be important for later proofs. 

We stress once again that here $(\mathcal C_0^j)'$ and $\mathcal S_0^j$, and therefore also $\tilde \chi_0^j(\zeta)$ and $M_j$, \emph{depend on the height $\theta$} of the line $(\mathcal C_0^j)'$ up to which we extend. They are not dependent only on  the petal, but also on the height in the petal up to which we extend. Here and in the sequel, we omit this dependence in the notation for simplicity.

\smallskip

(3) Since $\tilde V_{0}^j=\widetilde{\mathcal R}_{0,j}^+\cap \widetilde{\mathcal R}_{0,j}^{-},\ \tilde V_{\infty}^j=\widetilde{\mathcal R}_{\infty,j}^+\cap \widetilde{\mathcal R}_{\infty,j}^{-}$, \eqref{eq:ahhh} follows directly by the residue theorem after analytic extensions of $\tilde F_{0,\infty}^{\pm}(\zeta)$ to $\widetilde{\mathcal R}_{0,\infty,j}^\pm$ described in $(2)$. To illustrate, let us prove the first line of \eqref{eq:ahhh}. Take any $\zeta\in\tilde V_0^j$. Take any two lines inside petal $\tilde V_0^j$ such that $\zeta$ is strictly between them. Denote them by $\mathcal C_{\theta_1}$ and $\mathcal C_{\theta_2}$, at heights $\theta_1>\theta_2$. Now, by part $(2)$, we have:
\begin{align*}
\tilde F_{0,j}^+(\zeta)&=\frac{1}{2\pi i}\int_{\mathcal C_{\theta_2}}\frac{\tilde G_0^j(w)}{w-\zeta} dw+\frac{1}{2\pi i}\int_{\mathcal S_{\theta_2}}\frac{\tilde G_0^j(w)}{w-\zeta} dw,\\
\tilde F_{0,j}^-(\zeta)&=\frac{1}{2\pi i}\int_{\mathcal C_{\theta_1}}\frac{\tilde G_0^j(w)}{w-\zeta} dw+\frac{1}{2\pi i}\int_{\mathcal S_{\theta_1}}\frac{\tilde G_0^j(w)}{w-\zeta} dw,
\end{align*}
where $\mathcal S_{\theta_1}$ resp. $\mathcal S_{\theta_2}$ are the portions of the boundary of $\tilde V_0^j$ between the lines $\mathcal C_0^j$ and $\mathcal C_{\theta_1}$ and between $\mathcal C_0^j$ and $\mathcal C_{\theta_2}$ respectively.
Subtracting $\tilde F_{0,j}^+(\zeta)-\tilde F_{0,j}^-(\zeta)$, the statement follows by the residue theorem. See Figure~\ref{fig:residue}.

\begin{figure}[h!]
\includegraphics[scale=0.1]{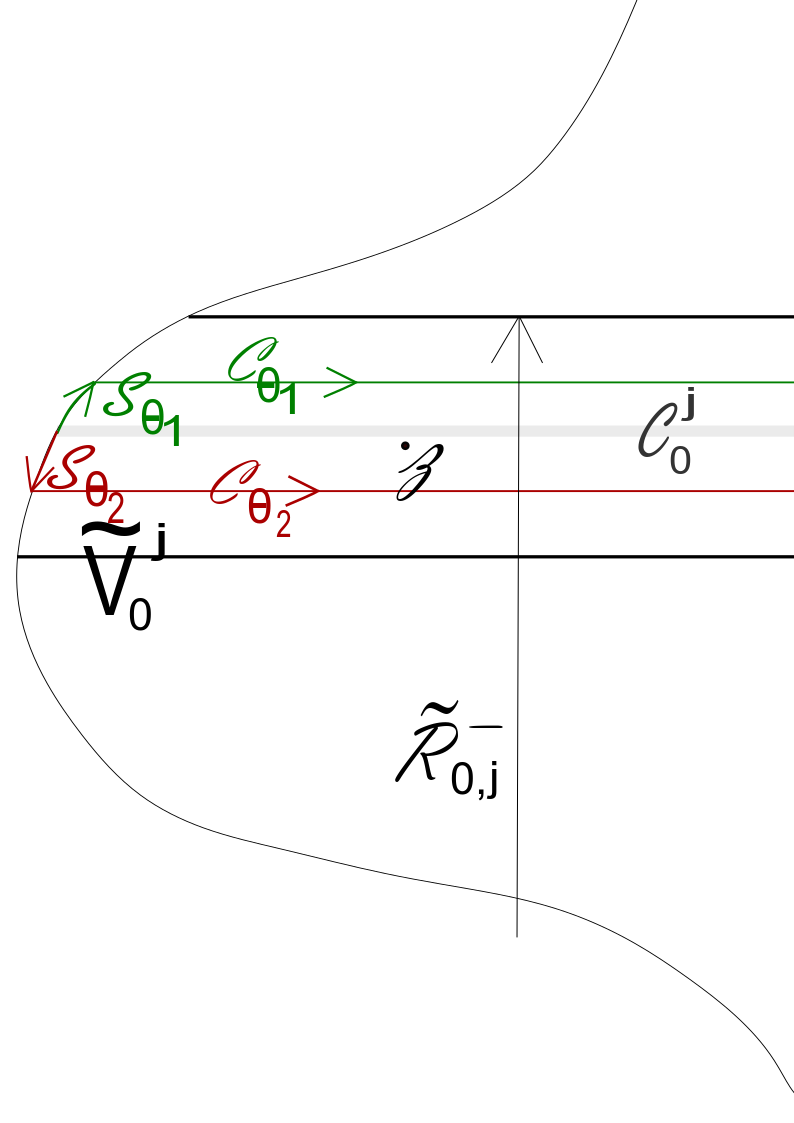}
\caption{The residue theorem in the proof of Lemma~\ref{prop:cah} $(3)$.}\label{fig:residue}
\end{figure}
\end{proof}

\begin{lem}\label{prop:conv} Let $(\tilde G_0^j,\tilde G_\infty^j)_{j\in\mathbb Z}$ be an infinite cocycle as described in Propositions~\ref{lem:cocyc} or \ref{lem:cocyc1}. Let $\tilde F_{0,j}^{\pm},\ \tilde F_{\infty,j}^{\pm}$ and their corresponding domains $\widetilde{\mathcal R}_{0,j}^{\pm},\ \widetilde{\mathcal R}_{\infty,j}^\pm$ be as defined in Lemma~\ref{prop:cah}. Let
\begin{align}\label{eq:FJ}
&\tilde R_j^+:=\Big(\big(\sum_{k=-\infty}^{j}\tilde F_{0,k}^+ + \sum_{k=-\infty}^{j}\tilde F_{\infty,k}^+\big) +\big(\sum_{k=j+1}^{+\infty}\tilde F_{0,k}^- + \sum_{k=j+1}^{+\infty}\tilde F_{\infty,k}^-\big)\Big) \Big|_{\tilde V_j^+}\Big. , \ j \in\mathbb Z,\\
&\tilde R_j^-:=\Big(\big(\sum_{k=-\infty}^{j}\tilde F_{0,k}^+ + \sum_{k=-\infty}^{j-1}\tilde F_{\infty,k}^+\big) +\big(\sum_{k=j+1}^{+\infty}\tilde F_{0,k}^- + \sum_{k=j}^{+\infty}\tilde F_{\infty,k}^-\big)\Big) \Big|_{\tilde V_j^-}\Big. , \ j \in\mathbb Z.\nonumber
\end{align}
Then $\tilde R_j^{\pm}$ are well-defined analytic functions on petals $\tilde V_j^{\pm}$, $j\in\mathbb Z$.
\smallskip

\noindent Moreover, the functions $\tilde R_j^{\pm}$ \emph{realize} the cocycle $(\tilde G_0^j,\tilde G_\infty^j)_{j\in\mathbb Z}$:
\begin{align}\begin{split}\label{eq:dobi}&\tilde R_{j-1}^+(\zeta)-\tilde R_j^-(\zeta)=\tilde G_0^j(\zeta),\ \zeta\in \tilde V_0^j,\\
&\tilde R_j^-(\zeta)-\tilde R_j^+(\zeta)=\tilde G_\infty^j(\zeta),\ \zeta\in \tilde V_\infty^j,\ \ j\in\mathbb Z.
\end{split}\end{align}
\end{lem}
As shown in Figure~\ref{fig:dom} below, to get functions $\tilde R_j^\pm$ defined by \eqref{eq:FJ} on $\tilde V_j^{\pm}$, on corresponding petal (strip) $\tilde V_j^{\pm}$ we sum all functions $\tilde F_{0,k}^{\pm}$, $\tilde F_{\infty,k}^{\pm}$, $k\in\mathbb Z$, from \eqref{eq:ch} which are well-defined on $\tilde V_j^{\pm}$.
\smallskip

The proof of Lemma~\ref{prop:conv} is in the Appendix. We prove that, for every $j\in\mathbb Z$, the series in \eqref{eq:FJ} converges uniformly on compacts in $\tilde V_j^\pm$, thus defining analytic functions $\tilde R_j^\pm$ on $\tilde V_j^\pm$ by the Weierstrass theorem.
\vspace{-0.2cm}

\begin{figure}[h!]
\includegraphics[scale=0.17]{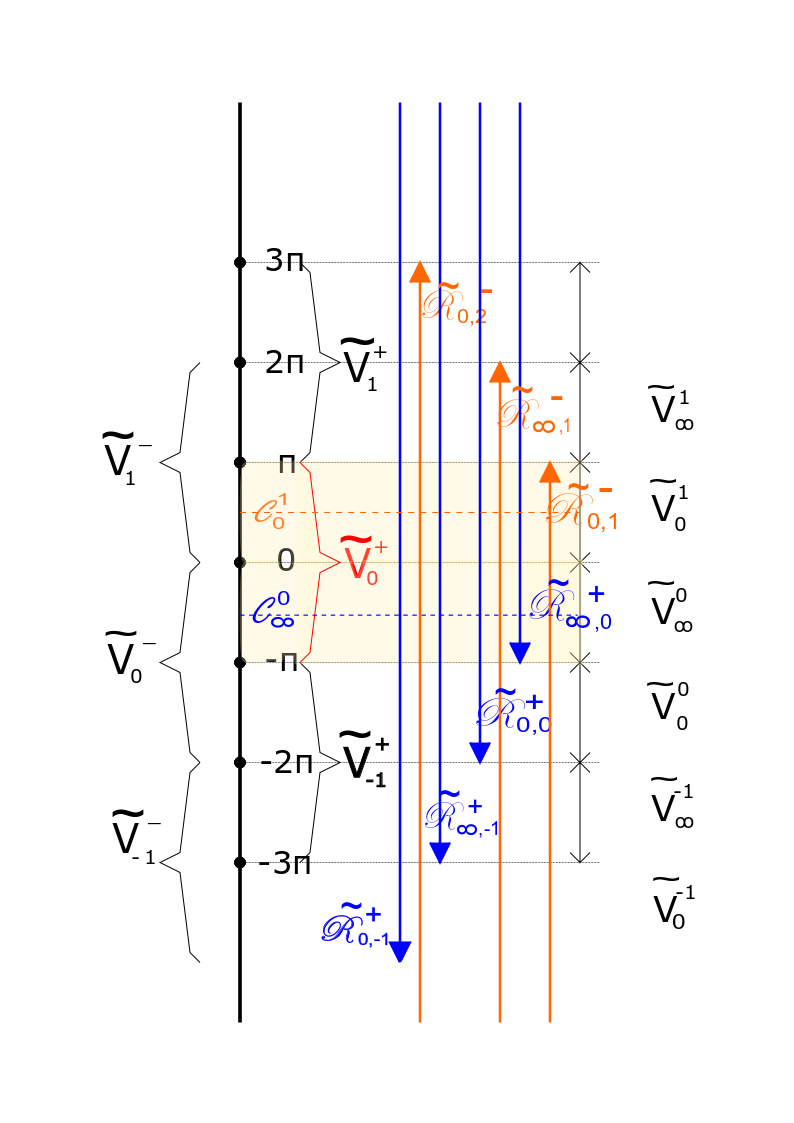}
\vspace{-0.4cm}
\caption{Ilustration of formula \eqref{eq:FJ} for $\tilde R_0^+$ on $\tilde V_0^+$. The figure illustrates in which domains $\widetilde{\mathcal R}_{0/\infty,j}^\pm$ the petal $\tilde V_0^+$ is fully contained. To get $\tilde R_0^+$, we sum the corresponding functions $\tilde F_{0/\infty,j}^\pm$ from \eqref{eq:ch}.}\label{fig:dom}
\end{figure}

\bigskip

We prove in Lemma~\ref{prop:asylin} below the asymptotics for $\tilde R_j^{\pm}$ constructed on $\tilde V_j^{\pm}$ in Lemma~\ref{prop:conv}. It holds that $\tilde R_j^{\pm}(\zeta)=o(1)$, as $\text{Re}(\zeta)\to\infty$ in $\tilde V_j^{\pm}$, moreover \emph{uniformly in $j\in\mathbb Z$}. Also, for standard
linear domains we show additionaly the complete \emph{$\log$-Gevrey asymptotic expansion} of $\tilde R_j^{\pm}(\zeta)$ in $\mathbb C[[\zeta^{-1}]]$, as $\text{Re}(\zeta)\to\infty$ on $\tilde V_j^{\pm}$. 

\begin{lem}[$\log$-Gevrey asymptotic expansion of $\tilde R_j^{\pm}(\zeta),\ j\in\mathbb Z$]\label{prop:asylin} Let $\tilde R_j^{\pm},\ j\in\mathbb Z$, be constructed as in Lemma~\ref{prop:conv} on petals $\tilde V_j^{\pm}$ on  a standard quadratic or a standard linear domain. Then:
\begin{enumerate}
\item On both domains $($standard linear and standard quadratic$)$, there exist subdomains $($linear resp. quadratic$)$ $\widetilde{\mathcal R}_{C'}\subset\widetilde{\mathcal R}_C$ such that, for substrips $\tilde U_j\subset\tilde V_j^{\pm}\cap \widetilde{\mathcal R}_{C'}$ centered at center lines of $\tilde V_j^{\pm}$ and of width $0<\theta<2\pi$ independent of $j\in\mathbb Z$, there exists a \emph{uniform} in $j\in\mathbb Z$ constant $C_\theta>0$ such that: 
$$|\tilde R_j^{\pm}(\zeta)|\leq C_\theta|\zeta|^{-1},\ \zeta\in\tilde U_j.$$ 

\item If $\tilde R_j^\pm$ are constructed on a standard linear domain, then there exists a formal series $\widehat R\in\mathbb C[[\zeta^{-1}]]$, such that any $\tilde R_j^\pm(\zeta),\ j\in\mathbb Z,$ admits $\widehat R$ as its $\log$-Gevrey asymptotic expansion of order $m$, as $\text{Re}(\zeta)\to+\infty$ in $\tilde V_j^{\pm}$. Here, $m>0$ is given in \eqref{foot}.
\end{enumerate}
\end{lem}
\noindent The proof is in the Appendix. Also, in Remark~\ref{rem:sq} in the Appendix we show a technical obstacle for proving statement $(2)$ on a standard quadratic domain.
\bigskip

\noindent \emph{Proof of Propositions~\ref{lem:cocyc} and \ref{lem:cocyc1}.} 
Let $\tilde R_j^{\pm}$ be as constructed in Lemma~\ref{prop:conv} on petals $\tilde V_j^{\pm}$ in the $\zeta$-variable, $j\in\mathbb Z$, either on a standard quadratic or a standard linear domain. Returning to the variable $z=e^{-\zeta}$, we put:
$$
R_j^{\pm}(z):=\tilde R_j^{\pm}(\zeta),\ z\in V_j^{\pm},\ j\in\mathbb Z.
$$
By Lemma~\ref{prop:conv}, $R_j^{\pm}(z)$ are analytic on $V_j^\pm$ and we have:
\begin{align}
\begin{split}\label{eq:es}
&R_+^{j-1}(z)-R_-^j(z)=G_0^j(z),\ z\in V_0^j, \\
&R_-^j(z)-R_+^j(z)=G_\infty^j(z),\ z\in V_\infty^j,\ \ j\in\mathbb Z.
\end{split}
\end{align}
Moreover, putting $\boldsymbol\ell:=\zeta^{-1}$, from Lemma~\ref{prop:asylin} we get that the functions $\breve R_j^{\pm}(\boldsymbol\ell):=R_j^{\pm}(z)$ constructed on a standard linear domain on $\boldsymbol\ell$-cusps $\boldsymbol\ell(V_j^{\pm})$, $j\in\mathbb Z$, admit a $\log$-Gevrey power asymptotic expansion of order $m$. By exponentially small differences \eqref{eq:es} on intersections of petals, we get that all $\breve R_j^{\pm}(\boldsymbol\ell)$ admit a common $\widehat R(\boldsymbol\ell)\in \mathbb{C}[[\boldsymbol\ell]]$ as their $\log$-Gevrey asymptotic expansion of order $m$. The uniform bound \eqref{eq:ar} for both domains (linear and quadratic) follows by statement $(1)$ of Lemma~\ref{prop:asylin}. Thus Propositions~\ref{lem:cocyc} and \ref{lem:cocyc1} are proven. 
\hfill $\Box$

\begin{obs}\label{rem:linqua} Observe that the functions $R_j^{\pm}(z)$ constructed in the proof of Proposition~\ref{lem:cocyc1} by Cauchy-Heine integrals on petals along standard linear domain are \emph{not petalwise restrictions} of $R_j^{\pm}(z)$ constructed along standard quadratic domain in the proof of Proposition~\ref{lem:cocyc}. 

Indeed, the lines of integration $\mathcal C_{0,\infty}^k$ are changed (asymptotically \emph{shorter} for standard linear domains). Therefore, we cannot claim that $R_j^{\pm}$ defined on petals of a standard linear domain can be analytically extended to petals of a standard quadratic domain. Therefore, we do not claim in Proposition~\ref{lem:cocyc} that there exist $\breve R_j^{\pm}(\boldsymbol\ell)$ defined on $\boldsymbol\ell$-images of petals of a larger standard quadratic domain which admit a $\log$-Gevrey asymptotic expansion, as $\boldsymbol\ell\to 0$. 
\end{obs}

\medskip

\section{Proof of Theorem~A}\label{subsec:drugi}

\noindent The proof is very involved, so we first give an outline of the proof. We then state necessary lemmas, and prove Theorem~A at the end of the section.

\smallskip

\noindent \emph{Outline of the proof of Theorem~A.} Let $(h_{0}^j,h_{\infty}^j; \sigma_j)_{j\in\mathbb Z}$ be a symmetric sequence \eqref{eq:sim} of analytic germs of diffeomorphisms from $\mathrm{Diff}(\mathbb C,0)$, satisfying the uniform bound \eqref{dru}. Let $\rho\in\mathbb R$ and $m\in\mathbb Z$. Here we construct a parabolic germ $f$, defined on a standard quadratic domain, of the prenormalized form 
$$
f(z)=z-z^2\boldsymbol\ell^m+\rho z^3\boldsymbol\ell^{2m+1}+o(z^3\boldsymbol\ell^{2m+1}),\ z\in\mathcal R_C,
$$
whose sectorial Fatou coordinates realize the given sequence as its horn maps. Let $V_j^+$ resp. $V_j^{-}$, $j\in\mathbb Z$, be petals covering a standard quadratic domain of opening $2\pi$, centered at $2j\pi$ resp. $(2j-1)\pi$, and let $V_0^j:=V_{j-1}^+\cap V_j^-$, $V_\infty^j:=V_j^+\cap V_j^-$ be their intersecting petals of opening $\pi$, as shown in Figure~\ref{fig:dyn}. We construct $f$ by constructing its sectorial Fatou coordinates $\Psi_{\pm}^j$ on $V_j^\pm,\ j\in\mathbb Z$, in an iterative construction described below, which satisfy:
\begin{align}\begin{split}\label{eq:fatouu}
&\Psi_{j-1}^+(z)-\Psi_j^-(z)=g_0^j(e^{-2\pi i \Psi_{j-1}^+(z)}),\ z\in V_0^j, \\
&\Psi_j^-(z)-\Psi_j^+(z)=g_\infty^j(e^{2\pi i\Psi_j^+(z)}),\ z\in V_\infty^j,\ \ j\in\mathbb Z.\end{split}
\end{align}
Here, $g_{0}^j,\ g_\infty^j,\ j\in\mathbb Z$, are analytic germs at $t\approx 0$, related to given $h_0^j,\ h_\infty^j$, $j\in\mathbb Z$, by: \begin{equation}\label{eq:relationg}(h_0^j)^{-1}(t)=te^{2\pi i g_0^j(t)},\ h_\infty^j(t)=te^{2\pi i g_\infty^j(t)},\ t\approx 0.\end{equation}
Then, due to \eqref{eq:fatouu} and \eqref{eq:relationg}, $f$ realizes the sequence of pairs of diffeomorphisms $(h_0^j,h_\infty^j)_{j\in\mathbb Z}$ as its horn maps. Indeed, \eqref{eq:fatouu} is an equivalent formulation of this statement, see Subsection~\ref{subsec:mt} for more details.

\smallskip

The idea of \emph{successive approximations} is taken from \cite{teyssier} for realizing the moduli of analytic classification for saddle-node vector fields. We will use the cocycle realization Proposition~\ref{lem:cocyc} and, by Lemma~\ref{lema:konver} (1), iteratively realize the cocycles $({}^n G_0^j,{}^n G_{\infty}^j)_{j\in\mathbb Z}$, $n\in\mathbb N_0$, where
\begin{align*}
{}^n  G_0^j(z):&=g_0^j(e^{-2\pi i \Psi_{j-1,+}^n(z)}),\ z\in V_0^j,\\
{}^n G_\infty^j(z):&=g_\infty^j(e^{2\pi i \Psi_{j,+}^n(z)}),\ z\in V_\infty^j.
\end{align*}

\noindent Here, $(\Psi_{j,\pm}^n)_{n\in\mathbb N_0}$ on $V_{\pm}^j$ are \emph{successive approximations} of the final Fatou coordinate $\Psi_j^\pm$, $j\in\mathbb Z$, starting with the Fatou coordinate of the $(2,m,\rho)$-formal normal form $\Psi_{j,\pm}^0:=\Psi_\mathrm{nf}$ on $V_{\pm}^j$. More precisely, we construct them as follows:
$$
\Psi_{j,\pm}^n(z):=\Psi_\mathrm{nf}(z)+R_{j,\pm}^n(z),\ z\in V_{\pm}^j,\ n\in\mathbb N,
$$
where 
\begin{align}\begin{split}\label{eq:lim}
&R_{j,\pm}^0(z):=0,\ z\in V_{\pm}^j, \text{ and}\\
&R_{j-1,+}^n(z)-R_{j,-}^n(z)=g_{0}^j(e^{-2\pi i \Psi_{j-1,+}^{n-1}(z)}):={}^{n-1} G_{0}^j(z),\ z\in V_{0}^j,\\
&R_{j,-}^n(z)-R_{j,+}^n(z)=g_{\infty}^j(e^{2\pi i \Psi_{j,+}^{n-1}(z)}):={}^{n-1} G_{\infty}^j(z),\ z\in V_{\infty}^j,\ n\in\mathbb N.\end{split}
\end{align}
At each step $n$, the functions $R_{j,\pm}^n(z)=o(1)$, $z\to 0$, are obtained using Proposition~\ref{lem:cocyc} for the realization of the previous cocycle $({}^{n-1} G_0^j,{}^{n-1} G_\infty^j)_{j\in\mathbb Z}$. The cocycle itself is obtained by applying $g_0^j,\ g_\infty^j$ to the exponentials of the Fatou coordinates from the previous step. In this manner, we make \emph{corrections} of the Fatou coordinate at each step, starting from the natural initial choice $\Psi_\mathrm{nf}$, the Fatou coordinate of the fomal normal form.
\smallskip

We then prove, in Lemma~\ref{lema:konver} (2), the uniform convergence of the Fatou coordinates $\Psi_{j,\pm}^n$ (that is, of $R_{j,\pm}^n$), as $n\to\infty$, on compact subsectors of petals $V_{\pm}^j$. Thus, as limits, we get analytic Fatou coordinates, which we denote by $\Psi_j^\pm:=\Psi_\mathrm{nf}+R_j^{\pm}$, on petals $V_{j}^\pm$. By taking pointwise limit, as $n\to\infty$, to \eqref{eq:lim}, we get that $\Psi_j^{\pm}$ satisfy \eqref{eq:fatouu} and thus \emph{realize} the given sequence of pairs of horn maps $(h_0^j,h_\infty^j)_{j\in\mathbb Z}$. 
\medskip

Finally, we recover the germ $f$ from its sectorial Fatou coordinates, using the Abel equation. On each petal, $f(z):=(\Psi_j^\pm)^{-1}\big(1+\Psi_j^\pm\big(z)\big),\ z\in V_j^{\pm}$. We show that $f$ \emph{glues} to an analytic function on a standard quadratic domain. It is of the prenormalized form \eqref{eq:forma} due to the form of $\Psi_j^{\pm}:=\Psi_{\mathrm{nf}}+R_j^{\pm},\ R_j^{\pm}=o(1)$, as $z\to 0$ on $V_j^{\pm}$, and Proposition~\ref{prop:pgd} in the Appendix. The uniform bound \eqref{eq:fuin} is proven by Lemma~\ref{lema:konver} (3). To prove Lemma~\ref{lema:konver} (3), we prove that the uniform bound \eqref{eq:ar} from Proposition~\ref{lem:cocyc} holds with the same constant for $R_{j,\pm}^n$ in each iterative step $n\in\mathbb N$.

We prove in Lemma~\ref{lem:simi} that symmetry of horn maps \eqref{eq:sim} implies that $\mathbb R_+$ is invariant by $f$. 
\bigskip

\subsection{The main lemmas}
\begin{lem}\label{lema:konver} Let $(h_{0}^j,h_{\infty}^j; \sigma_j)_{j\in\mathbb Z}$, where
$$
\sigma_j\geq K_1 e^{-Ke^{C\sqrt{|j|}}}, \ |j|\to\infty,\text{ for some $C,\ K,\ K_1>0$},
$$
be a symmetric sequence \eqref{eq:sim} of pairs of analytic germs from $\mathrm{Diff}(\mathbb C,0)$, satisfying the uniform bound \eqref{dru}. Let the sequence of pairs of analytic germs of diffeomorphisms $(g_0^j,g_\infty^j; \sigma_j)_{j\in\mathbb Z}$ be defined from $(h_0^j,h_\infty^j; \sigma_j)_{j\in\mathbb Z}$ by \eqref{eq:relationg}. Let $\rho\in\mathbb R$ and $m\in\mathbb Z$, and let $\Psi_\mathrm{nf}$ be the Fatou coordinate of the $(2,\rho,m)$-model $f_0$ from \eqref{eq:mod}. Let $\{V_j^{\pm}\}_{j\in\mathbb Z}$ be a collection of petals of opening $2\pi$, centered at $j\pi$, along a standard quadratic domain.

\begin{enumerate}
\item The following sequence of analytic maps $\Psi_{j,\pm}^n$, $n\in\mathbb N_0$, on petals $V_{j}^\pm$, is well-defined by the following iterative procedure:
$$
\Psi_{j,\pm}^n(z):=\Psi_\mathrm{nf}(z)+R_{j,\pm}^n(z),\ z\in V_{\pm}^j,\ n\in\mathbb N_0,
$$
where 
\begin{align}\begin{split}\label{e:pp}
&R_{j,\pm}^0(z):=0,\ z\in V_j^\pm, \text{ and}\\
&R_{j-1,+}^n(z)-R_{j,-}^n(z)=g_{0}^j(e^{-2\pi i \Psi_{j-1,+}^{n-1}(z)})=:{}^{n-1} G_{0}^j(z),\ z\in V_{0}^j,\\
&R_{j,-}^n(z)-R_{j,+}^n(z)=g_{\infty}^j(e^{2\pi i \Psi_{j,+}^{n-1}(z)})=:{}^{n-1} G_{\infty}^j(z),\ z\in V_{\infty}^j,\ n\in\mathbb N.\end{split}
\end{align}
Here, for every $n\in\mathbb N$, $\big(^{n-1} G_0^j(z),^{n-1} G_\infty^j(z)\big)_{j\in\mathbb Z}$ is an infinite cocycle satisfying all assumptions of Proposition~\ref{lem:cocyc}, and $R_{j,\pm}^n,\ n\in\mathbb N,$ are analytic germs on petals $V_j^{\pm}$, $j\in\mathbb Z$, that realize this cocycle, given by Proposition~\ref{lem:cocyc}.
\smallskip

\item For every $j\in\mathbb Z$, the sequence $(\Psi_{j,\pm}^n)_{n\in\mathbb N}$ converges uniformly on compact subsectors of $V_j^\pm$, thus defining analytic functions $\Psi_{j}^{\pm}$ on petals $V_j^{\pm}$ at the limit. Moreover, $\Psi_{j}^\pm$, $j\in\mathbb Z$, satisfy:
\begin{align}\begin{split}\label{eq:pt}
&\Psi_{j-1}^+(z)-\Psi_j^-(z)=g_0^j(e^{-2\pi i \Psi_{j-1}^+(z)}),\ z\in V_0^j, \\
&\Psi_j^-(z)-\Psi_j^+(z)=g_\infty^j(e^{2\pi i\Psi_j^+(z)}),\ z\in V_\infty^j,\ \ j\in\mathbb Z.\end{split}
\end{align}

\smallskip

\item For the petalwise limits $R_{j}^\pm$, $j\in\mathbb Z$, the following uniform bound holds. For every collection of subsectors $S_j\subset V_j^{\pm}$ centered at $j\pi$ and of opening strictly less than $2\pi$ independent of $j\in\mathbb Z$, there exists a uniform constant $C>0$ $($independent of $j)$, such that:
\begin{equation}\label{eq:a}
|R_{j}^\pm(z)|\leq C|\boldsymbol\ell|,\ z\in S_j,\ j\in\mathbb Z.
\end{equation}
\end{enumerate}
\end{lem}

For simplicity, in the proof of Lemma~\ref{lema:konver}, we pass to the logarithmic chart. We denote by $\tilde V_j^{\pm}$ the petals $V_j^\pm$ in the logarithmic chart. Let \begin{equation}\label{eq:chart}\tilde \Psi_{j,\pm}^n(\zeta):=\Psi_{j,\pm}^n(e^{-\zeta}),\ \tilde R_{j,\pm}^n(\zeta):=R_{j,\pm}^n(e^{-\zeta}),\ \zeta\in\tilde V_j^{\pm},\ j\in\mathbb Z,\ n\in\mathbb N.\end{equation}

In the proof of statement $(2)$ in Lemma~\ref{lema:konver}, we use the following auxiliary lemma, whose proof is in the Subsection~\ref{subsec:ap} of the Appendix. Due to a technical detail in the Cauchy-Heine construction (the presence of a logarithmic singularity at the border of the standard quadratic domain), we are unable to prove uniform convergence of $(\tilde R_{j,\pm}^n)_n$ on $\tilde V_{j}^{\pm}$, as $n\to\infty$. Instead, we prove uniform convergence of their exponentials on petals, which then implies uniform convergence \emph{on compact subsets} for the initial sequence.
\begin{lem}\label{lem:uses} Let the assumptions of Lemma~\ref{lema:konver} hold. Let $\tilde R_{j,\pm}^n$, $n\in\mathbb N$, be as defined in the statement $(1)$ of Lemma~\ref{lema:konver} $($in the logarithmic chart, see \eqref{eq:chart}$)$. The sequence $$\big(e^{2\pi i\tilde R_{j,\pm}^n}\big)_{n\in\mathbb N}$$ is a Cauchy sequence in the sup-norm on petal $\tilde V_j^\pm$, for every $j\in\mathbb Z$. 
\end{lem}
\medskip

\noindent \emph{Proof of Lemma~\ref{lema:konver}}.  

\emph{Proof of statement $(1)$}. We check that, in every step of the construction, all assumptions of Proposition~\ref{lem:cocyc} are satisfied. The basis of the induction is obvious by putting $\tilde R_{j,\pm}^0\equiv 0,\ \tilde \Psi_{j,\pm}^0:=\tilde \Psi_\mathrm{nf}$ on $\tilde V_{\pm}^j$. Suppose that $\tilde \Psi_{j,\pm}^k$ are constructed and analytic for $0\leq k<n$. By Remark~\ref{rem:use} in the Appendix and the uniform bound \eqref{prv} on $g_0^j$, we get that there exist constants $c>0$ and $C_1>0$ independent of $j\in\mathbb Z$, such that:
\begin{align}\label{eq:psii}
\big|{}^{n-1} \tilde G_0^j(\zeta)\big|=\Big|g_0^j\big(e^{-2\pi i \tilde\Psi_{j-1,+}^{n-1}(e^{-\zeta})}\big)\Big|\leq c\big|e^{-2\pi i \Psi_{j-1,+}^{n-1}(e^{-\zeta})}\big|&\leq C_1 \big|e^{-\pi i \tilde \Psi_\mathrm{nf}(\zeta)}\big|\\
&=C_1 e^{\pi \text{Im}(\tilde \Psi_\mathrm{nf}(\zeta))},\ \zeta\in \tilde V_0^j.\nonumber
\end{align}
Now, for every collection of central \emph{substrips}\footnote{Important for the bound \eqref{eq:bo} below, since, for every collection of substrips $\tilde U_j\subset \tilde V_0^j$ of width $0<\theta<2\pi$ independent of $j$, there exists a constant $c_\theta>0$ such that $-\mathrm{Im}(e^{-\zeta})>c_\theta\cdot |\mathrm{Re}(e^{-\zeta})|,\ \zeta\in\tilde U_j$.} $\tilde U_j\subset \tilde V_0^j$ of width independent of $j\in\mathbb Z$ and for every $\delta>0$, there exist constants $C,\,D>0$ independent of $j\in\mathbb Z$ such that:
\begin{equation}\label{eq:bo}\big|\mathrm{Im}(\tilde \Psi_\mathrm{nf}(\zeta))\big|=- \mathrm{Im}(\tilde \Psi_\mathrm{nf}(\zeta))\geq C|\tilde \Psi_\mathrm{nf}(\zeta)|\geq D e^{(1-\delta)\mathrm{Re}(\zeta)},\ \zeta\in\tilde U_j.\end{equation}
The last inequality is obtained using the exact form of $\Psi_{\mathrm{nf}}(z)$ given in \eqref{eq:abo} and the fact that, for a standard quadratic domain $\widetilde{\mathcal R}_C$, there exists $d>0$ such that $\mathrm{Im}(\zeta)\leq d\cdot \mathrm{Re}^2(\zeta)$, $\zeta\in\widetilde{\mathcal R}_C$.

\noindent Therefore, combining \eqref{eq:psii} and \eqref{eq:bo}, for a collection of substrips $\tilde U_j\subset \tilde V_0^j$ of a given width $0<\theta<2\pi$ independent of $j\in\mathbb Z$, there exist constants $C,\ M>0$ \emph{independent of $j\in\mathbb Z$ and of the step $n\in\mathbb N$}, such that:
\begin{equation}\label{eq:ge}
\big|{}^{n-1} \tilde G_0^j(\zeta)\big|\leq C e^{-Me^{(1-\delta)\mathrm{Re}(\zeta)}},\ \zeta\in\tilde U_j,\ j\in\mathbb Z,\ n\in\mathbb N.
\end{equation}
A similar analysis is done for ${}^{n-1}\tilde G_\infty^j(\zeta)$ on $\tilde V_\infty^j$, $j\in\mathbb Z$. 
Therefore, assumption~\eqref{foot} of Proposition~\ref{lem:cocyc} is satisfied in every step with $m=1-\delta$, for every $\delta>0$. The existence and analyticity of $\tilde R_{j,\pm}^n$ on $\tilde V_{j}^{\pm}$ then follows directly by Proposition~\ref{lem:cocyc}. Also, \eqref{e:pp} follows directly from \eqref{e:p} in Proposition~\ref{lem:cocyc}.
\smallskip

Precisely, for later use, by Lemma~\ref{prop:conv} in the proof of Proposition~\ref{lem:cocyc}, $\tilde R_{j,\pm}^n$ on petals $\tilde V_j^{\pm}$, $j\in\mathbb Z$, $n\in\mathbb N$, are given as the sum of the Cauchy-Heine integrals as follows:
\begin{align}\label{eq:FJi}
&\tilde R_{j,+}^n:=\Big(\big(\sum_{k=-\infty}^{j} {}^n \tilde F_{0,k}^{+} + \sum_{k=-\infty}^{j} {}^n \tilde F_{\infty,k}^+\big) +\big(\sum_{k=j+1}^{+\infty}{}^n \tilde F_{0,k}^- + \sum_{k=j+1}^{+\infty}{}^n \tilde F_{\infty,k}^-\big)\Big) \Big|_{\tilde V_j^+}\Big.,\\
&\tilde R_{j,-}^n:=\Big(\big(\sum_{k=-\infty}^{j} {}^n \tilde F_{0,k}^+ + \sum_{k=-\infty}^{j-1} {}^n \tilde F_{\infty,k}^+\big) +\big(\sum_{k=j+1}^{+\infty} {}^n \tilde F_{0,k}^{-} + \sum_{k=j}^{+\infty} {}^n \tilde F_{\infty,k}^-\big)\Big) \Big|_{\tilde V_j^-}\Big. , \ j \in\mathbb Z,\nonumber
\end{align}
where: 
\begin{align}\begin{split}\label{eq:redi}
&\sum_{k=j+1}^{+\infty} {}^n \tilde F_{0,k}^-(\zeta)=\\
&\begin{cases}
=\!\frac{1}{2\pi i}\int_{\mathcal C_{0}^{j+1}}\frac{g_{0}^{j+1}\big(e^{-2\pi i(\tilde\Psi_\mathrm{nf}(w)+\tilde R_{j,+}^{n-1}(w))}\big)}{w-\zeta} dw+\\[0.2cm]
\quad +\frac{1}{2\pi i}\sum_{k=j+2}^{+\infty}\int_{\mathcal C_{0}^k}\frac{g_{0}^k\big(e^{-2\pi i(\tilde\Psi_\mathrm{nf}(w)+\tilde R_{k-1
,+}^{n-1}(w))}\big)}{w-\zeta} dw,\\[0.2cm]
\qquad\qquad\ \zeta\in \tilde V_j^+,\ \text{Im}(\zeta)\leq (4j+1)\frac{\pi}{2}-\varepsilon,\quad \text{(region $(1)$)}\\[0.5cm]
=\!\frac{1}{2\pi i}\int_{\mathcal C_{0}^{j+1}}\frac{g_{0}^{j+1}\big(e^{-2\pi i(\tilde\Psi_\mathrm{nf}(w)+\tilde R_{j,+}^{n-1}(w))}\big)}{w-\zeta} dw+g_{0}^{j+1}\big(e^{-2\pi i(\tilde\Psi_\mathrm{nf}(\zeta)+\tilde R_{j,+}^{n-1}(\zeta))}\big)+\\[0.2cm]
\ \ \ +\frac{1}{2\pi i}\sum_{k=j+2}^{+\infty}\int_{\mathcal C_{0}^k}\frac{g_{0}^k\big(e^{-2\pi i(\tilde\Psi_{\mathrm{nf}}(w)+\tilde R_{k-1
,+}^{n-1}(w))}\big)}{w-\zeta} dw,\\[0.2cm]
\qquad\qquad\ \zeta\in \tilde V_j^+,\ \text{Im}(\zeta)\geq (4j+1)\frac{\pi}{2}+\varepsilon,\quad \text{(region $(2)$)}\\[0.5cm]
=\!\frac{1}{2\pi i}\int_{\mathcal C_{0,+2\varepsilon}^{j+1}}\frac{g_{0}^{j+1}\big(e^{-2\pi i(\tilde\Psi_\mathrm{nf}(w)+\tilde R_{j,+}^{n-1}(w))}\big)}{w-\zeta} dw+\frac{1}{2\pi i}\int_{\mathcal S_{0,+2\varepsilon}^{j+1}}\frac{g_{0}^{j+1}\big(e^{-2\pi i(\tilde\Psi_\mathrm{nf}(w)+\tilde R_{j,+}^{n-1}(w))}\big)}{w-\zeta} dw+\\[0.2cm]
\quad +\frac{1}{2\pi i}\sum_{k=j+2}^{+\infty}\int_{\mathcal C_{0}^k}\frac{g_{0}^k\big(e^{-2\pi i(\tilde\Psi_\mathrm{nf}(w)+\tilde R_{k-1,+}^{n-1}(w))}\big)}{w-\zeta} dw,\\[0.2cm]
\quad\qquad\quad\ \zeta\in \tilde V_j^+,\ (4j+1)\frac{\pi}{2}-\varepsilon<\text{Im}(\zeta)< (4j+1)\frac{\pi}{2}+\varepsilon. \ \text{\ (region $(3)$)}
\end{cases}
\end{split}
\end{align}
\smallskip

\noindent The other three sums in $\tilde R_{j,+}^n$ and the sums in $\tilde R_{j,-}^n$ in \eqref{eq:FJi} can be written analogously. Here, $\varepsilon>0$ is sufficiently small. Recall that $\mathcal C_{0}^{j+1}=\{\zeta\in\widetilde{\mathcal R}_C:\,\mathrm{Im}(\zeta)=(4j+1)\frac{\pi}{2}\}$ is the central line of the petal $\tilde V_{0}^{j+1}$. The line $\mathcal C_{0,+2\varepsilon}^{j+1}$ is the line $\mathcal C_{0}^{j+1}$ shifted upwards by  $+2\varepsilon$ in $\tilde V_0^{j+1}$, and $\mathcal S_{0,+2\varepsilon}^{j+1}$ is the boundary arc of $\tilde V_{0}^{j+1}$ between the lines $\mathcal C_{0}^{j+1}$ and $\mathcal C_{0,+2\varepsilon}^{j+1}$,  independent of $n\in\mathbb N$. Note that 
$$
\int_{\mathcal S_{0,+2\varepsilon}^{j+1}}\frac{g_{0}^{j+1}\big(e^{-2\pi i(\tilde\Psi_\mathrm{nf}(w)+\tilde R_{j,+}^{n-1}(w))}\big)}{w-\zeta} dw
$$
is, as in the proof of Lemma~\ref{prop:cah}, an analytic function at $\zeta=\infty$. It depends on $j\in\mathbb Z$ and on $n\in\mathbb N$. 

Regions (1) -- (3) in \eqref{eq:redi} are regions where Cauchy-Heine formulas differ due to \emph{critical} line of integration $\mathcal C_0^{j+1}$ lying inside the petal $V_j^+$. To simplify calculations, we assume that there is only one \emph{critical} line of integration inside $V_j^+$, while in reality there is another, $\mathcal C_{\infty}^{j}$, the central line of $\tilde V_{\infty}^j$. No new phenomena are generated if we add another line, just more regions and longer expressions in \eqref{eq:redi}, so we simplify without real loss of generality. The regions are shown in Figure~\ref{fig:regions}. More details are given in Remark~\ref{rem:regions} below.

\begin{figure}[h!]
\includegraphics[scale=0.13]{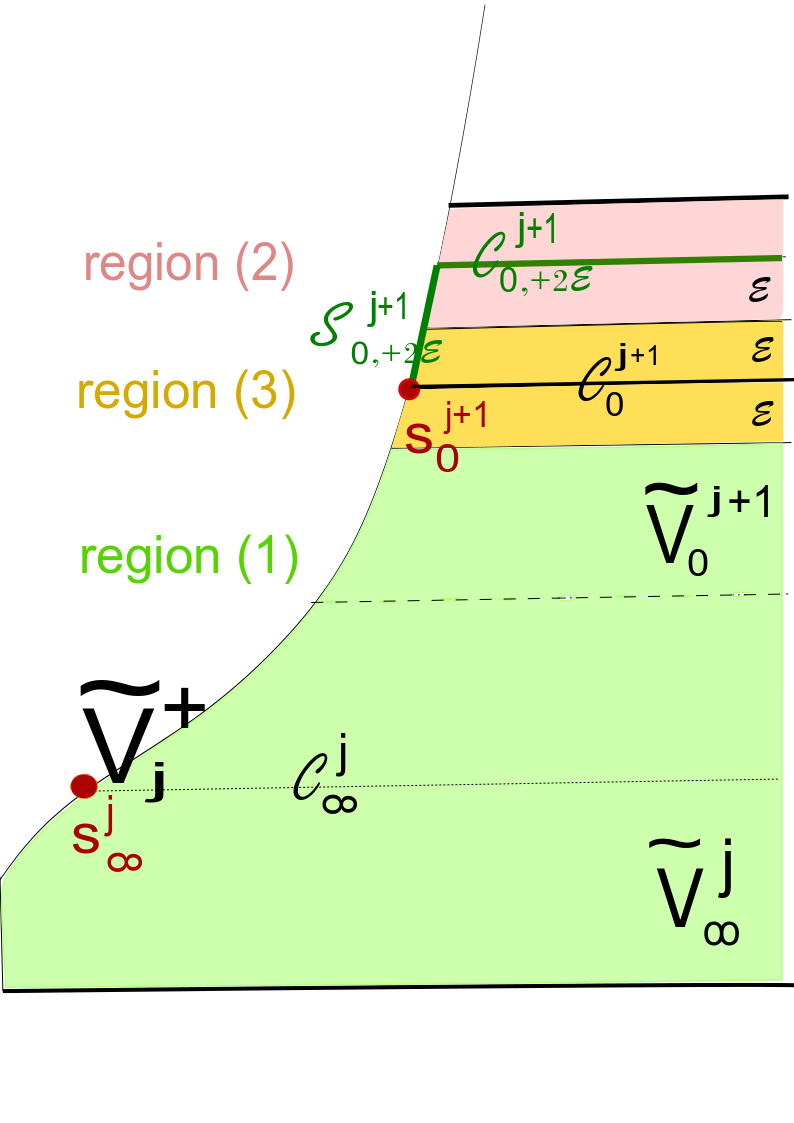}
\vspace{-0.4cm}
\caption{The three regions of $\tilde V_j^+$ with respect to the \emph{critical} line $\mathcal C_0^{j+1}$ of integration, and the \emph{critical} points $s_0^{j+1},\,s_\infty^j\in\tilde V_j^+$ generating logarithmic singularities in the proof of Lemma~\ref{lema:use}.}\label{fig:regions}
\end{figure}

\begin{obs}[Regions $(1)$, $(2)$, $(3)$ introduced in \eqref{eq:redi}]\label{rem:regions} The functions $\tilde R_{j,+}^n,\ n\in\mathbb N,$ in our iterative process are defined as infinite sums of Cauchy-Heine integrals on corresponding petals $V_j^{\pm}$, similarly as in \eqref{eq:ch} and \eqref{eq:FJ}. In every step we use another exponentially small cocycle defined from functions obtained in the previous step. Note that functions $\tilde R_j^{\pm}$ in \eqref{eq:FJ} cannot be expressed  by the same formula throughout the whole petal $\tilde V_j^{\pm}$, since integrals are not well-defined along two \emph{critical} lines of integration that fall inside each petal. Recall that, standardly, in the Cauchy-Heine construction, to extend the function analytically beyond the line of integration, we change the paths of integration, as in the proof of Lemma~\ref{prop:cah}. 

Each petal $\tilde V_j^\pm$ in $\zeta$-chart is divided into \emph{horizontal strip-like} regions (\emph{sectors} in $z$-variable). In each region, we have an explicit, but different integral formula. 

We take $\varepsilon>0$ small. Take petal $V_j^+$. The \emph{region $(3)$} is the open $\varepsilon$-neighborhood of two \emph{critical} lines $\mathcal C_0^{j+1}$ and $\mathcal C_\infty^j$. These two lines are among the lines of integration in \eqref{eq:FJ} for $V_j^+$, and analogously later in the iterative construction given by \eqref{eq:FJi}. At the same time, they lie inside $\tilde V_j^+$. The problem in this region is that, although we may exchange the line of integration with a line outside the region and a part of the boundary (here: $\mathcal C_{0,+2\varepsilon}^{j+1}$ and $\mathcal S_{0,+2\varepsilon}^{j+1}$), we cannot bound the variable $\zeta\in \tilde V_j^+$ away from the part of the boundary, and logarithmic singularities appear in iterations at $s_0^{j+1}$ and $s_\infty^j$, see Figure~\ref{fig:regions}. This prevents an easy proof of convergence in our iterative process. The other strips of $\tilde V_j^+$ constitute \emph{regions $(1)$ and $(2)$}, which are simpler to analyze, as there are no logarithmic singularities. In region $(3)$, the bounds that we need for convergence of iterates in the proof of Lemma~\ref{lema:use} will be significantly more complicated.  
\end{obs}

\emph{Proof of statement $(2)$}. At each step of the iterative Cauchy-Heine construction, two logarithmic singularities  appear at points $s_{0}^{j+1}$ and $s_{\infty}^j$ at the boundary of each petal $V_j^+$ in region $(3)$, $j\in\mathbb Z$. Precisely, they appear at  endpoints of $\mathcal C_{0}^{j+1}$ and $\mathcal C_{\infty}^j$ at the boundary of the domain. Therefore, we will not be able to prove that the sequence of iterates $\big(\tilde R_{j,+}^{n-1}(\zeta)\big)_n$ is \emph{uniformly} Cauchy on the whole petal $\tilde V_{+}^j$. More details about the nature of the singularities can be found in Subection~\ref{subsec:ap} in the Appendix. However, by Lemma~\ref{lem:uses}, the sequence
\begin{equation}\label{eq:C}
\big(e^{2\pi i \tilde R_{j,+}^{n}(\zeta)}\big)_n
\end{equation}
is uniformly Cauchy on petals $\tilde V_+^j$, $j\in\mathbb Z$. By taking the exponential, we have \emph{eliminated} the logarithmic singularities. It follows from \eqref{eq:C} that $\big(\tilde R_{j,+}^{n}(\zeta)\big)_n$ is uniformly Cauchy on all compact subsets of the petal $\tilde V_+^j$, away from singular points $s_{0}^{j+1}$ and $s_{\infty}^j$ with logarithmic singularities, which lie at the boundary of the petal $V_j^+$. Indeed, note that $e^{2\pi i \tilde R_{j,+}^{n}(\zeta)}$ does not vanish in any point $\zeta\in V_{+}^j$. 
By the mean value theorem, writing $\tilde R_{j,+}^n=\frac{1}{2\pi i} \log (e^{2\pi i \tilde R_{j,+}^n})$, we have:
\begin{align*}
&|\tilde R_{j,+}^{n}(\zeta)-\tilde R_{j,+}^{n+1}(\zeta)|\leq \\
&\leq\frac{1}{\sqrt 2\pi}\sup_{t\in[0,1]}\frac{1}{\big|t e^{2\pi i\tilde R_{j,+}^{n}(\zeta)}+(1-t)e^{2\pi i\tilde R_{j,+}^{n+1}(\zeta)}\big|}\!\cdot\! |e^{2\pi i \tilde R_{j,+}^{n}(\zeta)}-e^{2\pi i \tilde R_{j,+}^{n+1}(\zeta)}|,\, \zeta\in\tilde V_{j}^+.
\end{align*}  
By Lemma~\ref{lema:use} (1), we get that $\zeta\mapsto \sup_{t\in[0,1]}\frac{1}{\big|t e^{2\pi i\tilde R_{j,+}^{n}(\zeta)}+(1-t)e^{2\pi i\tilde R_{j,+}^{n+1}(\zeta)}\big|}$ is uniformly bounded on every compact in the petal  $\tilde V_{j}^+$ away from singular points $s_{0}^{j+1}$ and $s_\infty^j$. We conclude that the sequence $(\tilde R_{j,+}^n)_n$ is \emph{uniformly Cauchy on every compact} in the petal $\tilde V_j^+$. Therefore, by the Weierstrass theorem, it converges to an analytic function $\tilde R_{j}^+$ on the petal $\tilde V_{j}^+$, $j\in\mathbb Z$. The same can be concluded for $\tilde R_j^-$ on petals $\tilde V_j^-$, $j\in\mathbb Z.$ 
\smallskip

\noindent Let us now denote the pointwise limits by $\tilde R_{j}^\pm$:
$$
\tilde R_{j}^{\pm}(\zeta):=\lim_{n\to\infty} \tilde R_{j,\pm}^n (\zeta),\ \tilde\Psi_{j}^\pm(\zeta):=\tilde \Psi_\mathrm{nf}(\zeta)+\tilde R_{j,\pm}(\zeta),\ \zeta\in\tilde V_{\pm}^j.
$$
That is, returning from the $\zeta$-variable to the original variable $z$, we put:
$$
\breve R_{j}^{\pm} (\boldsymbol\ell):=\tilde R_{j}^\pm(\boldsymbol\ell^{-1}),\ \Psi_{j}^{\pm}(z):=\Psi_\mathrm{nf}(z)+\breve R_{j}^{\pm}(\boldsymbol\ell),\ z\in V_{\pm}^j,\ j\in\mathbb Z.
$$
Here, $\Psi_\mathrm{nf}(z)$ i.e. $\tilde \Psi_\mathrm{nf}(\zeta)$ are the Fatou coordinates of $(2,m,\rho)$-model, analytic on whole $\mathcal R_C$, and given explicitely in \eqref{eq:abo}. All functions defined above are analytic on respective petals.
Now, passing to the limit in \eqref{e:pp}, we see that $R_{j}^\pm(z)$ and thus also $\Psi_j^{\pm}(z)$ (since $\Psi_\mathrm{nf}(z)$ is analytic on the standard quadratic domain) realize the requested sequence of pairs $(g_{0}^j,g_\infty^j)_{j\in\mathbb Z}$ at intersections of petals, as in \eqref{eq:pt}.
\smallskip

\emph{Proof of statement $(3)$.} We use the uniform estimate \eqref{eq:ge} for ${}^n \tilde G_{0,\infty}^j(\zeta)$ by $n\in\mathbb N$, deduced in the proof of statement $(1)$, and repeat the proof of \eqref{eq:ar} in Proposition~\ref{lem:cocyc} (see the proof of Lemma~\ref{prop:asylin} in the Appendix), but with this uniform estimate. We get that there exists a \emph{uniform} (in $j$) constant $C>0$ such that, for substrips $\tilde S_j\subset \tilde V_j^{\pm}$ centered at line $\{\mathrm{Im}(\zeta)=j\pi\}$ and of the same opening for all $j\in\mathbb Z$, the following estimate holds:
\begin{equation}\label{eq:aaa}
|\tilde R_{j,\pm}^n(\zeta)|\leq C|\zeta|^{-1},\ \zeta\in \tilde S_j,\ j\in\mathbb Z,\ n\in\mathbb N.
\end{equation}Passing to the limit as $n\to\infty$ in \eqref{eq:aaa}, and returning to the original variable $z=e^{-\zeta}$, statement $(3)$ is proven.
\hfill $\Box$
\bigskip

\subsection{The symmetry of the horn maps and $\mathbb R_+$-invariance}\ 

We have proven in \cite[Proposition 9.2]{prvi} that, for a parabolic generalized Dulac germ $f$, the fact that $f(\mathbb R_+\cap \mathcal R_C)\subset \mathbb R_+\cap\mathcal R_C$ implies the \emph{symmetry} \eqref{eq:sim} of its analytic moduli. Here, by abuse, $\mathbb R_+:=\{z\in\mathcal R:\,\mathrm{Arg}(z)=0\}$. In general, the converse of \cite[Proposition 9.2]{prvi} does not hold. That is, the symmetry of horn maps of $f$ does not imply $\mathbb R_+$-invariance of $f$ in general, as Example~\ref{ex:nosim} below shows. Instead, Lemma~\ref{lema:simi} provides a characterization of analytic germs on standard quadratic domains having symmetric sequences of horn maps.

\begin{example}\label{ex:nosim}
Take $f(z)=z-z^2$ on $\mathcal R_C$. Obviously, $f$ is a simple parabolic generalized Dulac germ and $f(\mathbb R_+)\subseteq \mathbb R_+$. By \cite[Proposition 9.2]{prvi}, since $f$ is $\mathbb R_+$-invariant, its moduli are symmetric. Take now $\varphi(z)=z+iz^3$, and define an analytic germ $f_1:=\varphi^{-1}\circ f\circ\varphi$ on $\mathcal R_C$. Since $\widehat f_1(z)=z-z^2+o(z^2)$, $f_1$ admits the same petals as $f$. By \cite[Theorem B]{prvi}, since $\varphi(z)$ is analytic on $\mathcal R_C$, $f_1$ has the same horn maps as $f$. Therefore, the horn maps of $f_1$ are symmetric, but $\mathbb R_+$ is \emph{not} $f_1$-invariant. 

We can easily generate more complicated examples by taking an $\mathbb R_+$-invariant parabolic generalized Dulac germ and by conjugating it by $\varphi(z)=z+o(z)$ which is analytic on a standard quadratic domain, and whose asymptotic expansion $\widehat\varphi$ belongs to $\widehat{\mathcal L}(\mathbb C)$, but not to $\widehat{\mathcal L}(\mathbb R)$. Thus the invariance of $\mathbb R_+$ is not preserved in general.

Indeed, analytic modulus is an invariant of analytically conjugated parabolic germs. On the other hand, having the real axis invariant is obviously not an invariant
property under complex changes of coordinates. If one of the germs has the real axis invariant, all analytically conjugated germs also have an invariant real analytic curve through the singularity, but it is not in general the real axis.
\end{example}

\begin{lem}[Symmetry of the horn maps]\label{lema:simi} Let $f$ be an analytic germ on a standard quadratic domain $\mathcal R_C$ with a sequence of horn maps\footnote{Note that, by saying that $f$ has horn maps $(h_{0}^j,h_{\infty}^j; \sigma_j)_{j\in\mathbb Z}$, we have implicitely assumed the dynamics and the existence of invariant petals $V_j^\pm\subset\mathcal R_C$, $j\in\mathbb Z$.} $(h_{0}^j,h_{\infty}^j; \sigma_j)_{j\in\mathbb Z}$, with $\sigma_j$ as in \eqref{eq:fall}. The sequence of horn maps is \emph{symmetric}, that is, 
\begin{equation}\label{eq:symi}
\overline{\big(h_{0}^{-j+1}\big)^{-1}(t)}\equiv h_{\infty}^j(\overline t),\ t\in(\mathbb C,0),\ j\in\mathbb Z,
\end{equation}
if and only if there exists an analytic germ $\varphi(z)=z+o(z)$ on $\mathcal R_C$ such that
\begin{equation}\label{eq:req}
\overline {f(\overline z)}=\varphi^{-1}\circ f\circ \varphi(z),\ z\in\mathcal R_C.
\end{equation}
\end{lem}
Note that \eqref{eq:req} is trivially satisfied for germs $f$ such that $f(\mathbb R_+)\subseteq \mathbb R_+$, taking $\varphi=\mathrm{id}$, by the Schwarz reflection principle.

\begin{proof} Let $f$ be analytic on a standard quadratic domain $\mathcal R_C$. Let $f_1(z):=\overline{f(\overline z)},\ z\in\mathcal R_C$. It is an analytic function on $\mathcal R_C$ by the Cauchy-Riemann conditions. Let $(k_{0}^j,k_{\infty}^j; \sigma_j)_{j\in\mathbb Z}$ be its sequence of horn maps ($\sigma_j$ remains the same, due to symmetry of standard quadratic domains). Then, by the proof of \cite[Proposition 9.2]{prvi}, it holds that 
\begin{equation}\label{eq:simk}
\overline{\big(k_{0}^{-j+1}\big)^{-1}(t)}\equiv h_{\infty}^j(\overline t),\ \overline{(k_\infty^j)^{-1}(t)}=h_0^{-j+1}(\overline t),\ t\in(\mathbb C,0),\ j\in\mathbb Z.
\end{equation}
By \eqref{eq:simk} and symmetry \eqref{eq:symi} of the horn maps of $f$, we conclude that $f_1$ and $f$ have the \emph{same} sequence of horn maps. By \cite[Theorem B]{prvi}, there exists an analytic function $\varphi(z)=z+o(z)$ on $\mathcal R_C$ such that
$$
\overline {f(\overline z)}=\varphi^{-1}\circ f \circ \varphi(z),\ z\in\mathcal R_C.
$$
The other direction is proven similarly.
\end{proof}
However, in Lemma~\ref{lem:simi} we show that, if we take a symmetric sequence of pairs of analytic germs from $\mathrm{Diff}(\mathbb C,0)$, by Cauchy-Heine construction from Lemma~\ref{lema:konver} we realize the sequence by a \emph{representative that is indeed $\mathbb R_+$-invariant}, as its horn maps. The reason lies in the symmetry of the Cauchy-Heine construction.
\begin{lem}\label{lem:simi} Let $(h_{0}^j,h_{\infty}^j; \sigma_j)_{j\in\mathbb Z}$, with $\sigma_j$ as in \eqref{eq:fall}, be a symmetric sequence of pairs of analytic germs from $\mathrm{Diff}(\mathbb C,0)$, such that:
\begin{equation}\label{eq:sym}
\overline{\big(h_{0}^{-j+1}\big)^{-1}(t)}\equiv h_{\infty}^j(\overline t),\ t\in (\mathbb C,0),\ j\in\mathbb Z. 
\end{equation}
Let $\Psi_j^{\pm}(z):=\Psi_{\mathrm{nf}}+\breve R_j^{\pm}(\boldsymbol\ell)$, $\breve R_j^{\pm}(\boldsymbol\ell):=R_j^{\pm}(z)$, $z\in V_j^{\pm}$, be as constructed by the iterative Cauchy-Heine construction in Lemma~\ref{lema:konver}, realizing the sequence of pairs $(h_{0}^j,h_{\infty}^j; \sigma_j)_{j\in\mathbb Z}$ on intersections of petals $V_0^j,\ V_\infty^j,\ j\in\mathbb Z$, either on a standard linear or a standard quadratic domain. Then:
\begin{enumerate}
\item[(a)] $\Psi_0^+$ on $V_0^+$ is $\mathbb R_+$-invariant. That is, $$\Psi_0^+(\mathbb R_+\cap \mathcal R_C)\subseteq \mathbb R_+\cap\mathcal R_C \text{\ \ $($resp. $\mathcal R_{a,b}$$)$}.$$

\item[(b)] In the case of construction on a standard linear domain, the asymptotic expansion $\widehat R(\boldsymbol\ell)$ of $\breve R_j^{\pm}(\boldsymbol\ell)$, as $\boldsymbol\ell\to 0$ on $\boldsymbol\ell$-cusps $\boldsymbol\ell(V_j^\pm)$, belongs to $\mathbb R[[\boldsymbol\ell]]$. That is, the coefficients of the expansion are \emph{real}.
\end{enumerate}
\end{lem}
\noindent The proof is in the Appendix.

\subsection{Proof of Theorem~A}\

\begin{proof} Let $(h_{0}^j,h_{\infty}^{j};\,\sigma_j)_{j\in\mathbb Z}$ be a sequence of pairs of analytic germs of diffeomorphisms, as in Theorem~A. Let $V_j^{\pm}$ be the petals of opening $2\pi$, centered at $j\pi$, $j\in\mathbb Z$, along a standard quadratic domain, as in Figure~\ref{fig:dyn}. By Lemma~\ref{lema:konver}, we construct analytic functions $\Psi_j^{\pm}$ on $V_j^\pm$ that satisfy \eqref{eq:pt}. This is equivalent to the relation \eqref{eq:moduli} for the realization of horn maps. We now define $f$ such that $\Psi_j^\pm$ are its petalwise Fatou coordinates. We define $f$ by petals, using Abel equation, as:
\begin{equation} \label{eq:eff}
f_j^\pm(z):=(\Psi_j^\pm)^{-1}(1+\Psi_j^{\pm}(z)),\ z\in V_j^{\pm},\ j\in\mathbb Z.
\end{equation}
Now we prove that $f_j^\pm$, defined and analytic on petals $V_j^{\pm}$, \emph{glue} to an analytic function $f$ on the whole standard quadratic domain $\mathcal R_C$. That is, we prove that 
\begin{align}\begin{split}\label{eq:treba}
&f_j^+(z)=f_j^-(z),\ z\in V_\infty^j=V_+^j\cap V_-^j,\\
&f_{j-1}^+(z)=f_j^-(z),\ z\in V_0^j=V_{j-1}^+\cap V_j^-, \ j\in\mathbb Z.\end{split}
\end{align} 
Indeed, for Fatou coordinates $\Psi_{j}^{\pm}$ of two consecutive petals by \eqref{eq:pt} of Lemma~\ref{lema:konver} it holds that:
\begin{align*}
&\Psi_-^j\circ (\Psi_+^{j-1})^{-1}(w)=w-g_0^j(e^{-2\pi i w}),\ w\in\Psi_+^{j-1}(V_0^j),\\
&\Psi_-^j\circ (\Psi_+^{j})^{-1}(w)=w+g_\infty^j(e^{2\pi i w}),\ w\in\Psi_+^{j}(V_\infty^j),\ j\in\mathbb Z.
\end{align*}
This implies:
\begin{align*}
&\Psi_-^j\circ(\Psi_{+}^{j-1})^{-1}(w+1)=\Psi_-^j\circ(\Psi_{+}^{j-1})^{-1}(w)+1, \ w\in\Psi_+^{j-1}(V_0^j),\\
&\Psi_-^j\circ(\Psi_{+}^{j})^{-1}(w+1)=\Psi_-^j\circ(\Psi_{+}^{j})^{-1}(w)+1,\ w\in\Psi_+^{j}(V_\infty^j),\ j\in\mathbb Z.
\end{align*}
Composing the first equation by $\Psi_{+}^{j-1}$ from the right and by $(\Psi_-^j)^{-1}$ from the left, and the second by $\Psi_{+}^{j}$ from the right and $(\Psi_-^j)^{-1}$ from the left, by \eqref{eq:eff} we get \eqref{eq:treba}. 

The prenormalized form \eqref{eq:forma} of $f$ follows from Proposition~\ref{prop:pgd} in the Appendix and the \emph{prenormalized} form of the Fatou coordinates $\Psi_j^{\pm}=\Psi_\mathrm{nf}+R_j^{\pm}$ constructed in Lemma~\ref{lema:konver}. Here, $R_j^\pm(z)=o(1)$, as $z\to 0$ on $V_j^\pm$, and $\Psi_\mathrm{nf}$ is the Fatou coordinate of $(2,m,\rho)$-formal model.

The uniform bound $|f(z)-z+z^2\boldsymbol\ell^m-\rho z^3\boldsymbol\ell^{2m+1}|\leq C|z^3\boldsymbol\ell^{2m+2}|,\ C>0,\ z\in\mathcal R_c$, follows by Lemma~\ref{lema:konver} (3). Indeed, uniform bound \eqref{eq:a} gives that there exists $d>0$, independent of $j\in\mathbb Z$, such that $|\Psi_j^{\pm}(z)-\Psi_\mathrm{nf}(z)|\leq d|\boldsymbol\ell|,\ z\in S_j\subset V_j^{\pm}$, where $S_j^\pm$ are subsectors of $V_j^\pm$ of the same opening strictly bigger than $\pi$ for all $j\in\mathbb Z$. The same reasoning as in the proof of Proposition~\ref{prop:pgd} in the Appendix now gives the bound $|f(z)-f_0(z)|\leq e|z^3\boldsymbol\ell^{2m+2}|,\ z\in\mathcal R_c,\ e>0$. Then, using uniform bound\footnote{$f_0=\Psi_{\mathrm{nf}}^{-1}(1+\Psi_{\mathrm{nf}})$ on $\mathcal R_c$, and $\Psi_{\mathrm{nf}}$ is given on $\mathcal R_c$ explicitely by \eqref{eq:abo}.} for the model, $|f_0(z)-z+z^2\boldsymbol\ell^m-\rho z^3 \boldsymbol\ell^{2m+1}|\leq c_1|z^3\boldsymbol\ell^{2m+2}|,\ z\in\mathcal R_c$, $c_1>0$, we get the required bound for $f$.

Finally, on $V_0^+$ the germ $f$ is given by 
\begin{equation}\label{eq:invi}
f\Big|_{V_0^+}\Big.=(\Psi_0^+)^{-1}(1+\Psi_0^+),
\end{equation} and glues analytically along other petals. By Lemma~\ref{lem:simi}, $\Psi_0^+$ is $\mathbb R_+$-invariant. It is also injective on $\mathbb R_+\cap V_0^+$, so the inverse $(\Psi_0^+)^{-1}$ is $\mathbb R_+$-invariant on $\Psi_0^+(V_0^+)$. We conclude by \eqref{eq:invi} that $f$ is $\mathbb R_+$-invariant.
\end{proof}

\medskip

\section{Proof of Theorem~B}\label{subsec:treci}

The analogue of Lemma~\ref{lema:konver} holds (with the same proof) also on \emph{standard linear} domains. Given a sequence of pairs of analytic germs of diffeomorphisms $(h_0^j,h_\infty^j; \sigma_j)_{j\in\mathbb Z}$, with radii of convergence satisfying bounds \eqref{eq:velradlin}, we construct analytic functions $\Psi_j^\pm(z)$ on petals $V_j^{\pm}$ centered at directions $2j\pi$ (attracting), that is, $(2j-1)\pi$ (repelling), but along a \emph{standard linear domain}, that realize this sequence of diffeomorphisms on intersections of petals $V_{0,\infty}^j$, as in \eqref{eq:pt}. We construct them as the uniform limits $R_j^\pm$ on compact subsets of $V_j^\pm$ of iterates $R_{j,\pm}^n(z)$, as $n\to\infty$, defined inductively as in Lemma~\ref{lema:konver}. In each inductive step, we use Proposition~\ref{lem:cocyc1} for realization of cocycles on standard linear domains, instead of Proposition~\ref{lem:cocyc} for standard quadratic domains. Proposition~\ref{lem:cocyc1} additionaly gives us information on asymptotic expansion of $R_{j,\pm}^n$, $n\in\mathbb N$. Let $\breve R_{j,\pm}^n(\boldsymbol\ell):=R_{j,\pm}^n(z),\ z\in V_j^{\pm}$, where $\boldsymbol\ell:=-\frac{1}{\log z}$. Then, by Proposition~\ref{lem:cocyc1}, each $R_{j,\pm}^n(\boldsymbol\ell)$, $j\in\mathbb Z$, admits $\log$-Gevrey expansion in $\mathbb C[[\boldsymbol\ell]]$ of every order $1-\delta$, $\delta>0$, as $\boldsymbol\ell\to 0$ in $\boldsymbol\ell(V_j^{\pm})$.

We now prove that there exists $\widehat R(\boldsymbol\ell)\in\mathbb C[[\boldsymbol\ell]]$ such that the limits $$\breve R_j^{\pm}(\boldsymbol\ell):=\lim_{n\to\infty} \breve R_{j,\pm}^n(\boldsymbol\ell), \ \boldsymbol\ell\in\boldsymbol\ell (V_j^\pm),\ j\in\mathbb Z,$$ admit $\widehat R(\boldsymbol\ell)$ as their $\log$-Gevrey asymptotic expansion of order $1-\delta$, for every $\delta>0$, as $\boldsymbol\ell\to 0$ on $\boldsymbol\ell(V_j^{\pm})$. Moreover, we prove that $\widehat R(\boldsymbol\ell)\in\mathbb R[[\boldsymbol\ell]]$. 

We work again in the logarithmic chart $\zeta=-\log z$. As in the proof of Lemma~\ref{prop:asylin} in the Appendix, on standard \emph{linear} domains it follows that:
\begin{align*}
&\Big|\sum_{k=j+1}^{+\infty} {}^n \tilde F_{0,k}^- (\zeta)-\sum_{j=0}^{N} a_j^n \zeta^{-j}\Big|\leq\\
&\ \ \leq |\zeta|^{-N} \frac{1}{2\pi}\sum_{k=j+1}^{+\infty}\Big|\int_{\mathcal C_0^k}\frac{g_0^k(e^{-2\pi i (\tilde \Psi_\mathrm{nf}(w)+\tilde R_{k-1,+}^{n-1}(w))})w^N}{w-\zeta}\,dw\Big|,\\
& \qquad\qquad\qquad\qquad\qquad\qquad\qquad\qquad\qquad\qquad \zeta\in \tilde V_j^+ \text{ in the region (1)},\ N\in\mathbb N.
\end{align*}
Here we again consider, instead of the whole $\tilde R_{j,+}^n(\zeta)$ given by \eqref{eq:FJi}, only one part of the sum $\sum_{k=j+1}^{+\infty} {}^n \tilde F_{0,k}^- (\zeta),\ \zeta\in\tilde V_j^+$, see \eqref{eq:redi}. For the other three parts of the sum the conclusions follow similarly. To get the bound for $\tilde R_{j,+}^n(\zeta)$, we sum the bounds afterwards. Additionally, for $\zeta\in \tilde V_j^+$ in regions (2) and (3), the conclusion follows similarly. Finally, the same can be done for $\tilde R_{j,-}^n$ on $\tilde V_j^-$. 
Let $\delta>0$. Due to uniform bounds of $g_0^k$ from \eqref{prv} and of $\tilde R_{k,+}^n$ (see Remark~\ref{rem:use}) with respect to $n\in\mathbb N$ and $k\in\mathbb Z$, we conclude that there exist uniform constants $c>0$ and $d>0$, such that:
\begin{equation}\label{eq:twice}
\big|g_0^k(e^{-2\pi i (\tilde \Psi_\mathrm{nf}(w)+\tilde R_{k-1,+}^{n-1}(w))})\big|\!\!<\!c|e^{-2\pi i (\tilde \Psi_\mathrm{nf}(w)+\tilde R_{k-1,+}^{n-1}(w))}|\!<\!d|e^{-2\pi i(\tilde \Psi_\mathrm{nf}(w)/2)}|,\ w\in V_0^k,
\end{equation}
for every $n\in\mathbb N$ and $k\in\mathbb Z$.

\noindent Now, following the proof of Lemma~\ref{prop:asylin} in the Appendix and using \eqref{eq:twice}, we obtain Gevrey bounds which are \emph{uniform with respect to $n\in\mathbb N$}. That is, on every substrip  $\tilde W\subset \tilde V_j^+$, for every $N\in\mathbb N$, there exists a constant $C_N^{\tilde W}>0$ such that, for every $n\in\mathbb N$, it holds that:
\begin{equation} \label{eq:prvaj}
|\tilde R_{j,+}^n(\zeta)-\sum_{i=0}^{N} A_i^{j,n} \zeta^{-i}|\leq C_N^{\tilde W} (1-\delta)^{-N}e^{-\frac{N}{\log N}}\log^NN\cdot|\zeta|^{-N},\ \zeta\in \tilde W\subset \tilde V_j^+.
\end{equation}
Here, $C_N^{\tilde W}$ is uniform in the iterate $n\in\mathbb N$. Also, $A_i^{j,n}\in\mathbb C$ is given by:
\begin{equation*} \label{eq:drugaj}
A_i^{j,n}:=\sum_{k=j+1}^{\infty}\int_{\mathcal C_0^k}g_0^k(e^{-2\pi i (\tilde \Psi_\mathrm{nf}w)+\tilde R_{k-1,+}^{n-1}(w))})w^i \,dw.
\end{equation*}
As discussed before in the proof of Lemma~\ref{prop:asylin}, the above sum converges for every $n\in\mathbb N,\,j\in\mathbb Z$, so the coefficients $A_i^{j,n}\in\mathbb C$ are well-defined. To prove that, for every $j\in\mathbb Z$, $i\in\mathbb N_0$, $(A_i^{j,n})_n$ converges as $n\to\infty$, we use the \emph{dominated convergence theorem}. Indeed, by a change of variable of integration, the above integrals $\int_{\mathcal C_0^k}$ can be considered as line integrals. Now \eqref{eq:twice} and the convergence of the following integrals: 
$$
\int_{\mathcal C_0^k} |e^{-2\pi i(\tilde \Psi_\mathrm{nf}(w)/2)}|w^i\,dw,\ k\in\mathbb Z,
$$
due to the exponential flatness of $e^{-2\pi i(\tilde \Psi_\mathrm{nf}(w)/2)}$ on $\mathcal C_0^k$, $k\in\mathbb Z$, ensures all the assumptions of the dominated convergence theorem. We put:
$$
A_i^j:=\lim_{n\to\infty} A_i^{j,n}\,\in\mathbb C,\ j\in\mathbb Z,\ i\in\mathbb N_0.
$$
Now passing to the limit $\lim_{n\to\infty}$ in \eqref{eq:prvaj}, we get that $\tilde R_{j}^+(\zeta):=\lim_{n\to\infty} \tilde R_{j,+}^n(\zeta)$, $\zeta\in\tilde V_j^+,$ admits a $\log$-Gevrey asymptotic expansion of order $1-\delta$ in $\mathbb C[[\zeta^{-1}]]$, as $\mathrm{Re}(\zeta)\to \infty$.  

In addition, the asymptotic expansions of $\tilde R_j^{\pm}(\zeta)$ are \emph{the same} for every $j\in\mathbb Z$, because of exponentially small differences on intersections of petals \eqref{eq:pt}. Recall that $\tilde \Psi_j^{\pm}:=\tilde \Psi_\mathrm{nf}+\tilde R_j^{\pm}$ on $\tilde V_j^{\pm}$, where $\tilde \Psi_\mathrm{nf}$ is \emph{globally} analytic on a standard quadratic domain. We denote this expansion by $\widehat R(\zeta^{-1})\in\mathbb C[[\zeta^{-1}]]$. That is, putting $\boldsymbol\ell:=\zeta^{-1}=-\frac{1}{\log z}$, all $\breve R_j^{\pm}(\boldsymbol\ell):=\tilde R_j^\pm(\zeta)$ admit $\widehat R(\boldsymbol\ell)$, as their $\log$-Gevrey asymptotic expansion of order $1-\delta$, as $\boldsymbol\ell\to 0$  on $\boldsymbol\ell(V_j^\pm)$, $j\in\mathbb Z$. 
\medskip

Finally, we prove that $f$, expressed as in \eqref{eq:eff} from $\tilde \Psi_j^{\pm}$, and which, by the proof of Theorem~A, \emph{glues} to an analytic function on a standard linear domain $\widetilde{\mathcal R}_{a,b}$, is a parabolic generalized Dulac germ. The uniform bound \eqref{eq:fuin} and the prenormalized form of $f$ follow from Lemma~\ref{lema:konver} (3) and by Proposition~\ref{prop:pgd} in the Appendix, exactly as in the proof of Theorem~A. Also, the invariance of $\mathbb R_+$ follows by Lemma~\ref{lem:simi}, as in the proof of Theorem~A. 

We prove only the existence of the generalized Dulac expansion $\widehat f$ of $f$. It follows by  \eqref{eq:eff} and by the $\log$-Gevrey asymptotic expansions of $\breve R_j^{\pm}(\boldsymbol\ell) \text{ on } \boldsymbol\ell(V_j^{\pm}),\ j\in\mathbb Z$, proven above. We return to the original variable $z$. On each petal $V_{j}^{\pm}$, we expand \eqref{eq:eff} in Taylor expansion:
\begin{equation}\label{eq:prav}
f_{j}^{\pm}(z)=z+\frac{1}{(\Psi_{j}^{\pm})'(z)}+\frac{1}{2!}\Big(\frac{1}{(\Psi_{j}^{\pm})'(z)}\Big)'\cdot \frac{1}{(\Psi_{j}^{\pm})'(z)}+\frac{1}{3!}(\textrm{previous\ term})'\cdot\frac{1}{(\Psi_{j}^{\pm})'(z)}+\ldots
\end{equation}
In the sequel, we put $\breve R_j^{\pm}(\boldsymbol\ell):=R_j^{\pm}(z),\ z\in V_j^{\pm}$. Let $\widehat R(\boldsymbol\ell)$ denote its $\log$-Gevrey asymptotic expansion of order $1-\delta$, $\delta>0$, in $\mathbb C[[\boldsymbol\ell]]$, the same for all $j\in\mathbb Z$. It holds: \begin{equation*}\label{eq:kriv}(\Psi_{j}^{\pm})'(z)=-\frac{1}{z^2\boldsymbol\ell^m}+\frac{1}{z}+\big(\frac{m}{2}+\rho\big)\frac{\boldsymbol\ell}{z}+\frac{\boldsymbol\ell^2}{z}(\breve R_{j}^\pm)'(\boldsymbol\ell),\ \rho\in\mathbb R,\ m\in\mathbb Z,\ z\in V_{j}^{\pm}. \end{equation*} Here, the germs 
$(\breve R_{j}^\pm)'(\boldsymbol\ell)$ are analytic on $\boldsymbol\ell$-cusps $\boldsymbol\ell(V_{j}^{\pm})$, $j\in\mathbb Z$. By \cite[Proposition~4.7]{prvi}, they expand $\log$-Gevrey of order $1-\delta$, for every $\delta>0$, in their formal counterpart $\widehat R'(\boldsymbol\ell)$, as $\boldsymbol\ell\to 0$ on $\boldsymbol\ell$-cusp $\boldsymbol\ell(V_{j}^{\pm})$. By \cite[Proposition 4.7]{prvi}, $\widehat R'(\boldsymbol\ell)$ is obtained by termwise (formal) derivation of $\widehat R(\boldsymbol\ell)$. The same conclusion can be drawn for all finite derivatives $(\breve R_j^{\pm})^{(k)}(\boldsymbol\ell)$, $k\in\mathbb N$, by \cite[Proposition 4.7]{prvi}. 
Furthermore, we define analytic functions $H_{j}^{\pm}(\boldsymbol\ell)$ on $\boldsymbol\ell$-cusps $\boldsymbol\ell(V_{j}^{\pm})$, $j\in\mathbb Z$, through the following equation:
\begin{equation*}\label{eq:jjo}
\frac{1}{(\Psi_j^\pm)'(z)}=\frac{-z^2\boldsymbol\ell^m}{1-z\boldsymbol\ell^m-(\frac{m}{2}+\rho) z\boldsymbol\ell^{m+1}+z\boldsymbol\ell^{m+2}(\breve R_{j}^{\pm})'(\boldsymbol\ell)}=:\frac{-z^2\boldsymbol\ell^m}{1+zH_{j}^{\pm}(\boldsymbol\ell)}.
\end{equation*}
By \cite[Propositions 4.5-4.7]{prvi} about closedness of $\log$-Gevrey classes to algebraic operations and to differentiation, they expand $\log$-Gevrey of order $1-\delta$, for every $\delta>0$, in the common formal counterpart $\widehat H(\boldsymbol\ell)$, as $\boldsymbol\ell\to 0$ on respective $\boldsymbol\ell$-cusps $\boldsymbol\ell(V_j^{\pm})$, $j\in\mathbb Z$.
Note that
\begin{equation}\label{eq:jjj}
\frac{z^2\boldsymbol\ell^m}{1+zH_{j}^{\pm}(\boldsymbol\ell)}=z^2\boldsymbol\ell^m\sum_{k=0}^{\infty}(-1)^k  z^k (H_{j}^{\pm}(\boldsymbol\ell))^k.
\end{equation}
Putting \eqref{eq:jjj} in \eqref{eq:prav}, and regrouping the terms with the same powers of $z$, we get:
\begin{equation}\label{eq:q}
f(z)=z-z^2\boldsymbol\ell^m + \rho z^3\boldsymbol\ell^{2m+1}+\sum_{k=3}^{\infty} z^k Q_{j,k}^\pm(\boldsymbol\ell),\ z\in V_j^{\pm},\ j\in\mathbb Z.
\end{equation}
Here, $Q_{j,k}^\pm(\boldsymbol\ell)$, $k\in\mathbb N$, $k\geq 3$, are realized as finite sums of finite products of $\boldsymbol\ell$ and $H_{j}^{\pm}(\boldsymbol\ell)$ and their finite derivatives (of order at most $k-2$),  the same for all petals $j\in\mathbb Z$. Therefore, by \cite[Propositions 4.5-4.7]{prvi} about closedness of $\log$-Gevrey classes to algebraic operations and differentiation, they expand $\log$-Gevrey of order $1-\delta$, for every $\delta>0$, in their formal counterpart, denoted $\widehat Q_k(\boldsymbol\ell)$. Note that $\boldsymbol\ell$-cusps $\boldsymbol\ell(V_j^{\pm})$ are $\boldsymbol\ell$-images of sectors of opening $2\pi>\pi$. 

Finally, by Lemma~\ref{lem:simi} (b), $\widehat R(\boldsymbol\ell)\in \mathbb R[[\boldsymbol\ell]]$. Therefore, all $\widehat Q_k(\boldsymbol\ell)$, as algebraic combinations of $\widehat R(\boldsymbol\ell)$, its derivatives and powers of $\boldsymbol\ell$ with real coefficients, belong to $\mathbb R[[\boldsymbol\ell]]$. This proves the generalized Dulac expansion of $f$ from Definition~\ref{def:gD}.
\hfill $\Box$
\medskip

In Remark~\ref{rem:sq} in the Appendix we explain why the arguments giving the asymptotic expansion in Theorem~B do not work for quadratic domains in Theorem~A. 

\begin{obs}
Note that, although $f$ is analytic on the whole standard linear domain $\mathcal R_{a,b}$, the \emph{coefficient functions} $Q_{j,k}^\pm(\boldsymbol\ell)$, $k\in\mathbb N,\ k\geq 3$, in its expansion \eqref{eq:q} are analytic in general only on $\boldsymbol\ell$-cusps $\boldsymbol\ell(V_j^{\pm})$ and \emph{do not glue} (in $j$) to an analytic function on whole $\boldsymbol\ell(\mathcal R_{a,b})$. Indeed, this is obviously not true already for $Q_{j,3}^\pm(\boldsymbol\ell):=1-H_{j}^{\pm}(\boldsymbol\ell)$, by \eqref{eq:jjj}. On overlapping cusps $\boldsymbol\ell(V_j^{\pm})$, the $\boldsymbol\ell$-images of petals $V_j^\pm$, they have exponentially small differences.
\end{obs}

\begin{obs}\label{rem:final} Let the germs $f$ resp. $g$ be the germs obtained by Cauchy-Heine construction on a linear resp. quadratic domain realizing the same sequence of moduli. It is important to note that, in general, $f$ is not the restriction of the germ  $g$, since we apply Cauchy-Heine integrals along different lines, see Remark~\ref{rem:linqua}. 

Nevertheless, $f$ and the restriction of $g$ to a linear domain by construction have the same moduli on the linear domain, and are thus analytically conjugated on the linear domain. However, we are not sure if the analytic conjugacy between the two germs on the linear domain can be analytically extended to a quadratic domain, or if there is some singularity outside the smaller domain preventing the extension. If former was the case, we would have a representative of the analytic class of $g$ on a quadratic domain whose restriction to the linear domain is $f$; that is, a representative with the generalized Dulac asymptotic expansion. This would positively resolve the question of extending the realization result to parabolic Dulac germs on a standard quadratic domain, which for the moment remains open.
\end{obs}



\section{Prospects.}
The realization Theorem~B for uniformly bounded sequences of pairs of germs of analytic diffeomorphisms fixing the origin as horn maps is proven in the larger class of \emph{parabolic generalized Dulac germs} on \emph{standard linear domains}, that contains parabolic Dulac germs. The question if the construction can be extended to \emph{standard quadratic} domains remains open. Another important problem is to characterize uniformly bounded sequences of pairs of analytic diffeomorphisms which can be realized as horn maps of \emph{parabolic Dulac germs}.

\section{Appendix} 

\begin{prop}\label{prop:pgd}
Let $f$ be a parabolic generalized Dulac germ on a standard quadratic (or standard linear) domain. It is \emph{prenormalized}, that is, of the form:
\begin{equation*}
f(z)=z-z^2\boldsymbol\ell^m+\rho z^3\boldsymbol\ell^{2m+1}+o(z^2\boldsymbol\ell^m),\ m\in\mathbb Z,\ \rho\in\mathbb R,
\end{equation*}
if and only if its sectorial Fatou coordinate is of the form:
$$
\Psi_j^{\pm}=\Psi_\mathrm{nf}+R_j^{\pm}, \text { on }V_j^{\pm}, 
$$
where $R_j=o(1)$, as $z\to 0,\ z\in V_j^{\pm}$, and $\Psi_\mathrm{nf}$ is the \emph{global} Fatou coordinate of the formal normal form $f_0$ given by:
\begin{equation}\label{eq:abo}\Psi_\mathrm{nf}(z):=-\int_{z_0}^z \frac{dz}{z^2\boldsymbol\ell^m}+\log z-(\frac{m}{2}+\rho) \log(-\log z).\end{equation}
Here, $z_0$ is a freely chosen initial point in the standard quadratic $($or linear$)$ domain  $($the choice of additive constant in $\Psi_\mathrm{nf})$.
\end{prop}

\begin{proof} One direction is proven by Taylor expansion of the Abel equation. For the other, putting $f=f_0+h$ and $\Psi_j^{\pm}=\Psi_\mathrm{nf}+R_j^{\pm}$ in $f=(\Psi_j^{\pm})^{-1}(1+\Psi_j^{\pm})$ and comparing initial terms, we estimate $h(z)=O(z^{3}\boldsymbol\ell^{2m+2})$, as $z\to 0$. The estimate is not necessarily uniform for all petals.
\end{proof}
\smallskip

\subsection{Proof of Proposition~\ref{prop:modest}}
\smallskip 

\noindent In the proof of Proposition~\ref{prop:modest}, we use Lemma~\ref{lem:modest1}.

\begin{lem}[Uniform bound on the Fatou coordinate of a uniformly bounded germ]\label{lem:modest1} Let $f(z)=z-z^2\boldsymbol\ell^m+\rho z^3\boldsymbol\ell^{2m+1}+o(z^3\boldsymbol\ell^{2m+1}),\ m\in\mathbb Z,\ \rho\in\mathbb R,$ be a prenormalized analytic germ on a standard quadratic or standard linear domain $\mathcal R_c$. Let $f$ satisfy the uniform bound \eqref{eq:fuin}. Let $\Psi_\mathrm{nf}(z)$, $z\in\mathcal R_c$, be the Fatou coordinate of the formal $(2,m,\rho)$-normal form $f_0$ defined in \eqref{eq:abo}. Then, for every $0<\theta<2\pi$, there exists a constant $C_\theta>0$, such that, for all subsectors $W_\theta^j\subset V_j^{\pm}$ of opening $0<\theta<2\pi$, $j\in\mathbb Z$, it holds that:
\begin{equation}\label{eq:totreba}
|\Psi_j^{\pm}(z)-\Psi_\mathrm{nf}(z)|\leq C_\theta \boldsymbol\ell(|z|),\ z\in W_\theta^j\subset V_j^{\pm}.
\end{equation}
\end{lem}

\begin{proof}
The proof is divided in two steps. In \emph{Step 1}, we show a uniform bound on $|\Psi_\mathrm{nf}(z)|$ on a standard quadratic (linear) domain. In \emph{Step 2}, using this bound, we prove \eqref{eq:totreba}.
\smallskip

\emph{Step 1.} Using the explicit form \eqref{eq:abo} of $\Psi_\mathrm{nf}$, we prove that there exists $C>0$ such that:
\begin{equation}\label{eq:totreba1}
|\Psi_\mathrm{nf}(z)|\leq C |z^{-1}\boldsymbol\ell^{-m}|,\ z\in\mathcal R_c.
\end{equation}
In the course of the proof, we will pass to a smaller standard quadratic subdomain whenever needed, because we work with \emph{germs}. Note that, for every $(\alpha,m) \prec (\beta,k)$, there exists a constant $C$ and a sufficiently small standard quadratic domain $\mathcal R_c$ such that $|z^\beta\boldsymbol\ell^k|\leq C|z^\alpha\boldsymbol\ell^m|,\ z\in\mathcal R_c$. Note also that this is not the case for the whole Riemann surface of the logarithm of sufficiently small radius. 

\noindent By two partial integrations, we get, up to a constant term:
\begin{align}\label{eq:partial}
\big|\Psi_\mathrm{nf}&(z)-z^{-1}\boldsymbol\ell^{-m}+mz^{-1}\boldsymbol\ell^{-m+1}\big|=\\
&=\Big|-m(m-1)\int_{z_0}^{z}z^{-2}\boldsymbol\ell^{-m+2}+\log z-(\frac{m}{2}+\rho) \log(-\log z)\Big|\leq \nonumber\\
&\leq |m|\,|m-1|\cdot |G(z)-G(z_0)|+|\boldsymbol\ell^{-1}|+\Big|\frac{m}{2}+\rho\Big|\cdot \big|\log(-\log z)\big|\leq\nonumber\\
&\leq C\Big(|G(z)|+|\boldsymbol\ell^{-1}|+|\boldsymbol\ell_2^{-1}|\Big),\ z\in\mathcal R_c.\nonumber
\end{align}
Here, $z_0\in\mathcal R_c$ is fixed, and $G(z)$ denotes the primitive function, such that $G'(z)=z^{-2}\boldsymbol\ell^{-m+2}$. We prove now that there exists a constant $d>0$ such that:
\begin{equation*}\label{eq:prove}
|G(z)|\leq d\,|z^{-1}\boldsymbol\ell^{-m+2}|,\ z\in\mathcal R_c.
\end{equation*}
We pass to the logarithmic chart $\zeta=-\log z$ and put $H(\zeta):=G(e^{-\zeta})$. Then $H'(\zeta)=-e^{\zeta}\zeta^{m-2}$. Let $\zeta_0:=-\log z_0$ be fixed. We may take e.g. $\zeta_0\in\mathbb R_+$. Let $\gamma_\zeta$ be the rectangular path from $\zeta_0$ to $\zeta$, $\zeta\in\widetilde{\mathcal R}_c$, consisting of horizontal segment $[\zeta_0,\zeta_1]$ and vertical segment $[\zeta_1,\zeta]$. Then:
$$
H(\zeta)-H(\zeta_0)=\int_{\gamma_\zeta} H'(\eta) d\eta,\ \zeta\in\widetilde{\mathcal R}_c.
$$
Evidently, the integral depends only on $\zeta_0$ and $\zeta$, and not on the integration path, since $\widetilde{\mathcal R}_c$ is simply connected. We integrate partially $r-2$ times, where $r$ is such that $m-r<0$, and get:
\begin{align}\begin{split}\label{eq:taa}
&|G(z)-G(z_0)|=|H(\zeta)-H(\zeta_0)|=\\
&=\Big|\int_{\zeta_0}^{\zeta_1} e^\eta \eta^{m-r} d\eta+\int_{\zeta_1}^{\zeta}e^\eta \eta^{m-r} d\eta+ c_{m-2}e^{\zeta} \zeta^{m-2}+\dots +c_{m-r+1}e^\zeta \zeta^{m-r+1}-H(\zeta_0)\Big|\\
&\leq |H(\zeta_0)|+|c_{m-2}||e^{\zeta}||\zeta|^{m-2}+\ldots+|c_{m-r+1}||e^{\zeta}||\zeta|^{m-r+1}+\\
&\qquad\qquad\ \ +(\mathrm{sup}_{\eta\in[\zeta_0,\zeta_1]}|e^{\eta}||\eta|^{m-r})|\zeta_1-\zeta_0|+(\mathrm{sup}_{\eta\in[\zeta_1,\zeta]}|e^{\eta}||\eta|^{m-r})|\zeta_1-\zeta|.\end{split}
\end{align}
We now bound the remainder, using $m-r<0$:
\begin{align}\begin{split}\label{eq:ta}
(\mathrm{sup}_{\eta\in[\zeta_0,\zeta_1]}&|e^{\eta}||\eta|^{m-r})|\zeta_1-\zeta_0|+(\mathrm{sup}_{\eta\in[\zeta_1,\zeta]}|e^{\eta}||\eta|^{m-r})|\zeta_1-\zeta|\leq\\
&\leq (\mathrm{sup}_{\eta\in\gamma_\zeta}e^{\mathrm{Re}(\eta)}\mathrm{Re}(\eta)^{m-r})\cdot |\zeta-\zeta_0|\leq \\
&\leq Ce^{\mathrm{Re}(\zeta)}\mathrm{Re}(\zeta)^{m-r}|\zeta|\leq C|z|^{-1}\boldsymbol\ell(|z|)^{-(m-r)}|-\log z|,\\
&\leq D |z|^{-1}|\boldsymbol\ell|^{-\frac{m-r}{2}-1},\ \zeta\in\widetilde{\mathcal{R}}_{c'},\ z\in\mathcal R_{c'}.
\end{split}\end{align}
Here, $\mathcal R_{c'}$ is a standard quadratic subdomain such that $\mathrm{Re}(\zeta)>\mathrm{Re}(\zeta_0)$, $c>0,\,D>0$, and $c_{m-2},\ldots,c_{m-r+1}\in\mathbb C$. Indeed, note that $|\zeta_1-\zeta_0|\leq|\zeta-\zeta_0|$ and $|\zeta_1-\zeta|\leq|\zeta-\zeta_0|$, that $x\mapsto e^{x}x^{m-r}$ is an increasing function as for $x\in\mathbb R_+$ sufficiently big and $\mathrm{Re}(\zeta)\geq \mathrm{Re}(\eta),\ \eta\in\gamma_\zeta$, $\zeta\in\mathcal R_{c'}$.

The last inequality follows from the fact that $\mathcal R_{c'}$ is a standard quadratic domain. Therefore, for $z\in\mathcal R_{c'}$, it holds that $|\log z|^2=\log^2|z|+\mathrm{Arg}(z)^2$. Moreover, there exists $d>0$ such that $|\mathrm{Arg}(z)|\leq d\log^2 |z|$, $z\in\mathcal R_{c'}$. Therefore we get that there exists $d_1>0$ such that the following inequalities hold:
\begin{equation}\label{eq:ineq}
\boldsymbol\ell(|z|)\leq d_1 |\boldsymbol\ell|^{1/2},\ |\boldsymbol\ell|\leq \boldsymbol\ell(|z|),\ z\in\mathcal R_{c'},
\end{equation}
for some $d_1>0$. For a standard linear domain, there exists $d>0$ such that $|\mathrm{Arg}(z)|\leq d (-\log |z|),\ z\in\mathcal R_{a,b}$, and we get similar bounds as in \eqref{eq:ineq} and proceed similarly.

By \eqref{eq:taa} and \eqref{eq:ta}, for $r\in\mathbb N$ sufficiently big, such that $-\frac{m-r}{2}-1>-m+1$, there exist constants $C_1,\, D>0$ such that:
\begin{align}
&|G(z)-c_{m-2} z^{-1}\boldsymbol\ell^{-m+2}-\ldots-c_{m+r-1}z^{-1}\boldsymbol\ell^{-m+r-1}|\leq D|z|^{-1}|\boldsymbol\ell|^{-\frac{m-r}{2}-1},\nonumber\\
&|G(z)|\leq C_1 |z^{-1}\boldsymbol\ell^{-m+2}|,\ z\in\mathcal R_{c'}.\label{eq:tooo}
\end{align}
Here, the last inequality \eqref{eq:tooo}, and then \eqref{eq:totreba1} from \eqref{eq:partial} and \eqref{eq:tooo}, follow by the comment on the lexicographic order of power-logarithmic monomials on standard quadratic or standard linear  domain at the beginning of \emph{Step 1}.
\smallskip

\emph{Step 2.} We prove \eqref{eq:totreba} using \eqref{eq:totreba1} proven in \emph{Step 1.} We repeat the construction of the Fatou coordinates for $f$ on petals, described in detail in \cite{MRRZ2Fatou} and in \cite[Section 8]{prvi}, but deducing the uniform bounds. Consider the Abel equation for $f$:
$$
\Psi_j^{\pm}(f(z))-\Psi_j^{\pm}(z)=1,\ z\in V_j^{\pm}.
$$
Denote by $R_j^{\pm}=\Psi_j^{\pm}-\Psi_\mathrm{nf}$ on $V_j^{\pm}$. The Abel equation becomes:
$$
R_j^{\pm}(f(z))-R_j^{\pm}(z)=1-(\Psi_\mathrm{nf}(f(z))-\Psi_\mathrm{nf}(z)),\ z\in V_j^{\pm}.
$$
Denote by $\delta(z):=1-(\Psi_\mathrm{nf}(f(z))-\Psi_\mathrm{nf}(z))$. It is an analytic function on $\mathcal R_c$. Let $h(z)=f(z)-f_0(z)$. Then, by uniform bound \eqref{eq:fuin}, $|h(z)|=O(z^3\boldsymbol\ell^{2m+2})$, uniformly as $z\to 0$ on $\mathcal R_c$. We compute:
\begin{align*}
|\delta(z)|&=|1-(\Psi_{\mathrm{nf}}(f_0(z)+h(z)))+\Psi_\mathrm{nf}(z)|=\\
&=|1-\Psi_{\mathrm{nf}}(f_0(z))-R_1(z)+\Psi_\mathrm{nf}(z)|=|R_1(z)|.
\end{align*}
Here, by Taylor's theorem, e.g. \cite{alfors}, $\Psi_\mathrm{nf}(f_0(z)+h(z))=\Psi_\mathrm{nf}(f_0(z))+R_1(z)$,
where
$$
|R_1(z)|\leq \frac{M(z)|h(z)|}{\rho-|h(z)|},\text{ for } z\in\mathcal R_c \ \text{such that\ } |h(z)|<\frac{\rho}{2}.
$$
Here, $M(z):=\max_{\xi\in\partial B(f_0(z),\rho)} |\Psi_\mathrm{nf}(\xi)|$. For $z\in\mathcal R_c$, put $\rho(z):=\frac{|f_0(z)|}{4}>0$.  We now take $r_0>0$ such that $|z|<r_0$ implies $|h(z)|<\frac{\rho(z)}{2}$. Indeed, by the uniform bound \eqref{eq:fuin}, there exists $r>0$, such that $|h(z)|\leq C|z^3\boldsymbol\ell^{2m+2}|\leq D|z|,\ z\in\mathcal R_c,\ |z|<r$.  
As in \emph{Step 1}., $|\Psi_{\mathrm{nf}}(z)|\leq E|z^{-1}\boldsymbol\ell^{-m}|$, $E>0$, $z\in\mathcal R_c.$ By uniform bound \eqref{eq:fuin} on $f_0$, it follows\footnote{Write e.g. $\xi_\theta=|f_0(z)|e^{i\mathrm{Arg}(f_0(z))}+\frac{|f_0(z)|}{4}\cdot (\cos\theta+i\sin\theta)$, $\theta\in[0,2\pi)$, $z\in\mathcal R_c$, and $f_0(z)=|z|e^{i\mathrm{Arg}(z)}+O(z^{1+\varepsilon})\cdot (\cos\theta_1+i\sin\theta_1), \ \theta_1\in[0,2\pi)$, $\varepsilon>0$, with $O(z^{1+\varepsilon})$ uniform on $\mathcal R_{c}$.} that there exist constants $C_i>0,\ D_i>0,\ i=1,\ldots,4,$ such that, for $z\in\mathcal R_{c},\ \xi\in \partial B(f_0(z),\frac{|f_0(z)|}{4})$, it holds:
\begin{align*}
C_1&|\xi|\leq C_2|z|\leq C_3|f_0(z)|\leq C_4|\xi|,\\
D_1&\mathrm{Arg}(\xi)\leq D_2\mathrm{Arg}(z)\leq D_3\mathrm{Arg}(f_0(z))\leq D_4\mathrm{Arg}(\xi).
\end{align*}
Therefore, there exists a constant $K>0$ such that $|\xi^{-1}\boldsymbol\ell(\xi)^{-m}|\leq K|z^{-1}\boldsymbol\ell^{-m}|,$ $z\in\mathcal R_c,\ \xi\in \partial B(f_0(z),\frac{|f_0(z)|}{4})$. Hence, $M(z)\leq d|z^{-1}\boldsymbol\ell^{-m}|,\ z\in\mathcal R_c$, for some constant $d>0$. Finally,
$$
|\delta(z)|=|R_1(z)|\leq C|z\boldsymbol\ell^{m+2}|,\ z\in\mathcal R_c,\ C>0.
$$

Now, iterating the equation $R_j^+(f(z))-R_j^+(z)=\delta (z)$ on each petal $V_j^+$ (on repelling petals $V_j^-$ we consider the inverse $f^{-1}$),
we get the series:
$$
R_j^+(z)=-\sum_k \delta(f^{\circ k}(z)),\ z\in V_j^+,
$$
uniformly convergent on compact subsets of the petal (see \cite{MRRZ2Fatou}). 
Note that here $|\delta(f^{\circ k}(z))|\leq c \big|f^{\circ k}(z)\boldsymbol\ell(f^{\circ k}(z))^{m+2}\big|,\ z\in\mathcal R_c,$ holds uniformly on petals. On the other hand, directly as in \cite[Section 8]{prvi}, due to the bound \eqref{eq:fuin} of $f$, the bound on 
$
|f^{\circ k}(z)|
$ is deduced
uniformly in $j\in\mathbb Z$ on subsectors $W_\theta^j\subset V_j^+$ of the same opening $\theta\in(0,2\pi)$. Finally, applying \cite[Proposition 8.3]{prvi}, and using the existence of uniform bounds for $|f^{\circ k}(z)|$ and for $|\delta(z)|$ by levels, we get that there exists $K_\theta>0$, independent of $j\in\mathbb Z$, such that, for every subsector $W_\theta^j\subset V_j^+$ of opening $0<\theta<2\pi$, it holds:
$$
|R_j^+(z)|\leq K_\theta\cdot \boldsymbol\ell(|z|), \ z\in W_\theta^j\subset V_j^+.
$$
We repeat the procedure similarly for repelling petals $V_j^-$, $j\in\mathbb Z$, and take the maximum of the two constants.
\end{proof}

\noindent \emph{Proof of Proposition~\ref{prop:modest}}.\ 

Let $f$ be prenormalized and let the uniform bound \eqref{eq:fuin} hold. Let $\Psi_\mathrm{nf}(z)$, $z\in\mathcal R_c$, be the Fatou coordinate of the formal $(2,m,\rho)$-normal form $f_0$, defined in \eqref{eq:abo}. By Lemma~\ref{lem:modest1}, for the Fatou coordinate of $f$, the following uniform bound holds:
\begin{equation*}
|\Psi_j^{\pm}(z)-\Psi_\mathrm{nf}(z)|\leq C_\theta \boldsymbol\ell(|z|),\ z\in W_\theta\subset V_j^{\pm},
\end{equation*}
where $W_\theta\subset V_j^\pm$ are subsectors of opening $0<\theta<2\pi$, and $C_\theta>0$ is \emph{uniform} for all $j\in\mathbb Z$. On standard quadratic domains, there exists $a>0$ such that\footnote{On a standard quadratic domain, the following bound holds: $$|z^\varepsilon(\log z)^{m}|\leq |z|^\varepsilon\big(\sqrt{\log^2 |z|+\varphi^2}\big)^m\leq C |z|^{\varepsilon}\log^{2m} |z|,$$ $\varphi=\textrm{Arg}(z)$, $\varepsilon>0,\ m\in\mathbb Z,$ since $|\varphi|<\log^2 |z|$ and $|\varphi|$ cannot increase to $+\infty$ uncontrolled by $|z|$.} $|\boldsymbol\ell|\leq \boldsymbol\ell(|z|)\leq a \sqrt{|\boldsymbol\ell|}$. On standard linear domains, there exists $a>0$ such that $|\boldsymbol\ell|\leq \boldsymbol\ell(|z|)\leq a |\boldsymbol\ell|$. Therefore,
\begin{equation*}
|\Psi_j^{\pm}(z)-\Psi_\mathrm{nf}(z)|\leq C_\theta \sqrt{|\boldsymbol\ell|},\ z\in W_\theta\subset V_j^{\pm},\ j\in\mathbb Z.
\end{equation*}
Let us estimate the horn maps of $f$ from \eqref{eq:moduli}:
\begin{align*}
&h_0^j(t):=e^{-2\pi i \Psi_+^{j-1}\circ (\Psi_{-}^{j})^{-1}(-\frac{\log t}{2\pi i})},\ t\approx 0\nonumber\\
&h_\infty^j(t):=e^{2\pi i \Psi_-^j\circ (\Psi_{+}^j)^{-1}(\frac{\log t}{2\pi i})},\ t\approx 0,\ j\in\mathbb Z.
\end{align*}
By uniform bound \eqref{eq:totreba} on $\Psi_j^{\pm}$ (that is, by its prenormalized form $\Psi_j^{\pm}(z)=\Psi_\mathrm{nf}(z)+R_j^{\pm}(z)$, $R_j^\pm=o(1),\ z\to 0,\ z\in W_\theta\subset V_j^{\pm}$ uniformly in $j$), we compute:
\begin{equation}\label{eq:eqqq}
\Psi_+^{j-1}\circ (\Psi_{-}^{j})^{-1}(w)=w+o(1),
\end{equation}
where $o(1)$ is uniform in $j$ as $\mathrm{Im}(w)\to \pm\infty$ in $\Psi_\mp^j(W_\theta)$. Since the spaces of orbits of both positive and negative petal $V_+^{j-1}$ and $V_-^j$ are contained in every sector around the centerline of $V_{0}^j$, then \eqref{eq:eqqq} implies: 
\begin{equation*}\label{eq:hj}
h_0^j(t)=t\big(1+o(1)\big),\ t\to 0,
\end{equation*} 
uniformly in $j\in\mathbb Z$. Since $h_0^j$ are parabolic analytic diffeomorphisms, for $\delta>0$ and for every $j\in\mathbb Z$, there exist constants $c_j>0$, $j\in\mathbb Z$, such that
\begin{equation}\label{eq:hj1}
|h_0^j(t)-t|\leq c_j |t|^2,\ |t|<\delta.
\end{equation}
Let us take here $c_j:=\sup_{|t|<\delta}\frac{|h_0^j(t)-t|}{|t|^2}=\sup_{|t|<\delta}\frac{|o(t)|}{|t|}\frac{1}{|t|}$. Since $o(t)$ is uniform in $j$, $(c_j)_j$ is bounded from above, and from \eqref{eq:hj1} it follows that:
$$
|h_0^j(t)-t|=O(t^2),\ |t|\to 0,
$$
where $O(.)$ is uniform in $j\in\mathbb Z$. The same analysis is repeated for $h_\infty^j(t)$, $j\in\mathbb Z$. 
\hfill $\Box$

\subsection{Proof of Lemma~\ref{prop:conv}}

\begin{proof}
We prove the uniform convergence of the series \eqref{eq:FJ} in definition of $\tilde R_j^\pm$ on compacts in $\tilde V_j^{\pm}$, hence analyticity of $\tilde R_j^\pm$ on $\tilde V_j^\pm$ follows by the Weierstrass theorem.

Let us fix $j\in\mathbb Z$. Take, for example, $\tilde R_j^+$ on $\tilde V_j^+$. It suffices to show the uniform convergence on compact subsets of $\tilde V_j^+$ of $\sum_{k=j+1}^{+\infty}\tilde F_{0,k}^-$. The convergence of the other three terms in the sum for $\tilde R_j^+$ follows analogously. Let $\tilde K\subset \tilde V_j^+$ be a compact \emph{substrip} of $\tilde V_j^+$ (i.e. the image in the logarithmic chart of the closed subsector $K\subset V_j^+$ in the original $z$-chart). Let $(\mathcal C_0^{j+1})'$ be the line at height $\theta'$ in $\tilde V_0^{j+1}$ such that $\tilde K$ is completely contained in part of $\tilde V_j^+$ up to the line $(\mathcal C_0^{j+1})'$.
Let us analyze the series \eqref{eq:FJ} for $\zeta\in \tilde K$, using \eqref{eq:ch} and the fact that two Cauchy-Heine integrals along different lines $\mathcal C_0^{j+1}$ and $(\mathcal C_0^{j+1})'$ in $\tilde V_0^{j+1}$ differ by an analytic germ at $\zeta=\infty$:
\begin{align*}
\sum_{k=j+1}^{+\infty}\tilde F_{0,k}^-(\zeta)=\frac{1}{2\pi i}\int_{(\mathcal C_0^{j+1})'}\frac{\tilde G_0^{j+1}(w)}{w-\zeta} dw+\tilde \chi_0^{j+1}(\zeta)+\frac{1}{2\pi i}\sum_{k=j+2}^{+\infty}\Big(\int_{\mathcal C_0^k}    &\frac{\tilde G_0^k(w)}{w-\zeta} dw\Big),\\
& \ \zeta\in \tilde K,
\end{align*}
see Figure~\ref{fig:indi}.
Here, $$\tilde \chi_0^{j+1}(\zeta):=\int_{\mathcal S_0^{j+1}}\    \frac{\tilde G_0^{j+1}(w)}{w-\zeta} dw,\ \zeta\in \tilde K,$$ is an analytic function for $\zeta\in\tilde K$ and at $\zeta=\infty$, as explained before, which depends on the chosen height $\theta'$, that is, on $\tilde K$. Indeed, the integration is done along the boundary arc $\mathcal S_0^{j+1}$ of $\tilde V_0^{j+1}$ between heights corresponding to lines $\mathcal C_0^{j+1}$ and $(\mathcal C_0^{j+1})'$, where subintegral function has no singularities for $w\in \tilde K$. Indeed, we can always restrict to a \emph{smaller} standard quadratic domain. 

\smallskip

\noindent It suffices to show the uniform convergence on $\tilde K$ of $\sum_{k=j+2}^{+\infty}\Big(\int_{\mathcal C_0^k}    \frac{\tilde G_0^k(w)}{w-\zeta} dw\Big)$. In the following computation, we assume the lines of integration $\mathcal C_0^k$ along a standard quadratic domain; thus $\tilde V_j^{\pm}$ are covering a standard quadratic domain. Even sharper estimates for convergence can be repeated for a standard linear domain. By \eqref{eq:b}, we have the following bounds:
\begin{align}\label{eq:bd}
\Big| \int_{\mathcal C_0^k}\frac{\tilde G_0^k(w)}{w-\zeta}dw\Big|=&\Big| \int_{-\log r_k + i(4k-3)\frac{\pi}{2}}^{+\infty+i(4k-3)\frac{\pi}{2}}\frac{\tilde G_0^k(w)}{w-\zeta}dw\Big|=\Big|t=w-i(4k-3)\frac{\pi}{2}\Big|\\
&\leq \int_{\sqrt k}^{+\infty}\frac{\Big|\tilde G_0^k\big(t+i(4k-3)\frac{\pi}{2}\big)\Big|}{|w-\zeta|}dt\leq \frac{1}{b}\int_{\sqrt k}^{+\infty}\Big|\tilde G_0^k\big(t+i(4k-3)\frac{\pi}{2}\big)\Big|dt\nonumber\\
&\leq \frac{1}{b}\int_{\sqrt k}^{+\infty}Ce^{-Me^{m\big|t+i(4k-3)\frac{\pi}{2}\big|}}dt\nonumber\\
&\leq \frac{C}{b}\int_{\sqrt k}^{+\infty}e^{-Me^{mt}}dt=\frac{C}{b}\int_{\sqrt k}^{+\infty}e^{-Me^{mt}}\cdot \frac{e^{mt}}{e^{mt}}dt\nonumber\\
&\leq \frac{C}{be^{m\sqrt k}}\int_{\sqrt k}^{+\infty}e^{-Me^{mt}}\cdot e^{mt} dt= C_1 e^{-m\sqrt k}e^{-Me^{m\sqrt k}}.\nonumber
\end{align}
Indeed, for every $k>j+1$, $\mathcal C_0^k$ is on some (uniformly) bounded distance from $\tilde K$ in the logarithmic chart. That is, for every $\zeta\in \tilde K$ and every $\mathcal C_0^k$, $k>j+1$, $|\zeta-w|>b$, where $b>0$ is independent of $k$. Also note that, by \eqref{eq:imam}, we have at least $-\log r_k\sim \sqrt{k}$, $k\to\infty$. Since $t>0$, $\big|t+i(4k-3)\frac{\pi}{2}\big|\geq t$.

\noindent Now the convergence of the series
$\sum_k e^{-m\sqrt k}e^{-Me^{m\sqrt k}}$, for $m>0,\ M>0,$ proves the uniform convergence of the above series on $\tilde K$. 
\smallskip

Once we have proven that $\tilde R_j^\pm$ are analytic on $\tilde V_j^\pm$, by \eqref{eq:ahhh} and \eqref{eq:FJ} we get \eqref{eq:dobi}.
\end{proof}
\smallskip

\subsection{Proof of Lemma~\ref{prop:asylin}}\
\smallskip

The proof is an adaptation of the proof in \cite{loray2} for the simpler case of holomorphic germs. Let us fix $j\in\mathbb Z$ and $\tilde V_j^+$ ($\tilde V_j^-$ is treated analogously), and let us choose a fixed horizontal substrip $\tilde U\subset \tilde V_j^+$. By \eqref{eq:FJ}, $\tilde R_j^+$ on $\tilde U$ is a sum of countably many Cauchy-Heine integrals. If lines $\mathcal C_{\infty}^j$ and $\mathcal C_0^{j+1}$ that lie in the petal $\tilde V_j^+$ intersect the strip $\tilde U$, the integration is done along the shifted lines $(\mathcal C_{\infty}^j)',\ (\mathcal C_0^{j+1})'$ at some bounded distance from $\tilde U$, whereas error terms $\tilde \chi_\infty^j(\zeta)$, $\tilde \chi_0^{j+1}(\zeta)$ (integrals along parts of the boundary $\mathcal S_\infty^{j},\ \mathcal S_{0}^{j+1}$, as in e.g. proof of Lemma~\ref{prop:cah}) are added. They depend on  $\tilde U$, that is, on the choice of lines $(\mathcal C_{\infty}^j)',\ (\mathcal C_0^{j+1})'$. They are  analytic at infinity, so they expand in Taylor series $\widehat \chi_\infty^{j},\ \widehat \chi_0^{j+1}\in\mathbb C[[\zeta^{-1}]]$. In particular, germs analytic at $0$ admit $\log$-Gevrey asymptotic expansion of every order, see Definition~\ref{def:logg}.

\noindent We divide the proof in the following three steps. Note that Steps 1. and 2. are independent of the type of the domain (standard quadratic or standard linear).

\emph{Step 1.} We prove that each integral $\int_{\mathcal C_{0,\infty}^k}\frac{\tilde G_{0,\infty}^k(w)}{w-\zeta}dw$, $k\in\mathbb Z$, from the series \eqref{eq:FJ}, on its domain of analiticity admits an asymptotic expansion in $\mathbb C[[\zeta^{-1}]]$, as $\mathrm{Re}(\zeta)\to+\infty$.

\emph{Step 2.} It is sufficient to treat any of the eight sums in \eqref{eq:FJ}, since others are treated analogously. Therefore, we choose one of the sums: \begin{equation}\label{eq:sumi}\sum_{k\geq j+1}\tilde F_{0,k}^-(\zeta)=\sum_{k\geq j+1}\int_{\mathcal C_0^k}\frac{\tilde G_0^k(w)}{w-\zeta}dw.\end{equation} By Step (1), for every $k\geq j+1$, $\int_{\mathcal C_0^k}\frac{\tilde G_0^k(w)}{w-\zeta}dw$ admits an asymptotic expansion in $\mathbb C[[\zeta^{-1}]]$, as $\mathrm{Re}(\zeta)\to +\infty$. By appropriate bounds on partial sums of \eqref{eq:sumi}, we prove the \emph{convergence} of coefficients in front of each monomial $\zeta^{-n}$, $n\in\mathbb N$, in \eqref{eq:sumi}, and thus prove the existence of the asymptotic expansion of \eqref{eq:sumi} in $\mathbb C[[\zeta^{-1}]]$. We also prove statement $(1)$ of the lemma.

\emph{Step 3.} In the case of construction on a \emph{standard linear domain}, we prove statement (2) of the lemma: that the asymptotics of $\breve R_j^{\pm}(\boldsymbol\ell):=\tilde R_j^{\pm}(\boldsymbol\ell^{-1})$ is in addition \emph{$\log$-Gevrey of order $m$}, as $\boldsymbol\ell\to 0$ in $\boldsymbol\ell$-cusp $\boldsymbol\ell(V_j^{\pm})$. In the final Remark~\ref{rem:sq} we state the technical problem in deducing $\log$-Gevrey bounds on a standard quadratic domain.

\begin{proof}\

\emph{Step 1.} For every $k\in\mathbb Z$ and for every $n\in\mathbb N$, it holds that:
\begin{align*}
\frac{\tilde G_{0,\infty}^k(w)}{w-\zeta}&=\sum_{p=0}^{n-1}(-\tilde G_{0,\infty}^k(w)w^p)\zeta^{-p-1}+\frac{\tilde G_{0,\infty}^k(w)w^n}{w-\zeta}\zeta^{-n}.
\end{align*}
Therefore we get, for every $k\in\mathbb Z$,
\begin{equation}\label{eq:d}
\int_{\mathcal C_{0,\infty}^k}\frac{\tilde G_{0,\infty}^k(w)}{w-\zeta}dw-\sum_{p=0}^{n-1}a_p^k \zeta^{-p-1}=\zeta^{-n}\int_{\mathcal C_{0,\infty}^k}\frac{\tilde G_{0,\infty}^k(w)w^n}{w-\zeta}dw,\ n\in\mathbb N,
\end{equation}                  
where coefficients $a_p^k$ are given by:
\begin{equation}\label{eq:sec}
a_p^k=-\int_{\mathcal C_{0,\infty}^k} \tilde G_{0,\infty}^k(w)w^p dw.
\end{equation}
Due to (even superexponential) flatness of $\tilde G_{0,\infty}^k(w)$ as $\text{Re}(w)\to\infty$ given in \eqref{eq:b}, the integrals in \eqref{eq:sec} converge. The same holds for integrals $\int_{\mathcal C_{0,\infty}^k}\frac{\tilde G_{0,\infty}^k(w)w^n}{w-\zeta}dw$ for $\zeta$ on some bounded distance from the integration line.
\medskip

\emph{Step 2.} To prove convergence of partial sums of \eqref{eq:sumi}, let us take formula \eqref{eq:d} for $k\in\{j+1,\ldots,N\},\ N\in\mathbb N,\ N\geq j+1$, and make the sum of these. For $\zeta\in\tilde U$, where $\tilde U$ is a fixed horizontal substrip of $\tilde V_j^+$, there exists $b>0$ such that $|\zeta-w|>b$, $\zeta\in \mathcal C_0^k$, uniformly for every $k> j+1$. Now, very similar bounds as in \eqref{eq:bd} in the proof of Lemma~\ref{prop:conv} performed on the right-hand side of \eqref{eq:d} and on \eqref{eq:sec}  give us the convergence of $\sum_{k\geq j+1}^N a_p^k$, as $N\to\infty$,  and an uniform bound on $\tilde U$ on the remainder $\sum_{k\geq j+1}^N\int_{\mathcal C_0^k}\frac{\tilde G_0^k(w)w^n}{w-\zeta}dw$, as $N\to\infty$. Let us now denote by $a_p\in\mathbb R$, $p\in\mathbb N$, the limit $a_p:=\sum_{k\geq j+1}a_p^k$. We get the asymptotic expansion: \begin{equation}\label{eq:ekspan}\sum_{k\geq j+1}\int_{\mathcal C_0^k}\frac{\tilde G_0^k(w)}{w-\zeta}dw\sim\sum_{p=0}^{+\infty} a_p \zeta^{-p-1},\ \text{Re}(\zeta)\to\infty,\ \zeta\in \tilde U\subseteq \tilde V_j^+.\end{equation}
\smallskip
              
Let us now prove statement $(1)$ of the lemma about the uniform bound. Note that all bounds on the remainders: \begin{align}\label{eq:jj}\sum_{k\geq j+1}\int_{\mathcal C_{0,\infty}^k}\frac{\tilde G_{0,\infty}^k(w)w^n}{w-\zeta}dw+\sum_{k\leq j} \int_{\mathcal C_{0,\infty}^k}\frac{\tilde G_{0,\infty}^k(w)w^n}{w-\zeta}dw\end{align} from \eqref{eq:d} can be made uniform in $j\in\mathbb Z$ and $\zeta\in\tilde U_j\subset \tilde V_j^+$, where $\tilde U_j$ are strips of the same width for all $j\in\mathbb Z$, due to uniform estimate \eqref{eq:b} of $\tilde G_{0,\infty}^j$, $j\in\mathbb Z$. In particular, for $n=1$. We conclude here similarly as in the proof of convergence \eqref{eq:bd} in the proof of Lemma~\ref{prop:conv}. In fact, in \eqref{eq:jj}, for every $j\in\mathbb Z$ and $\tilde U\subset \tilde V_j^+$, exactly two lines of integration $\mathcal C_0^{j+1}$ and $\mathcal C_\infty^j$, are changed to shifted lines $(\mathcal C_0^{j+1})'$ and $(\mathcal C_\infty^j)'$, connected to previous ones by the boundary arcs $\mathcal S_0^{j+1}$ and $\mathcal S_\infty^j$, and at uniform (in $j$) distance from them. But \eqref{eq:d} with $n=1$ and applied to border lines $\mathcal S_0^{j+1},\ \mathcal S_\infty^j$ gives similarly:
$$
\int_{\mathcal S_0^{j+1}}\frac{\tilde G_0^{j+1}(w)}{w-\zeta}dw-b_0^{j+1} \zeta^{-1}=\zeta^{-1}\int_{\mathcal S_0^{j+1}}\frac{\tilde G_0^{j+1}(w)w}{w-\zeta}dw,\ b_0^{j+1}\in\mathbb C.
$$

\noindent Take $\varepsilon>0$ small. We find a subdomain (quadratic resp. linear) $\widetilde{\mathcal R}_{C'}\subset\widetilde{\mathcal R}_C$ such that:
\begin{equation}\label{eq:domena}
\widetilde{\mathcal{R}}_{C'}\subseteq \{\zeta\in\widetilde{\mathcal R}_C:\ d(\zeta,\partial\widetilde{\mathcal R}_C)>\varepsilon\}.
\end{equation}
For $\zeta\in\widetilde{\mathcal R}_{C'}$ it therefore holds that $|\zeta-w|>\varepsilon$, $w\in \mathcal S_{0,\infty}^j$, uniformly in $j\in\mathbb Z$. Since $\mathcal S_{0,\infty}^j$ are bounded arcs connecting at most $w=\sqrt j+ij$ and $w=\sqrt{j+1}+i(j+1)$, and $\tilde G_{0,\infty}^j$ are uniformly (in $j$) super-exponentially small, the bound on the remainder $$\Big|\int_{\mathcal S_\infty^j}\frac{\tilde G_{\infty}^{j}(w)w}{w-\zeta}dw+\int_{\mathcal S_0^{j+1}}\frac{\tilde G_{0}^{j+1}(w)w}{w-\zeta}dw\Big|$$ can be made uniform in $j\in \mathbb Z$, for $\zeta\in \tilde U_j\subset \tilde V_j^+\cap\widetilde{\mathcal R}_{C'}$. This proves the statement (1).
\medskip




\emph{Step 3.} We prove, on standard linear domains, the \emph{$\log$-Gevrey bounds of order $m$} for the expansion \eqref{eq:ekspan}. 

The lines of integration $\mathcal C_{0,\infty}^k$ in Cauchy-Heine integrals on a standard linear domain in the logarithmic chart are, by \eqref{eg:lint} and \eqref{eq:imam1}, the half-lines: 
\begin{align*}
&\mathcal C_{0}^k\ldots\,\big[\sim k+i(4k-3)\frac{\pi}{2},+\infty+i(4k-3)\frac{\pi}{2}\big),\\
&\mathcal C_{\infty}^k\ldots\,\big[\sim k+i(4k-1)\frac{\pi}{2},+\infty+i(4k-1)\frac{\pi}{2}\big),\ 
k\in\mathbb Z.
\end{align*} 
Let $j\in\mathbb Z$. On every substrip $\tilde U\subset \tilde V_j^+$ (the same analysis can be repeated for $\tilde V_j^-$), by \eqref{eq:d}, it holds that:
\begin{align}\label{beg}
\Big|&\sum_{k=j+2}^{\infty}\int_{\mathcal C_0^k}\frac{\tilde G_0^k(w)}{w-\zeta}dw-\sum_{p=0}^{n-1}a_p \zeta^{-p-1}\Big|=\Big|\zeta^{-n}\sum_{k=j+2}^{\infty}\int_{\mathcal C_0^k}\frac{\tilde G_0^k(w)w^n}{w-\zeta}dw \Big|\leq \\
&\leq |\zeta|^{-n} \Big|\sum_{k=j+2}^{\infty}\int_{\mathcal C_0^k} \frac{\tilde G_0^k(w)^{1/2}\cdot \tilde G_0^k(w)^{1/2}\cdot w^n}{w-\zeta}dw\Big|=\nonumber \\
&=\Big| \xi=\mathrm{Re}(w)\ \Rightarrow \xi=w-i(4k-3)\frac{\pi}{2} \Big|=\nonumber \\
&\leq |\zeta|^{-n} \sum_{k=j+2}^{\infty} \int_{-\log r_k\sim k}^{+\infty}\frac{|\tilde G_0^k(w)|^{1/2}}{|\zeta-w|}\cdot |\tilde G_0^k(w)|^{1/2}\big|\xi+i(4k-3)\frac{\pi}{2}\big|^n\,d\xi,\ \zeta\in\tilde U.\nonumber
\end{align}
By \eqref{eq:b}, $\tilde G_{0,\infty}^k(w)$ are superexponentially small on lines $\mathcal C_{0,\infty}^k$, and moreover uniformly in $k\in\mathbb Z$. That is, there exist constants $C,\ M>0$, independent of $k\in\mathbb Z$, such that
$$
|\tilde G_{0,\infty}^k(w)^{1/2}|\leq C e^{-M e^{m\mathrm{Re}(w)}},\ w\in\mathcal C_{0,\infty}^k. 
$$
Thus, in \eqref{beg}, by direct integration, we get that: \begin{equation}\label{eq:of}\sum_{k=j+2}^{\infty}\int_{-\log r_k\sim k}^{+\infty}\frac{|\tilde G_0^k(w)^{1/2}|}{|w-\zeta|}dw \leq D_j.\end{equation} Let us now bound
\begin{equation}\label{eq:dai}
|\tilde G_0^k(w)|^{1/2}\big|\xi+i(4k-3)\frac{\pi}{2}\big|^n\leq Ce^{-M e^{m\xi}}(\sqrt{\xi^2+k^2})^n\leq D e^{-M e^{m\xi}} \xi^{n},\ D>0,
\end{equation}
for $\xi=\mathrm{Re}(w)\in (-\log r_k,+\infty)\sim (k,+\infty)$. We sometimes omit constants for simplicity (where they don't influence the type of the final result). The last inequality is the consequence of the fact that lines $\mathcal C_0^k$ lie in a standard linear domain. Therefore, for $w\in \mathcal C_0^k$, it holds that $\xi=\mathrm{Re}(w)>\mathrm{Im}(w)\sim k$. 

Similarly, we estimate the term $$|\zeta|^n\cdot \big|\int_{\mathcal S_0^{j+1}} \frac{\tilde G_0^{j+1}(w)}{w-\zeta}dw-\sum_{p=0}^{\infty}b_p \zeta^{-p-1}\big|.$$We get similar bounds as \eqref{eq:of} and \eqref{eq:dai}, but on a subdomain $\widetilde{\mathcal R}_{a',b'}\subset\widetilde{\mathcal R}_{a,b}$, defined as in \eqref{eq:domena}.

\smallskip

Now, maximizing the function $\xi\mapsto D e^{-M e^{m\xi}} \xi^{n}$ by $\xi$, we easily get that the point of maximum is $\xi_0$ such that $e^{m\xi_0}\sim\frac{1}{M}\frac{n}{\log n}$ and $m\xi_0\sim \log n$, as $n\to\infty$. Therefore, there exists $D_1>0$ such that:
\begin{equation}\label{eq:ocj}
De^{-M e^{m\xi}}\xi^{n}\leq D_1 m^{-n} e^{-\frac{n}{\log n}}\log^{n} n ,\ \xi>0,\ n\in\mathbb N.
\end{equation}
By \eqref{eq:of},\ \eqref{eq:dai} and \eqref{eq:ocj}, from \eqref{beg} we get that there exists $c>0$ such that: 
\begin{align}\label{eq:final1}
\Big|\int_{(\mathcal C_0^{j+1})'}\frac{\tilde G_0^{j+1}(w)}{w-\zeta}dw+\sum_{k=j+2}^{\infty}&\int_{\mathcal C_0^k}\frac{\tilde G_0^k(w)}{w-\zeta}dw-\sum_{p=0}^{n-1}c_p \zeta^{-p-1}\Big|\leq \\
&\leq c m^{-n} e^{-\frac{n}{\log n}}\log^{n}n \cdot |\zeta|^{-n},\ \zeta\in \tilde U\cap \widetilde{\mathcal R}_{C'}.\nonumber
\end{align}
Here, $c_p=a_p+b_p,\ p\in\mathbb N$. The same can be concluded for the other three terms of the sum for $\tilde R_j^+$ given in \eqref{eq:FJ}. By Definition~\ref{def:logg} of $\log$-Gevrey asymptotic expansions, we conclude that $\tilde R_j^+(\zeta)$ admits $\log$-Gevrey power asymptotic expansion  of order $m>0$ in $\zeta^{-1}$, as $\mathrm{Re}(\zeta)\to+\infty$ in $\tilde V_j^+$. Thus statement $(2)$ is proven.
\end{proof}

\begin{obs}[Bounds for asymptotic expansion of $\tilde R_j^{\pm}$ on standard quadratic domains]\label{rem:sq}

On a standard quadratic domain, the lines of integration $\mathcal C_{0,\infty}^k$, $k\in\mathbb Z$, in the logarithmic chart are the half-lines:
\begin{align*}
&\mathcal C_{0}^k\ldots\,\big[\sim \sqrt k+i(4k-3)\frac{\pi}{2},+\infty+i(4k-3)\frac{\pi}{2}\big),\\
&\mathcal C_{\infty}^k\ldots\,\big[\sim \sqrt k+i(4k-1)\frac{\pi}{2},+\infty+i(4k-1)\frac{\pi}{2}\big),\ 
k\in\mathbb Z.
\end{align*} 
The other difference with respect to standard linear domains is the bound \eqref{eq:dai}. On a standard quadratic domain we have $\xi^2=\mathrm{Re}(w)^2\geq k\sim \mathrm{Im}(w)=(2k+1)\frac{\pi}{2}$, so \eqref{eq:dai} becomes:
\begin{equation*}\label{eq:da1}
|\tilde G_0^k(w)|^{1/2}\big|\xi+i(4k-3)\frac{\pi}{2}\big|^n\leq Ce^{-M e^{m\xi}}(\sqrt{\xi^2+k^2})^n\leq D e^{-M e^{m\xi}} \xi^{2n}.	
\end{equation*}
In the same way as in the proof of Lemma~\ref{prop:asylin} for standard linear domains, for a standard quadratic domain we get:
$$
De^{-M e^{m\xi}} \xi^{2n}\leq D_1 m^{-2n} e^{-\frac{2n}{\log (2n)}}\log^{2n} (2n), \ \xi>0.
$$
The final bound \eqref{eq:final1} on a standard quadratic domain is as follows:
\begin{align}\begin{split}\label{eq:final2}
\Big|\int_{(\mathcal C_0^{j+1})'}\frac{\tilde G_0^{j+1}(w)}{w-\zeta}dw+&\sum_{k=j+2}^{\infty}\int_{\mathcal C_0^k}\frac{\tilde G_0^k(w)}{w-\zeta}dw-\sum_{p=0}^{n-1}c_p \zeta^{-p-1}\Big|\leq \\
&\leq c m^{-2n} e^{-\frac{2n}{\log 2n}}\log^{2n}(2n) \cdot |\zeta|^{-n},\ \zeta\in \tilde U\cap \widetilde{\mathcal R}_{C'},\end{split}
\end{align}
where $\widetilde{\mathcal R}_{C'}\subset \widetilde{\mathcal R}_C$ is a quadratic subdomain, as in \eqref{eq:domena}, and $\tilde U\subset \tilde V_j^+$ a horizontal substrip. 

The bounds \eqref{eq:final2} obtained on standard quadratic domain are weaker than $\log$-Gevrey of order $m$, for any $m>0$. Therefore, they are too weak to attribute a unique $\log$-Gevrey sum to $\widehat R(\boldsymbol\ell)$ on $\boldsymbol\ell$-cusps $\boldsymbol\ell(V_j^\pm)$, $j\in\mathbb Z$.
\end{obs}

\medskip

\subsection{Proof of Lemma~\ref{lem:uses}}\label{subsec:ap}\
\smallskip

The following Lemma~\ref{lema:use} for uniform bounds on iterates $\tilde R_{j,\pm}^n$, $n\in\mathbb N$, is used in the proof. Lemma~\ref{lema:use} and Remark~\ref{rem:use} are also used in the proof of statement $(1)$ of Lemma~\ref{lema:konver}. In fact, in the proof of Lemma~\ref{lema:use}, we conclude inductively the bounds for every $n\in\mathbb N$, in the course of iterative construction of the sequence $\tilde R_{j,\pm}^n$ described in Lemma~\ref{lema:konver}. Therefore, the bounds in Lemma~\ref{lema:use} and Remark~\ref{rem:use} can be deduced simultaneously with the inductive construction in Lemma~\ref{lema:konver}, without apriori assuming the existence of the whole sequence.
\smallskip

Let us first introduce some notation. Let $\varepsilon>0$. As in proof of statement (1) of Lemma~\ref{lema:konver}, we denote by $\mathcal C_{0,\pm 2\varepsilon}^j$ the horizontal half-lines in the standard quadratic domain at distance $\pm 2\varepsilon$ from $\mathcal C_0^j$, and by $\mathcal C_{\infty,\pm 2\varepsilon}^j$ the horizontal half-lines in the standard quadratic domain at distance $\pm 2\varepsilon$ from $\mathcal C_\infty^j$, $j\in\mathbb Z$. By $\mathcal S_{0,\pm 2\varepsilon}^j$ resp. $\mathcal S_{\infty,\pm 2\varepsilon}^j$ we denote the portions of the boundary between $\mathcal C_{0,\pm 2\varepsilon}^j$ and $\mathcal C_0^j$ resp. between $\mathcal C_{\infty,\pm 2\varepsilon}^j$ and $\mathcal C_\infty^j$, $j\in\mathbb Z$. By $s_0^j$ we denote the endpoint of the half-line $\mathcal C_0^j$ and by $s_\infty^j$ the endpoint of the half-line $\mathcal C_\infty^j$, at the boundary of the standard quadratic domain, see Figure~\ref{fig:regions}. Then:
\begin{equation*}\label{eq:inters}
s_{0}^j:=\mathcal S_{0,\pm 2\varepsilon}^j \cap \mathcal C_{0}^j,\ s_{\infty}^j:=\mathcal S_{\infty,\pm 2\varepsilon}^j \cap \mathcal C_{\infty}^j,\ j\in\mathbb Z.
\end{equation*}

\begin{lem}\label{lema:use}
Let $\varepsilon>0$ $($arbitrarily small$)$ and let the iterates $\tilde R_{j,\pm}^n$ on petals\footnote{The \emph{shape} of the petals may be changed in the course of this proof, and the original standard quadratic domain may be changed to a \emph{smaller} one, but the petals remain petals of opening $2\pi$ (that is, of width $2\pi$ in the $\zeta$-variable), centered at directions $j\pi$, $j\in\mathbb Z$, of a standard quadratic domain.} $\tilde V_{j}^{\pm}$ of a standard quadratic domain be defined as in Lemma~\ref{lema:konver}. Let $s_{0}^j$ and $s_{\infty}^j$,$\ j\in\mathbb Z$, be the endpoints of the half-lines $\mathcal C_0^j$ and $\mathcal C_\infty^j$. Then the following bounds hold:
\begin{enumerate}
\item There exists $K>0$ such that:
\begin{itemize}
\item For $\zeta\in \tilde V_j^+$ such that $d(\zeta,\mathcal C_{0}^{j+1})<\varepsilon$ or $d(\zeta,\mathcal C_\infty^j)<\varepsilon$ $($region $(3))$,
\begin{align*}
|\tilde R_{j,+}^n(\zeta)|\leq
\begin{cases} K\log\frac{|\zeta-s_{0}^{j+1}|}{|\zeta|},&d(\zeta,\mathcal C_{0}^{j+1})<\varepsilon,\\[0.2cm]
K\log\frac{|\zeta-s_{\infty}^{j}|}{|\zeta|},&d(\zeta,\mathcal C_{\infty}^{j})<\varepsilon,\quad  j\in\mathbb Z,\ n\in\mathbb N_0.
\end{cases} 
\end{align*}

For $\zeta\in\tilde V_j^-$ such that $d(\zeta,\mathcal C_{\infty}^{j})<\varepsilon$ or $d(\zeta,\mathcal C_0^j)<\varepsilon$ $($region $(3))$,
\begin{align*}
|\tilde R_{j,-}^n(\zeta)|\leq
\begin{cases} K\log\frac{|\zeta-s_{\infty}^{j}|}{|\zeta|},&d(\zeta,\mathcal C_{\infty}^{j})<\varepsilon,\\[0.2cm]
K \log\frac{|\zeta-s_{0}^{j}|}{|\zeta|},&d(\zeta,\mathcal C_{0}^{j})<\varepsilon,\quad j\in\mathbb Z,\ n\in\mathbb N_0.
\end{cases} 
\end{align*}

\item For $\zeta\in \tilde V_j^+$ such that $d(\zeta,\mathcal C_{0}^{j+1})\geq \varepsilon$ and $d(\zeta,\mathcal C_{\infty}^j)\geq \varepsilon$, and for $\zeta\in \tilde V_j^-$ such that $d(\zeta,\mathcal C_{\infty}^{j})\geq \varepsilon$ and $d(\zeta,\mathcal C_{0}^j)\geq \varepsilon$ $($regions $(1)$ and $(2))$,
$$
|\tilde R_{j,\pm}^n(\zeta)|\leq K,\ j\in\mathbb Z,\ n\in\mathbb N_0.
$$
\end{itemize}

\item There exists $D>0$ such that:
\begin{align*}
&|e^{-2\pi i\big(\frac{\tilde \Psi_\mathrm{nf}(\zeta)}{2}+\tilde R_{j-1,+}^n(\zeta)\big)}|\leq D,\ \zeta\in  \tilde V_{0}^j,\\
&|e^{2\pi i\big(\frac{\tilde \Psi_\mathrm{nf}(\zeta)}{2}+\tilde R_{j,+}^n(\zeta)\big)}|\leq D,\ \zeta\in \tilde V_{\infty}^j,\quad j\in\mathbb Z,\ n\in\mathbb N_0.
\end{align*}
\end{enumerate}
The constants $D,\ K$ are independent of the choice of the petal $\tilde V_j^{\pm}$, $j\in\mathbb Z$, and on the iterate $n\in\mathbb N_0$. Moreover, by choosing a standard quadratic domain $\mathcal R_C$ of a sufficiently small radius\footnote{that is, with sufficiently big real parts $\text{Re}(\zeta)>D_0$, for all $\zeta\in\widetilde{\mathcal R}_C$.} as the domain of definition, the bounding constants $D$ and $K$ can be made arbitrarily small. 
\end{lem}
\smallskip

\noindent In Lemma~\ref{lema:use} (1), note that $|\zeta|>D_0>0$ on a standard quadratic domain, so $\frac{|\zeta-s_{0}^{j+1}|}{|\zeta|}$ and $\frac{|\zeta-s_{\infty}^{j}|}{|\zeta|}$ are bounded as $\text{Re}(\zeta)\to +\infty$ on $\tilde V_j^+$. Therefore, $\zeta=s_{0}^{j+1}$ and $\zeta=s_\infty^j$ are the only singularities on $\tilde V_j^+$.

\medskip

\begin{obs}\label{rem:use} From (2) in Lemma~\ref{lema:use}, it immediately follows that (on a standard quadratic or a standard linear domain):
\begin{align}\begin{split}\label{eq:oh}
&|e^{-2\pi i(\tilde \Psi_\mathrm{nf}(\zeta)+\tilde R_{j-1,+}^n(\zeta))}|\leq D|e^{-2\pi i(\tilde \Psi_\mathrm{nf}(\zeta)/2)}|,\ \zeta\in \tilde V_{0}^j, \\
&|e^{2\pi i(\tilde \Psi_\mathrm{nf}(\zeta)+\tilde R_{j,+}^n(\zeta))}|\leq D|e^{2\pi i(\tilde \Psi_\mathrm{nf}(\zeta)/2)}|,\ \zeta\in \tilde V_{\infty}^j,\quad j\in\mathbb Z,\ n\in\mathbb N_0.\end{split}
\end{align}
Given the sequence of pairs of analytic germs $(g_0^j,g_\infty^j; \sigma_j)_{j\in\mathbb Z}$ as in Lemma~\ref{lema:konver}, with radii of convergence $\sigma_j$ bounded from below as in \eqref{eq:fall} resp. \eqref{eq:velradlin}, there exists a standard quadratic resp. linear domain such that $|e^{- 2\pi i(\tilde \Psi_\mathrm{nf}(\zeta)/2)}|<\sigma_j/D,\ \zeta\in\tilde V_0^j$, and $|e^{2\pi i(\tilde \Psi_\mathrm{nf}(\zeta)/2)}|<\sigma_j/D,\ \zeta\in\tilde V_\infty^j$, $j\in\mathbb Z$. Now, we conclude by \eqref{eq:oh} that $e^{-2\pi i(\tilde \Psi_\mathrm{nf}(\zeta)+\tilde R_{j-1,+}^n(\zeta))}$, $\zeta\in\tilde V_{0}^j$, remains in the domain of the definition of $g_0^j$, and that $e^{2\pi i(\tilde \Psi_\mathrm{nf}(\zeta)+\tilde R_{j,+}^n(\zeta))}$, $\zeta\in\tilde V_{\infty}^j$, remains in the domain of the definition of $g_\infty^j$, $j\in\mathbb Z$, for all $n\in\mathbb N_0$. This is important to be able to define the iterative algorithm in Lemma~\ref{lema:konver} (1).
\end{obs}
\medskip

\noindent \emph{Proof of Lemma~\ref{lema:use}.} We prove (1) and (2) simultaneously by induction.
\smallskip

\emph{Step 1.} The \emph{induction basis} for $n=0$. Note that $\tilde R_{j,\pm}^0\equiv 0$ and that the functions $$\zeta\mapsto e^{-2\pi i(\frac{\tilde \Psi_\mathrm{nf}(\zeta)}{2})},\ \zeta\in \tilde V_0^j,\ \text{ and } \zeta\mapsto e^{2\pi i(\frac{\tilde \Psi_\mathrm{nf}(\zeta)}{2})},\ \zeta\in\tilde V_\infty^j,\ j\in\mathbb Z,$$ are \emph{uniformly exponentially flat of order $1-\delta$}, for every $\delta>0$ (see definition \eqref{eq:flaty} of exponential flatness of some order at the beginning of Section~\ref{subsec:prvi}). That is, 
for substrips $\tilde U_{0,\infty}^j\subset \tilde V_{0,\infty}^j$ bisected by $\mathcal C_{0,\infty}^j$ and of \emph{uniform opening in $j$}, there exist $M,\ C>0$ such that: 
\begin{align}\begin{split}\label{eq:unibd}
&|e^{-2\pi i(\frac{\tilde \Psi_\mathrm{nf}(\zeta)}{2})}|\leq Ce^{-Me^{(1-\delta)\text{Re}(\zeta)}} , \ \zeta\in\tilde{U}_{0}^j,\\
&|e^{2\pi i(\frac{\tilde \Psi_\mathrm{nf}(\zeta)}{2})}|\leq Ce^{-Me^{(1-\delta)\text{Re}(\zeta)}} , \ \zeta\in\tilde{U}_{\infty}^j,\ j\in\mathbb Z.\end{split}
\end{align}
This follows by the exact form of $\tilde \Psi_{\mathrm{nf}}$, in the $z$-chart is given in \eqref{eq:abo}, as in the proof of \eqref{eq:ge}.

\noindent From the above bounds, for every $D>0$, we can find a quadratic domain of sufficiently small radius (sufficiently shifted to the right in the logarithmic chart), such that:
\begin{align}\label{eq:jaoj}
&\big|e^{-2\pi i(\frac{\tilde \Psi_\mathrm{nf}(\zeta))}{2}}\big|\leq D,\ \zeta\in \tilde V_{0}^j,\quad |e^{2\pi i(\frac{\tilde \Psi_\mathrm{nf}(\zeta)}{2})}|\leq D,\ \zeta\in \tilde V_{\infty}^j,\
\end{align}
uniformly in $j\in\mathbb Z$.
Note that in \eqref{eq:jaoj} the petals $\tilde V_{0,\infty}^j$, $j\in\mathbb Z$, may have changed \emph{shape} compared to those in \eqref{eq:unibd}. Due to uniform exponential flatness \eqref{eq:unibd}, to ensure boundedness by the same $D$ in all substrips $\tilde U^j_{0,\infty}$ of openings approaching $\pi$, we may have to diminish their radii, resulting in new open petals $\tilde V_{0,\infty}^j$ of opening $\pi$, as unions of such retailored substrips. Thus $(2)$ is satisfied for $n=0$. Note that $(1)$ holds trivially for $n=0$ and for any $K>0$.

\smallskip

\emph{Step 2.} The \emph{induction step}. Suppose that $(1)$ and $(2)$ hold uniformly in $j\in\mathbb Z$ for the $n$-th iterate $\tilde R_{j,\pm}^n$. We prove $(1)$ and $(2)$ for the following iterate $\tilde R_{j,\pm}^{n+1}$ on petals $\tilde V_j^\pm$, with the same constants $D$ and $K$, independent of the induction step $n\in\mathbb N$ and of the petal $j\in\mathbb Z$. We proceed by regions in petals $\tilde V_{j}^\pm$. We prove here the induction step for $\tilde R_{j,+}^{n+1}$ on $\tilde V_{j}^+$. For the repelling petal and $\tilde R_{j,-}^{n+1}$ the same can be repeated. We will consider, as in \eqref{eq:redi}, only one term of the sum \eqref{eq:FJi} in $\tilde R_{j,+}^{n+1}$. For the other three terms the bounds follow analogously. We bound separately in each of the three \emph{regions} (horizontal strips) introduced in \eqref{eq:redi} and in Remark~\ref{rem:regions}:
\smallskip

1. \underline{Region (3)}: $\zeta\in \tilde V_j^+,\ (4j+1)\frac{\pi}{2}-\varepsilon<\text{Im}(\zeta)< (4j+1)\frac{\pi}{2}+\varepsilon$.
\begin{align}\label{eq:int}
&\Big|\sum_{k=j+1}^{+\infty}{}^{n+1} \tilde F_{0,k}^{-}(\zeta)\Big|\leq\\
&\frac{1}{2\pi}\Big|\int_{\mathcal C_{0,+2\varepsilon}^{j+1}}\!\!\!\frac{g_{0}^{j+1}\big(e^{-2\pi i(\tilde\Psi_\mathrm{nf}(w)+\tilde R_{j,+}^{n}(w))}\big)}{w-\zeta} dw\Big|\!\!+\!\!\frac{1}{2\pi}\Big|\int_{\mathcal S_0^{j+1}}\!\!\!\frac{g_{0}^{j+1}\big(e^{-2\pi i(\tilde\Psi_\mathrm{nf}(w)+\tilde R_{j,+}^{n}(w))}\big)}{w-\zeta} dw\Big|+\nonumber\\
&\qquad\qquad\qquad\qquad\qquad\qquad\qquad\quad +\frac{1}{2\pi}\sum_{k=j+2}^{+\infty}\Big|\int_{\mathcal C_{0}^k}\frac{g_{0}^k\big(e^{-2\pi i(\tilde\Psi_\mathrm{nf}(w)+\tilde R_{k-1
,+}^{n}(w))}\big)}{w-\zeta} dw\Big|.\nonumber
\end{align}
All denominators except for the one in the integral $\int_{\mathcal S_0^{j+1}} *\, dw $ can, by absolute value, be bounded away from $\zeta$ by $\varepsilon>0$, that is, $|w-\zeta|\geq \varepsilon$, since the lines of integration are more than $\varepsilon$-away from $\zeta$. In each of these integrals, we make a change of variables that transforms these integrals to integrals along real half-line, as before in \eqref{eq:bd}. 
Using the uniform bound \eqref{prv} on $g_{0,\infty}^k(t),\ k\in\mathbb Z$, we get that there exists $C>0$, such that:
\begin{align}\label{use}
&\big|g_{0}^k\big(e^{- 2\pi i(\tilde\Psi_\mathrm{nf}(\zeta)+\tilde R_{k-1,+}^{n}(\zeta)}\big)\big|\leq C |e^{- 2\pi i(\tilde\Psi_\mathrm{nf}(\zeta)+\tilde R_{k-1,+}^{n}(\zeta))}|\leq\\
&\qquad \leq  C \big|e^{- 2\pi i(\tilde\Psi_\mathrm{nf}(\zeta)/2+\tilde R_{k-1,+}^{n}(\zeta))}\big|\cdot |e^{- 2\pi i(\tilde\Psi_\mathrm{nf}(\zeta)/2)}|\leq CD|e^{- \pi i \tilde\Psi_\mathrm{nf}(\zeta)}|,\ \zeta\in\tilde{V}_{0}^k,\nonumber \\
&\big|g_{\infty}^k\big(e^{ 2\pi i(\tilde\Psi_\mathrm{nf}(\zeta)+\tilde R_{k,+}^{n}(\zeta)}\big)\big|\leq CD |e^{\pi i \tilde\Psi_\mathrm{nf}(\zeta)}|,\ \zeta\in\tilde{V}_{\infty}^k,\ \ k\in\mathbb Z.\nonumber
\end{align}
In the last equality we use the induction hypothesis. As in \eqref{eq:jaoj}, we conclude that $\big|e^{\mp \pi i\frac{\tilde\Psi_\mathrm{nf}(\zeta)}{2}}\big|$ on respective domains $\tilde V_{0,\infty}^j$ can be made smaller than any constant $E>0$ on a standard quadratic domain shifted sufficiently to the right, such that $\text{Re}(\zeta)>D_0$ for some $D_0>0$ (and, as before, with \emph{shapes} of $\tilde V_{0,\infty}^{j}$ possibly changed). 

Now a similar reasoning as in the proof of convergence of \eqref{eq:bd} leads us to conclusion that we can choose a standard quadratic domain of sufficiently small radius such that the sum of all integrals in \eqref{eq:int}, except for the integral $\int_{\mathcal S_{0}^{j+1}}*\,d\zeta$, is by absolute value smaller than any fixed number, so we take $K>0$. We note that the bounds made here do not depend on a specific petal $j\in\mathbb Z$, nor on the step of iteration $n\in\mathbb N$. 

To conclude the induction step (1), it is left to bound the following integral:
\begin{align}\label{eq:integ}
\Big|\int_{\mathcal S_0^{j+1}}\frac{g_{0}^{j+1}\big(e^{-2\pi i(\tilde\Psi_\mathrm{nf}(w)+\tilde R_{j,+}^{n}(w))}\big)}{w-\zeta} dw\Big|, \  \zeta\in {\tilde V_{j}^+},\ d(\zeta, \mathcal C_{0}^{j+1})\leq \varepsilon.
\end{align}
The problem in this region is the following: (1) at the point $s_{0}^{j+1}$ at the end of the line $\mathcal S_0^{j+1}$ of integration (that is, at the endpoint of $\mathcal C_0^{j+1}$ on the boundary of the domain), $\tilde R_{j,+}^n(\zeta)$ from the previous step has a logarithmic singularity, thus possibly preventing the mere well-definedness of this integral; and (2) $|w-\zeta|$ is unbounded as $\zeta$ approaches $s_0^{j+1}$, thus generating a new logarithmic singularity at the point $s_{0}^{j+1}$ in the next iterate $\tilde R_{j,+}^{n+1}$. First, the fact that the integral at each step is well-defined is verified by the induction hypothesis (2) or estimate \eqref{use}. We note that a logarithmic singularity at $s_{0}^{j+1}$ is generated in each iterate, but they are\emph{ not accumulating} in iteration, due to the fact that $\tilde R_{j,+}^n$ enters the next step of integration only as the argument of an exponential that is bounded and does not possess a logarithmic singularity any more. To solve problem (2), let $\gamma(t):[0,1]\to\mathcal S_{0}^{j+1}$ be a (smooth) parametrization of $\mathcal S_{0}^{j+1}$, and denote the endpoints by $s_{0}^{j+1}:=\gamma(0)$ and $v_0^{j+1}:=\gamma(1)$. Recall that $$s_{0}^{j+1}=\mathcal S_{0}^{j+1}\cap\mathcal C_{0}^{j+1},\ v_0^{j+1}=\mathcal S_{0}^{j+1}\cap \mathcal C_{0,+2\varepsilon}^{j+1}.$$ We now bound, using the complex mean value theorem for integrals (treating the real and the imaginary part separately, and applying the integral mean value theorem\footnote{For $f,g:[0,1]\to \mathbb C$ bounded, by integral mean value theorem for real functions of a real variable there exist $s_1,s_2,s_3,s_4\in[0,1]$ such that $|\int_0^1 f(t)g(t)dt|\leq |\int_0^1 \text{Re}(f(t))\text{Re}(g(t))dt|+|\int_0^1 \text{Re}(f(t))\text{Im}(g(t))dt|+|\int_0^1 \text{Im}(f(t))\text{Re}(g(t))dt|+|\int_0^1 \text{Im}(f(t))\text{Im}(g(t))dt|=|\text{Re}(f(s_1))|\cdot|\int_0^1 \text{Re}(g(t))dt|+|\text{Re}(f(s_2))|\cdot |\int_0^1 \text{Im}(g(t))dt|+|\text{Im}(f(s_3))|\cdot|\int_0^1 \text{Re}(g(t))dt|+|\text{Im}(f(s_4))|\cdot|\int_0^1 \text{Im}(g(t))dt|\leq 4||f||_{L^{\infty}[0,1]}\big(|\text{Re} (\int_0^1 g(t) dt)|+|\text{Im} (\int_0^1 g(t) dt)|\big)\leq 8||f||_{L^{\infty}[0,1]} \big|\int_0^1 g(t) dt\big|.$}), \eqref{use} and the fact that $|\gamma'(t)|$ is bounded (suppose by $1$) since $\gamma$ is smooth (the boundary of a standard quadratic domain):
\begin{align*}
&\Big|\int_{\mathcal S_0^{j+1}}\!\!\frac{g_{0}^{j+1}\big(e^{-2\pi i(\tilde\Psi_\mathrm{nf}(w)+\tilde R_{j,+}^{n}(w))}\big)}{w-\zeta} dw\Big|\!\!=\!\!\Big|\int_{0}^1\frac{g_{0}^{j+1}\big(e^{-2\pi i(\tilde\Psi_\mathrm{nf}(\gamma(t))+\tilde R_{j,+}^{n}(\gamma(t)))}\big)}{\gamma(t)-\zeta} \gamma'(t) dt\Big|\!\leq\\
&\qquad \quad\qquad \leq 8||g_{0}^{j+1}\big(e^{-2\pi i(\tilde\Psi_\mathrm{nf}(\gamma(t))+\tilde R_{j,+}^{n}(\gamma(t)))}\big)||_{L^\infty[0,1]}\cdot \Big|\int_0^1 \frac{\gamma'(t)}{\gamma(t)-\zeta}dt\Big|\leq\\
&\qquad \quad \qquad \leq 8 CD ||e^{-\pi i(\tilde\Psi_\mathrm{nf}(\gamma(t)))}||_{L^{\infty}[0,1]}\cdot \Big|\int_0^1 \frac{\gamma'(t)}{\gamma(t)-\zeta}dt\Big|.
\end{align*}
The norm of exponentially small $|e^{-\pi i(\tilde\Psi_\mathrm{nf}(\gamma(t))}|$ can, by shifting a standard quadratic domain to the right ($\gamma(t)$ lies in its boundary), be made arbitrarily small (independently of the step $n\in\mathbb N$). Furthermore, there exists a uniform constant $c>0$ (independent of $j\in \mathbb Z$) such that $$\Big|\int_0^1\frac{\gamma'(t)}{\gamma(t)-\zeta}dt\Big|=|\log(v_{0}^{j+1}-\zeta)-\log(s_{0}^{j+1}-\zeta)|\leq c\log \frac{|s_{0}^{j+1}-\zeta|}{|\zeta|},\ \text{$\zeta$ in region (3).}$$ Indeed, note that $v_0^{j+1}$ lies at some bounded distance from region (3), uniformly in $j$, and at $\mathrm{Re}(\zeta)=+\infty$ there is no singularity, so the only singularity is $\zeta=s_0^{j+1}$.  
Consequently, we may bound the whole integral \eqref{eq:integ} in the region (3) above by any positive constant multiplied by $\log \frac{|\zeta-s_{0}^{j+1}|}{|\zeta|}$. Take again $K>0$.

\medskip

2. \underline{Regions $(1)$ and $(2)$}: $\zeta\in \tilde V_j^+,\ \text{Im}(\zeta)\leq (4j+1)\frac{\pi}{2}-\varepsilon$ or $\text{Im}(\zeta)\geq (4j+1)\frac{\pi}{2}+\varepsilon$: The induction step is proven analogously, but easier, since the denominators in \emph{all} integrals are now bounded from below by $\varepsilon$ ($\zeta$ in these regions is at the distance bigger than $\varepsilon$ from all lines of integration), so the logarithm does not appear in bounds. Only one comment is needed. The line $\mathcal C_0^{j+1}$ indeed contains the point $s_{0}^{j+1}$ as its endpoint, but, as discussed before, the integral $\int_{\mathcal C_{0}^{j+1}}\frac{g_{0}^{j+1}\big(e^{-2\pi i(\tilde\Psi_\mathrm{nf}(w)+\tilde R_{j,+}^{n}(w))}\big)}{w-\zeta} dw$ is well-defined since the previous iterate $\tilde R_{j,+}^n(w)$ with logarithmic singularity at $s_{0}^{j+1}$ appears in the integral only as an argument of the exponential, which is bounded. To bound the integrals by any constant (take $K>0$), we use \eqref{use}.
\smallskip

Finally, once that we have proven the induction step for $(1)$ in Lemma~\ref{lema:use}, the induction step for (2) in Lemma~\ref{lema:use} follows easily. We have:
\begin{align*}
&\big|e^{- 2\pi i(\frac{\tilde \Psi_\mathrm{nf}(\zeta)}{2}+\tilde R_{j-1,+}^{n+1}(\zeta))}\big|\leq \big|e^{-2\pi i \frac{\tilde \Psi_\mathrm{nf}(\zeta)}{2}}\big|\cdot e^{2\pi|\tilde R_{j,+}^{n+1}(\zeta)|}\leq\\
&\qquad\qquad \leq \begin{cases}|e^{-\pi i \tilde \Psi_\mathrm{nf}(\zeta)}|D_0^{-2\pi K}|\zeta-s_{0}^{j+1}|^{2\pi K},& \zeta\in \tilde V_{0}^j \text{\ \ in region (3)},\\
|e^{- \pi i \tilde \Psi_\mathrm{nf}(\zeta)}|e^{2\pi K}, & \zeta\in \tilde V_{0}^j \text{\ \ in regions (1) and (2).}\end{cases}
\end{align*}
By shifting a standard quadratic domain sufficiently to the right ($\mathrm{Re}(\zeta)>D_0$), and by changing the shape of $\tilde V_0^j$, $j\in\mathbb Z$, if necessary, both can be made arbitrarily small (uniformly in $j\in\mathbb Z$ and independently of $n\in\mathbb N$), so we make them smaller than $D>0$. The same follows for $V_\infty^j$, $j\in\mathbb Z$. The induction step for $(2)$ in Lemma~\ref{lema:use} is thus proven.
\hfill $\Box$

\bigskip

\noindent \emph{Proof of Lemma~\ref{lem:uses}}. We prove that there exist $0<q<1$ and $c>0$ such that: $$\sup_{\zeta\in \tilde V_+^j}|e^{2\pi i\tilde R_{j,+}^{n+1}(\zeta)}-e^{2\pi i\tilde R_{j,+}^{n}(\zeta)}\big|\leq  c q^n,$$ for every $n\in\mathbb N_0$ and every $j\in\mathbb Z$. The proof is by induction, considering separately the three regions $(1)-(3)$ of $\tilde V_j^+$, as in \eqref{eq:redi}. 
\smallskip

Suppose that there exist $0<q<1$ and $c>0$ (independent of $j\in\mathbb Z$ and $n\in\mathbb N_0$) such that, for some $n\in\mathbb N$, it holds that
$$
\sup_{\zeta\in \tilde V_+^j}|e^{2\pi i\tilde R_{j,+}^{n+1}(\zeta)}-e^{2\pi i \tilde R_{j,+}^{n}(\zeta)}\big|\leq  c q^{n},
$$
for every $j\in\mathbb Z$.
We now prove that it implies, for $\zeta$ in each of the three regions of $\tilde V_j^+$, that:
 $$ \sup_{\zeta\in \tilde V_+^j}|e^{2\pi i \tilde R_{j,+}^{n+2}(\zeta)}-e^{2\pi i\tilde R_{j,+}^{n+1}(\zeta)}\big|\leq  c q^{n+1}.
$$
That is,
$$\sup_{\zeta\in\,\text{region $(i)$}}|e^{2\pi i\tilde R_{j,+}^{n+2}(\zeta)}-e^{2\pi i\tilde R_{j,+}^{n+1}(\zeta)}|\leq  c q^{n+1},\ i\in\{1,2,3\},\ j\in\mathbb Z.$$                                                                                                                                                                                                                                              
We will find now $0<q<1$ and $c>0$, independent of $j\in\mathbb Z$, such that the induction step and the basis of the induction holds. Note that, as before, we work for simplicity with only one term of the sum in \eqref{eq:FJi} for $\tilde R_{j,+}^n$ on $\tilde V_j^+$. For the other three terms of the sum the conclusion follows analogously. The same can be simultaneously done for $\tilde R_{j,-}^n$ on $\tilde V_j^-$, and we omit it. 
\smallskip 

1. \emph{The basis of the induction, $n=0$}. By Lemma~\ref{lema:use} (1), if we shift the standard quadratic domain sufficiently to the right (e.g. $\text{Re}(\zeta)>D_0$) and \emph{reshape} if necessary, then there exists an arbitarily small constant $K>0$ such that\footnote{By Taylor expansion, $|e^z-1|\leq e^{|z|}-1,\ z\in\mathbb C.$}:
\begin{align}\label{eq:twothree}
|e^{2\pi i\tilde R_{j,+}^1(\zeta)}-1|&\leq e^{2\pi |R_{j,+}^1(\zeta)|}-1\leq\nonumber\\
&\leq \begin{cases}
e^{2\pi K}-1,& \zeta\in \tilde V_j^{+},\ d(\zeta,\mathcal C_{0}^{j+1})\geq \varepsilon,\ d(\zeta,\mathcal C_{\infty}^{j})\geq \varepsilon,\\
\big(\frac{|\zeta-s_{0}^{j+1}|}{|\zeta|}\big)^{2\pi K}-1, & \zeta\in \tilde V_{j}^{+},\ d(\zeta,\mathcal C_{0}^{j+1})< \varepsilon,\\
\big(\frac{|\zeta-s_{\infty}^{j}|}{|\zeta|}\big)^{2\pi K}-1, & \zeta\in \tilde V_{j}^{+},\ d(\zeta,\mathcal C_{\infty}^{j})< \varepsilon.
\end{cases}
\end{align}
To conclude that $|e^{2\pi i\tilde R_{j,\pm}^1(\zeta)}-1|$ are bounded from above on the petals $\tilde V_{j}^+$ by some constant $C>0$, independent of $j\in\mathbb Z$, note that the second and the third term in \eqref{eq:twothree} are bounded at $\text{Re}(\zeta)=+\infty$ due to the division by $|\zeta|$ and $\text{Re}(\zeta)>D_0>0$. That is, there exists a constant $C$ and $0<q<1$ such that:
$$
|e^{2\pi i\tilde R_{j,+}^1(\zeta)}-e^{2\pi i\tilde R_{j,+}^0(\zeta)}|\leq Cq^0,\ \zeta\in \tilde V_{j}^{+},\ j\in\mathbb Z.
$$
In fact, we can take here any $0<q<1$, and we will determine the \emph{good} one in the induction process. This is the basis of the induction.

\smallskip

2. \emph{The induction step.} Now suppose that there exist $0<q<1$ and $C>0$ such that, for $n\in\mathbb N_0$, it holds that:
$$
|e^{2\pi i\tilde R_{j,+}^{n+1}(\zeta)}-e^{2\pi i\tilde R_{j,+}^n(\zeta)}|\leq Cq^n,\ \zeta\in \tilde V_{j}^+,\ j\in\mathbb Z.
$$
We prove the induction step $(n+1)$. 
It holds that:
\begin{align}\begin{split}\label{ah}
|e^{2\pi i\tilde R_{j,+}^{n+2}(\zeta)}-e^{2\pi i\tilde R_{j,+}^{n+1}(\zeta)}|&\leq |e^{2\pi i\tilde R_{j,+}^{n+1}(\zeta)}|\cdot \big|e^{2\pi i(\tilde R_{j,+}^{n+2}(\zeta)-\tilde R_{j,+}^{n+1}(\zeta))}-1\big|\\
&\leq e^{2\pi |\tilde R_{j,+}^{n+1}(\zeta)|}\cdot \Big(e^{2\pi|\tilde R_{j,+}^{n+2}(\zeta)-\tilde R_{j,+}^{n+1}(\zeta)|}-1\Big),\ \zeta\in \tilde V_{j}^+.\end{split}
\end{align}
We now estimate $|\tilde R_{j,+}^{n+2}(\zeta)-\tilde R_{j,+}^{n+1}(\zeta)|$ on $\tilde V_{j}^+$, using the induction hypothesis, in regions $(1)-(3)$. 
Note that the expression for the difference $\tilde R_{j,+}^{n+2}(\zeta)-\tilde R_{j,+}^{n+1}(\zeta)$ is similar as in \eqref{eq:redi}, just, instead of $g_{0}^k\big(e^{- 2\pi i(\tilde\Psi_\mathrm{nf}(w)+\tilde R_{k-1,+}^{n}(w))}\big)$, in every integral we have the difference of exponentials: $$g_{0}^k\big(e^{- 2\pi i(\tilde\Psi_\mathrm{nf}(w)+\tilde R_{k-1,+}^{n+1}(w))}\big)-g_{0}^k\big(e^{- 2\pi i(\tilde\Psi_\mathrm{nf}(w)+\tilde R_{k-1,+}^{n}(w))}\big), \ k\in\mathbb Z.$$
As in the proof of Lemma~\ref{lema:use}, we bound the difference $|\tilde R_{j,+}^{n+2}(\zeta)-\tilde R_{j,+}^{n+1}(\zeta)|$ in all regions $(1)-(3)$. By complex mean-value theorem, we first estimate:
\begin{align*}
&\big|g_{0}^k \big(e^{-2\pi i(\tilde\Psi_\mathrm{nf}(w)+\tilde R_{k-1,+}^{n+1}(w))}\big) - g_{0}^k \big( e^{- 2\pi i(\tilde\Psi_\mathrm{nf}(w)+\tilde R_{k-1,+}^{n}(w))} \big) \big|\leq\\
&\leq \sup_{t\in[0,1]} \big|(g_{0}^k)' (e^{- 2\pi i( \tilde\Psi_\mathrm{nf}(w)+(t\tilde R_{k-1,+}^{n}(w)+(1-t)\tilde R_{k-1,+}^{n+1}(w)))}) \big| \cdot |e^{- 2\pi i \tilde \Psi_\mathrm{nf}(w)}|\cdot\\
&\qquad\qquad\qquad  \cdot|e^{-2\pi i (\tilde R_{k-1,+}^n(w)+\tilde R_{k-1,+}^{n+1}(w))}|\cdot \big|e^{2\pi i \tilde R_{k-1,+}^{n+1}(w)}-e^{2\pi i \tilde R_{k-1,+}^n(w)}\big|\leq\\
&\leq d|e^{-2\pi i\tilde \Psi_\mathrm{nf}(w)}|\cdot e^{2\pi\big(|R_{k-1,+}^{n}(w)|+|R_{k-1,+}^{n+1}(w)|\big)}\cdot \big|e^{2\pi i \tilde R_{k-1,+}^{n+1}(w)}-e^{2\pi i \tilde R_{k-1,+}^n(w)}\big|\leq\\
&\leq c|e^{-2\pi i\tilde \Psi_\mathrm{nf}(w)}|\cdot |e^{2\pi i \tilde R_{k-1,+}^{n+1}(w)}-e^{2\pi i \tilde R_{k-1,+}^n(w)}\big|,\ w\in \tilde V_{0}^k,\ k\in\mathbb Z,
\end{align*}
where constants $c,\ d>0$ are uniform with respect to petal $j\in\mathbb Z$ and step $n\in\mathbb N$. 
The last two lines follow by Lemma~\ref{lema:use} (1) and by uniform bounds \eqref{prv} on $(g_0^k)'$. By the induction hypothesis, there exists $c>0$ such that:
\begin{align*}
&\big|g_{0}^k \big(e^{-2\pi i(\tilde\Psi_\mathrm{nf}(w)+\tilde R_{k-1,+}^{n+1}(w))}\big) - g_{0}^k \big( e^{- 2\pi i(\tilde\Psi_\mathrm{nf}(w)+\tilde R_{k-1,+}^{n}(w))} \big) \big|\leq c Cq^n\cdot|e^{-2\pi i\tilde \Psi_\mathrm{nf}(w)}|,\\
&\qquad \qquad \qquad \qquad \qquad \qquad \qquad \qquad \qquad \qquad \quad \qquad \quad w\in \tilde V_{0}^k,\ k\in\mathbb Z,\ n\in\mathbb N.\nonumber
\end{align*}
The same can be repeated for $\tilde V_\infty^k$, $k\in\mathbb Z$. 
\smallskip

Now, estimating as in the proof of Lemma~\ref{lema:use} (1), we get the following bounds by regions:
\begin{align*}
|&\tilde R_{j,+}^{n+2}(\zeta)-\tilde R_{j,+}^{n+1}(\zeta)|\leq \\
&\leq \!\!\begin{cases}
||e^{- \pi i \tilde \Psi_\mathrm{nf}(\zeta)}||_{\tilde V_{0}^{j+1}}\cdot (A \log\frac{|\zeta-s_{0}^{j+1}|}{|\zeta|}+B)\cdot  Cq^n ,& \zeta\in V_{j}^+,\ d(\zeta,\mathcal C_{0}^{j+1})< \varepsilon,\\
||e^{\pi i \tilde \Psi_\mathrm{nf}(\zeta)}||_{\tilde V_{\infty}^{j}}\cdot (A \log\frac{|\zeta-s_{\infty}^{j}|}{|\zeta|}+B)\cdot  Cq^n ,& \zeta\in V_{j}^+,\ d(\zeta,\mathcal C_{\infty}^{j})< \varepsilon,\\
(||e^{-\pi i \tilde \Psi_\mathrm{nf}(\zeta)}||_{\tilde V_{0}^{j+1}}+||e^{\pi i \tilde \Psi_\mathrm{nf}(\zeta)}||_{\tilde V_{\infty}^{j}}) \cdot B \cdot Cq^n ,& \zeta\in V_{j}^+, \ d(\zeta,\mathcal C_{\infty}^{j}\cup \mathcal C_{0}^{j+1})\geq \varepsilon.
\end{cases}
\end{align*}
Here, $A>0$ and $B>0$ are some positive constants, uniform in $n\in\mathbb N$ and in $j\in\mathbb Z$, obtained as sums of integrals with exponentially small numerators and bounded denominators, similarly as in the proof of Lemma~\ref{lema:use}. Note that, by shifting a whole standard quadratic domain to the right and possibly \emph{reshaping}, we can make the first norm \emph{arbitrarily small} (less than any $\delta>0$), uniformly in $n\in\mathbb N$ and in $j\in\mathbb Z$.  Now, for every $\delta>0$ there exists a standard quadratic domain $\widetilde{\mathcal R}_\delta$ such that, for $\zeta\in\widetilde{\mathcal R}_\delta$, it holds that:
\begin{align}\begin{split}\label{eq:aha}
&e^{2\pi |\tilde R_{j,+}^{n+2}(\zeta)-\tilde R_{j,+}^{n+1}(\zeta)|}\leq \\
&\!\!\!\leq\begin{cases}
\big(\frac{|\zeta-s_{0}^{j+1}|}{|\zeta|}\big)^{2\pi A\delta Cq^n} e^{2\pi B\delta Cq^k}\leq N^{2\pi A\delta Cq^n} e^{2\pi B\delta Cq^k},&\!\!\!\zeta\in \tilde V_j^+, \ d(\zeta,\mathcal C_{0}^{j+1})< \varepsilon,\\
\big(\frac{|\zeta-s_{\infty}^{j}|}{|\zeta|}\big)^{2\pi A\delta Cq^n} e^{2\pi B\delta Cq^k}\leq N^{2\pi A\delta Cq^n} e^{2\pi B\delta Cq^k},&\!\!\!\zeta\in \tilde V_j^+, \ d(\zeta,\mathcal C_{\infty}^{j})< \varepsilon,\\
e^{2\pi B \delta Cq^n},&\!\!\!\zeta\in \tilde V_{j}^+,\ d(\zeta,\mathcal C_{\infty}^{j}\cup \mathcal C_{0}^{j+1})\geq \varepsilon.
\end{cases}\end{split}
\end{align}
Here, $N>0$ is some positive constant that bounds $\frac{|\zeta-s_{0}^{j+1}|}{|\zeta|}$ in region (3), uniformly in $j\in\mathbb Z$. 
Taking $\delta>0$ sufficiently small (diminishing the domain), putting \eqref{eq:aha} in \eqref{ah}, we get:
$$
|e^{2\pi i\tilde R_{j,+}^{n+2}(\zeta)}-e^{2\pi i\tilde R_{j,+}^{n+1}(\zeta)}|\leq Cq^{n+1},\ \zeta\in \tilde V_{j}^+,\ j\in\mathbb Z.$$
All bounds are independent on the step $n\in\mathbb N$ and on the petal $j\in\mathbb Z$. The induction step is thus proven.
\hfill $\Box$

\medskip

\subsection{Proof of Lemma~\ref{lem:simi}} 

\begin{proof} $(a)$ Let $(g_0^j,g_\infty^j)_{j\in\mathbb Z}$ be as in \eqref{e:pp} and \eqref{eq:pt} from Lemma~\ref{lema:konver}. Since $$(h_0^j)^{-1}(t)=te^{2\pi i g_0^j(t)},\ h_\infty^j(t)=te^{2\pi i g_\infty^j(t)},\ t\approx 0,$$ the symmetry \eqref{eq:sym} of $(h_0^j,h_\infty^j)_j$ implies:
\begin{align*}
&\overline{te^{2\pi i g_0^{-j+1}(t)}}=\overline{t} e^{2\pi i g_\infty^j(\overline{t})},\\
&\overline{t} e^{-2\pi i \overline{g_0^{-j+1}(t)}}=\overline{t} e^{2\pi i g_\infty^j(\overline{t})}, \  t\in (\mathbb C,0),\ j\in\mathbb Z.
\end{align*}
Therefore,
\begin{equation}\label{eq:hhaha}
\overline{g_0^{-j+1}(t)}=-g_\infty^j(\overline{t}),\  t\in (\mathbb C,0),\ j\in\mathbb Z.
\end{equation}
Here, we use that $\overline{e^{-2\pi i\overline z}}=e^{2\pi i z},\ z\in\mathbb C$.
\smallskip

The remainder of the proof is done by induction on the iterates of the Fatou coordinate $\tilde {\Psi}_{j,\pm}^n(\zeta)$ (in the logarithmic chart). We prove, using symmetry \eqref{eq:hhaha}, that, for every $n\in\mathbb N_0$, the following symmetry of the iterates holds:
\begin{align}\begin{split}\label{eq:prov}
&\overline{\tilde {\Psi}_{-j+1,-}^n(\zeta)}=\tilde {\Psi}_{j,-}^n (\overline{\zeta}),\ \zeta\in \tilde {V}_+^{-j+1}, \\
&\overline{\tilde {\Psi}_{-j,+}^n(\zeta)}=\tilde {\Psi}_{j,+}^n (\overline{\zeta}),\ \zeta\in \tilde {V}_{+}^{-j},\ j\in\mathbb Z.\end{split}
\end{align}
Note that this is an analogue of $(9.6)$ in \cite[Proposition 9.2]{prvi} in the logarithmic chart (i.e. in the $\zeta$-variable). As a consequence, the same symmetry \eqref{eq:prov} holds for the limits $\tilde \Psi_{j}^{\pm}$ on $\tilde V_j^{\pm}$, $j\in\mathbb Z$, as $n\to\infty$, defined in Lemma~\ref{lema:konver} (2). In particular, for $j=0$ and for $y\in\mathbb R_+$ it holds that:
$$
\overline{\tilde {\Psi}_0^+(y)}=\tilde {\Psi}_0^+(\overline{y})=\tilde {\Psi}_0^+(y),\ y\in (0,+\infty)\cap \tilde V_0^+.$$
Finally, returning to the $z$-variable, this gives:
$$
\overline {\Psi_0^+(x)}=\Psi_0^+(x),\ x\in \mathbb R_+\cap V_0^+.
$$
That is, $\Psi_0^+(\mathbb R_+\cap V_0^+)\subseteq \mathbb R_+\cap V_0^+$, which completes the proof.
\smallskip

It is left to prove \eqref{eq:prov} by induction. For $n=0$, \eqref{eq:prov} is trivially satisfied, since $\tilde\Psi_{j,\pm}^0(\zeta)\equiv\tilde \Psi_\mathrm{nf}(\zeta),\ \zeta\in\tilde V_j^{\pm}$. Here, $\tilde \Psi_\mathrm{nf}$ is the Fatou coordinate of the $(2,m,\rho)$-normal form, $\rho\in\mathbb R$. It is analytic \emph{globally} on a standard quadratic domain and satisfies $\tilde\Psi_\mathrm{nf}(\mathbb R_+)\subseteq\mathbb R_+$, due to $\rho\in\mathbb R$. Therefore, the basis of induction follows by Schwarz's reflection principle.

\noindent We now suppose that \eqref{eq:prov} holds for all $0\leq m< n$, $n\in\mathbb N.$ We prove that it implies \eqref{eq:prov} for $n$. Take $\zeta\in \tilde V_j^+$, for some $j\in\mathbb Z$, then $\overline \zeta\in\tilde V_{-j}^{+}$. We can show the same for pairs $\zeta\in\tilde V_j^{-},\ \overline \zeta\in\tilde V_{-j+1}^{-},\ j\in\mathbb Z$. By the Cauchy-Heine construction from Lemma~\ref{lema:konver} (1), see \eqref{eq:redi}, $\tilde R_{j,+}^n(\zeta),\ \zeta\in \tilde V_j^+,$ is  a sum of terms of the form:
$$
T(\zeta):=\frac{1}{2\pi i}\int_{\mathcal C_{0}^k}\frac{g_{0}^k\big(e^{-2\pi i(\tilde\Psi_{+,k-1}^{n-1}(w))}\big)}{w-\zeta} dw,
$$
where each of them has, in the sum, its "pair" by symmetry: 
$$
P(w):=\frac{1}{2\pi i}\int_{\mathcal C_{\infty}^{-k+1}}\frac{g_{\infty}^{-k+1}\big(e^{-2\pi i(\tilde\Psi_{+,-k+1}^{n-1}(w))}\big)}{w-\zeta} dw, \ \ k\in\mathbb Z.
$$
This "pair" is obtained as the Cauchy-Heine integral along the line $$\mathcal C_{\infty}^{-k+1}=\{\zeta\in\widetilde{\mathcal R}_C: \mathrm{Im}(\zeta)=(-4k+3)\frac{\pi}{2}\},$$ which is exactly complex-conjugated to the line: $$\mathcal C_0^k=\{\zeta\in\widetilde{\mathcal R}_C: \mathrm{Im}(\zeta)=(4k-3)\frac{\pi}{2}\},$$ due to the symmetry of standard quadratic domains with respect to $\mathbb R_+$.  See Figure~\ref{fig:indi} for indexing. 

On the other hand, the same pair $T(\overline \zeta)$ and $P(\overline \zeta)$ appears also in the sum for $\tilde R_{-j,+}^n(\overline \zeta),\ \overline \zeta\in \tilde V_{-j}^+$. Therefore, to show that $\overline{\tilde R_{j,+}^n(\zeta)}=\tilde R_{-j,+}^n(\overline \zeta),\ \zeta\in \tilde V_j^+$, we show simply that, on symmetric petals with respect to $\mathbb R_+$, $P$ and $T$ \emph{exchange places} by conjugation. That is, we show that:
\begin{equation}\label{eq:uzas}\overline{T(\zeta)}=P(\overline \zeta),\ \overline{P(\zeta)}=T(\overline \zeta),\ \zeta\in \tilde V_j^+.\end{equation} 
Indeed, by the change of variables $\xi=\mathrm{Re}(w)$ in the integral, we get:
\begin{align*}
T(\zeta)&=\frac{1}{2\pi i}\int_{x_k}^{+\infty}\frac{g_{0}^{k}\big(e^{-2\pi i\tilde\Psi_{+,k-1}^{n-1}(\xi+i(4k-3)\frac{\pi}{2})}\big)}{\xi+i(4k-3)\frac{\pi}{2}-\zeta}d\xi=\\
&=\frac{1}{2\pi i}\int_{x_k}^{+\infty}\frac{-\overline{g_{\infty}^{-k+1}}\big(e^{2\pi i\cdot\overline{\tilde\Psi_{+,k-1}^{n-1}(\xi+i(4k-3)\frac{\pi}{2})}}\big)}{\xi+i(4k-3)\frac{\pi}{2}-\zeta}d\xi=\\
&=\frac{1}{2\pi i}\int_{x_k}^{+\infty}\frac{-\overline{g_{\infty}^{-k+1}}\big(e^{2\pi i\cdot\tilde\Psi_{+,-k+1}^{n-1}(\xi-i(4k-3)\frac{\pi}{2})}\big)}{\xi+i(4k-3)\frac{\pi}{2}-\zeta}d\xi.
\end{align*}
Here, $x_k>0$ is the real part of the initial point of half-lines $\mathcal C_0^k$ or $\mathcal C_\infty^{-k+1}$. It is the same for both lines, due to symmetry of standard quadratic domains with respect to $\mathbb R_+$. The second line is obtained directly using symmetry \eqref{eq:hhaha} of sequence of pairs $(g_0^j,\,g_\infty^j)_{j\in\mathbb Z}$. In the third line, we use the induction assumption \eqref{eq:prov} for the previous step $n-1$. 

\noindent Now, complex conjugation of the integral gives:                                                                       
\begin{align*}
\overline{T(\zeta)}&=-\frac{1}{2\pi i}\int_{x_k}^{+\infty}\frac{-g_{\infty}^{-k+1}\big(e^{2\pi i\cdot\tilde\Psi_{+,-k+1}^{n-1}(\xi-i(4k-3)\frac{\pi}{2})}\big)}{\xi-i(4k-3)\frac{\pi}{2}-\overline \zeta}d\xi\\
&=\frac{1}{2\pi i}\int_{\mathcal C_{\infty}^{-k+1}}\frac{g_{\infty}^{-k+1}\big(e^{-2\pi i(\tilde\Psi_{+,-k+1}^{n-1}(w))}\big)}{w-\overline \zeta} dw=P(\overline \zeta). 
\end{align*}
\smallskip

\noindent The same analysis is repeated for $\overline{P(\zeta)}$, and for all pairs of terms in the sum for $\tilde R_{j,+}^n(\zeta)$ resp. $\tilde R_{-j,+}^n(\overline \zeta)$. Thus \eqref{eq:uzas} is proven and $\overline{\tilde R_{j,+}^n(\zeta)}=\tilde R_{-j,+}^n(\overline \zeta),\ \zeta\in \tilde V_j^+$. Consequently, since $\overline {\tilde\Psi_\mathrm{nf}(\zeta)}=\tilde\Psi_\mathrm{nf}(\overline \zeta)$  on the whole standard quadratic domain (due to \emph{real} invariant $\rho\in\mathbb R$), it follows that $$\overline{\tilde \Psi_{j,+}^n(\zeta)}=\tilde \Psi_{-j,+}^n(\overline \zeta),\ \zeta\in \tilde V_j^+.$$ By induction, this holds for all $n\in\mathbb N$.
\smallskip

$(b)$ Let $R_0^+(z)$ be such that $\Psi_0^+(z)=\Psi_{\mathrm{nf}}(z)+R_0^+(z)$, $z\in V_0^+$, as constructed by iterative procedure in Lemma~\ref{lema:konver}. Let $\breve R_0^+(\boldsymbol\ell):=R_0^+(z)$, $\boldsymbol\ell\in\boldsymbol\ell(V_0^+)$. By $(1)$, since $\Psi_{\mathrm{nf}}$ is $\mathbb R_+$-invariant, \begin{equation}\label{eq:li}\overline {\breve R_0^+(u)}=\breve R_0^+(u),\ u\in\mathbb R_+\cap \boldsymbol\ell(V_0^+).\end{equation} 

Let $\widehat R\in\mathbb C[[\boldsymbol\ell]],\ \widehat R=\sum_{k\in\mathbb N} a_k\boldsymbol\ell^k$, $a_k\in\mathbb C,\ k\in\mathbb N$, be the common asymptotic expansion of $\breve R_j^\pm(\boldsymbol\ell)$, as $\boldsymbol\ell\to 0$ in $\boldsymbol\ell(V_j^{\pm})$ on standard linear domain, which was proven to exist in Section~\ref{subsec:treci} in the proof of Theorem B. Then, for every $N\in\mathbb N$, there exists $C_N\in\mathbb R$ such that:
\begin{equation}\label{eq:qiq}
\big|\breve R_0^+(u)-\sum_{k=1}^{N} a_k u^k\big|\leq C_N|u^{N+1}|,\ u\in\mathbb R_+,\ u\to 0.
\end{equation}
This implies, by \eqref{eq:li}, that:
\begin{align}\begin{split}\label{eq:qqq}
&\left|\overline{\breve R_0^+(u)-\sum_{k=1}^{N} a_k u^k}\right|\leq C_N|u^{N+1}|,\\
&\big|\breve R_0^+(u)-\sum_{k=1}^N \overline{a_k}\cdot u^k\big|\leq C_N|u^{N+1}|,\ N\in\mathbb N.\end{split}
\end{align}
Now, $\overline{a_k}=a_k,\ k\in\mathbb N$, follows by \eqref{eq:qiq} and \eqref{eq:qqq} and by the uniqueness of the asymptotic expansion of $\breve R_0^+(\boldsymbol\ell)$ in $\mathbb C[[\boldsymbol\ell]]$.
\end{proof}

\bigskip

\bigskip

\textbf{Acknowledgement.} The authors would like to warmly thank Jean-Philippe Rolin for numerous discussions on the subject, and Lo\" ic Teyssier and Daniel Panazzolo for ideas and comments that helped in the realization of this paper. We also thank the referee for his or her detailed reading and comments that have greatly helped us improve the exposition of the paper.

\bigskip

\emph{Address:}$\quad$$^{1}$: Institut de Math\' ematiques de Bourgogne, UMR 5584, CNRS, Universit\' e Bourgogne Franche-Comt\' e, F-21000 Dijon, France 

$^{2}$ : University of Zagreb, Faculty of Science, Department of Mathematics, Bijeni\v cka 30, 10000 Zagreb, Croatia

\end{document}